\documentclass[11pt,a4paper]{amsart}
\usepackage[mathcal]{eucal}
\usepackage{amsfonts}
\usepackage{amsbsy}
\usepackage{amssymb}
\usepackage{amsmath}
\usepackage{graphicx}
\usepackage[latin2]{inputenc}

\newtheorem{thm}{Theorem}[section]
\newtheorem{cor}[thm]{Corollary}
\newtheorem{lem}[thm]{Lemma}

\newcounter{hypothA}
\newcounter{VECHIhypothF}
\newcounter{hypothF}
\newcounter{hypothE}
\newcounter{hypothL}
\theoremstyle{definition}

\theoremstyle{remark}
\newtheorem{rem}[thm]{Remark}
\numberwithin{equation}{section}

\begin{document}

\title[On the existence ...] {On the existence of approximate problems
that preserve the type of a bifurcation point of a nonlinear
problem. Application to the
stationary Navier-Stokes equations}%
\author[C\u{a}t\u{a}lin - Liviu Bichir]%
{C\u{a}t\u{a}lin - Liviu Bichir}%

\address{independent researcher, Apartment 5, Flat V1, 74 Nicolae B\u{a}lcescu Street, Gala\c{t}i, 800001, Romania}%
\email{catalinliviubichir@yahoo.com}%






\begin{abstract}
We consider a nonlinear problem $F(\lambda,u)=0$ on
infinite-dimensional Banach spaces that correspond to the
steady-state bifurcation case. In the literature, it is found
again a bifurcation point of the approximate problem
$F_{h}(\lambda_{h},u_{h})=0$ only in some cases. We prove that, in
every situation, given $F_{h}$ that approximates $F$, there exists
an approximate problem $F_{h}(\lambda_{h},u_{h})-\varrho_{h} = 0$
that has a bifurcation point with the same properties as the
bifurcation point of $F(\lambda,u)=0$. First, we formulate, for a
function $\widehat{F}$ defined on general Banach spaces, some
sufficient conditions for the existence of an equation that has a
bifurcation point of certain type. For the proof of this result,
we use some methods from variational analysis, Graves' theorem,
one of its consequences and the contraction mapping principle for
set-valued mappings. These techniques allow us to prove the
existence of a solution with some desired components that equal
zero of an overdetermined extended system. We then obtain the
existence of a constant (or a function) $\widehat{\varrho}$ so
that the equation $\widehat{F}(\lambda,u)-\widehat{\varrho} = 0$
has a bifurcation point of certain type. This equation has
$\widehat{F}(\lambda,u) = 0$ as a perturbation. It is also made
evident a class of maps $C^{p}$ - equivalent (right equivalent) at
the bifurcation point to
$\widehat{F}(\lambda,u)-\widehat{\varrho}$ at the bifurcation
point. Then, for the study of the approximation of
$F(\lambda,u)=0$, we give conditions that relate the exact and the
approximate functions. As an application of the theorem on general
Banach spaces, we formulate conditions in order to obtain the
existence of the approximate equation
$F_{h}(\lambda_{h},u_{h})-\varrho_{h} = 0$. For example, we
consider the finite element approximation of stationary
Navier-Stokes equations.
\end{abstract}
\maketitle

\section{Introduction}
\label{sectiunea_0_introduction}

For a steady-state bifurcation problem on infinite-dimensional
Banach spaces, we study the existence of an approximate problem
that has a bifurcation point with the same properties as the
bifurcation point of the given nonlinear problem. The problem of
bifurcation is present in the analysis of many mathematical models
of phenomena from the physical world. Generally, these models are
formulated with an equation on Banach spaces, infinite or
finite-dimensional. Examples of infinite-dimensional problems are
from fluid mechanics, solid mechanics, elasticity, nonlinear
vibrations, structural analysis, ocean, atmosphere and climate
models and so on. As far as the finite-dimensional problems are
concerned, examples are from medicine (cardiology, neuroscience),
biology, chemistry, economy, etc. In both cases,
infinite-dimensional and finite-dimensional, practical
computations are necessary. For this purpose, the
infinite-dimensional problem must be approximated by a
finite-dimensional problem using methods such as finite element
method, finite differences method, finite volume method, spectral
methods or wavelets. The first problem is named "exact", defined
on exact spaces, by an exact equation and it has exact solutions.
The second equation is the approximate equation and the related
entities are called "approximate".

We consider equations that correspond to the steady-state
bifurcation case. We retain the exact equation (\ref{e5_1}) and
the hypothesis (\ref{ipotezaHypF}) on the bifurcation point from
\cite{CLBichir_bib_Cr_Ra1990}, where Crouzeix and Rappaz study the
bifurcation problems and their approximations. Usually, the
Liapunov - Schmidt method is applied and it is obtained that,
locally, around the bifurcation point, the solution set of the
equation on Banach spaces is in one-to-one correspondence with the
solution set of the classical bifurcation equation. The study of
the solutions of the classical bifurcation equation, using
singularity theory, is performed in
\cite{CLBichir_bib_Golubitsky_Schaeffer1985,
CLBichir_bib_Golubitsky_Stewart_Schaeffer1988}. As the authors of
\cite{CLBichir_bib_Cr_Ra1990} specify, their method, for the exact
equation, is equivalent to the Lyapunov-Schmidt method.

Let $W$ and $Z$ be real Banach spaces. Let $m \geq 1$, $p \geq 2$.
Let $F:\mathbb{R}^{m} \times W \rightarrow Z$ be a nonlinear
function of class $C^{p}$. Consider the equation in $(\lambda,u)
\in \mathbb{R}^{m} \times W$
\begin{equation}
\label{e5_1}
   F(\lambda,u)=0 \, .
\end{equation}

Assume that $(\lambda_{0},u_{0})$ is a solution of (\ref{e5_1})
that satisfies the hypothesis (\cite{CLBichir_bib_Cr_Ra1990}):
\begin{eqnarray}
   & \ & D_{u}F(\lambda_{0},u_{0}) \
            \textrm{is a Fredholm operator of} \
            W \ \textrm{onto} \ Z \ \textrm{with index zero} \, ,
         \label{ipotezaHypF} \\
   & \ & \quad n \geq 1 \ \textrm{and} \ q \geq 1\, ,
        \nonumber
\end{eqnarray}
where $n = dim \ Ker(D_{u}F(\lambda_{0},u_{0}))$ and $q=codim \
Range(DF(\lambda_{0},u_{0}))$. The solution $(\lambda_{0},u_{0})$
is called \textit{a bifurcation point of problem (\ref{e5_1})}. If
$(\lambda_{1},u_{1})$ is a solution of (\ref{e5_1}) and
$D_{u}F(\lambda_{1},u_{1})$ is an isomorphism of $W$ onto $Z$,
then $(\lambda_{1},u_{1})$ is \textit{a regular solution (a
regular point) of problem (\ref{e5_1})}.

Let $\breve{W}$, $\breve{Z}$ be some real Banach spaces and let
$\breve{F}:\mathbb{R}^{m} \times \breve{W} \rightarrow \breve{Z}$.
If a solution $(\breve{\lambda}_{0},\breve{u}_{0})$ of the
equation $\breve{F}(\breve{\lambda},\breve{u})=0$ satisfies
hypothesis (\ref{ipotezaHypF}), with the same $n$ and $q$, we say
that $(\breve{\lambda}_{0},\breve{u}_{0})$ is \textit{a
bifurcation point of the same type as} $(\lambda_{0},u_{0})$.

The above spaces $W$ and $Z$ are both infinite-dimensional or they
are both finite-dimensional. If they are infinite-dimensional,
then equation (\ref{e5_1}) is approximated by an equation
\begin{equation}
\label{e5_9_introd}
   F_{h}(\lambda_{h},u_{h})=0 \, ,
\end{equation}
where $F_{h}:\mathbb{R}^{m} \times W_{h} \rightarrow Z_{h}$,
$W_{h}$ is a closed subspace  of $W$ and $Z_{h}$ is a closed
subspace of $Z$. $W_{h}$ and $Z_{h}$ are both infinite-dimensional
spaces or they are both finite-dimensional spaces with $dim \
W_{h}$ $=$ $dim \ Z_{h}$. Usually, (\ref{e5_9_introd}) is obtained
by finite element method or by the other methods mentioned above
\cite{CLBichir_bib_Bochev_GunzburgerLSFEM2009,
CLBichir_bib_Bohmer2001, CLBichir_bib_Bohmer_Dahlke2003,
CLBichir_bib_Bohmer2010, CLBichir_bib_Brezzi_Rappaz_Raviart1_1980,
CLBichir_bib_Brezzi_Rappaz_Raviart2_1981,
CLBichir_bib_Brezzi_Rappaz_Raviart3_1981,
CLBichir_bib_Cliffe_Spence_Tavener2000, CLBichir_bib_Cr_Ra1990,
CLBichir_bib_Elman_Silvester_Wathen2005, CLBichir_bib_Gir_Rav1986,
CLBichir_bib_Ili1980, CLBichir_bib_Ka_Ak1986,
CLBichir_bib_Quarteroni_Valli2008, CLBichir_bib_Temam1979,
CLBichir_bib_E_Zeidler_NFA_IIA, CLBichir_bib_E_Zeidler_NFA_IIB}.
The theoretical Galerkin method is also taken into account.

A question arises with regard to an approximate equation
(\ref{e5_9_introd}):

\textit{($Q_{1}$) Has the approximate equation (\ref{e5_9_introd})
also a bifurcation point ? In case that it exists, is this point
of the same type as $(\lambda_{0},u_{0})$?}

In case of the simple limit point, the point is generic
(\cite{CLBichir_bib_Golubitsky_Schaeffer1985}) and the approximate
problem has a simple limit point $(\lambda_{0h},u_{0h})$
(\cite{CLBichir_bib_Brezzi_Rappaz_Raviart2_1981,
CLBichir_bib_Cr_Ra1990}).

In case of the simple bifurcation point of the problems on Banach
spaces, it is found again a bifurcation point of the approximate
problem only in some situations. In two particular cases
(\cite{CLBichir_bib_Brezzi_Rappaz_Raviart3_1981,
CLBichir_bib_CalozRappaz1997, CLBichir_bib_Cr_Ra1990}), which are
generic, bifurcation from the trivial branch and symmetry-breaking
bifurcation, a simple bifurcation point $(\lambda_{0h},u_{0h})$ of
the approximate problem exists (numerical bifurcation) and there
is a diffeomorphism between the solution set of the approximate
bifurcation equation and a degenerate hyperbola. In the general
case (\cite{CLBichir_bib_Beyn1980,
CLBichir_bib_Brezzi_Rappaz_Raviart3_1981,
CLBichir_bib_CalozRappaz1997, CLBichir_bib_Cr_Ra1990,
CLBichir_bib_Moore1980}), in the hyperbolic case, the solution set
of the approximate equation is composed of two branches that do
not intersect. This solution set and the solution set of the
approximate bifurcation equation are diffeomorphic to a part of a
nondegenerate hyperbola (imperfect numerical bifurcation). In the
course of their study, Brezzi, Rappaz and Raviart made evident a
perturbed approximate bifurcation equation (equation (3.21), page
11, \cite{CLBichir_bib_Brezzi_Rappaz_Raviart3_1981}) of the
approximate bifurcation equation. The branches of this perturbed
equation intersect transversally in a point. In this general case,
Weber \cite{CLBichir_bib_Weber1981} propose to calculate an
approximation of the exact solution as a component of the solution
of an approximate adequate extended system and then, compute two
approximate branches that intersect in this point. In this way,
bifurcation is not destroyed by approximation.

We mention that the terminology, in \cite{CLBichir_bib_Cr_Ra1990},
is the following: simple bifurcation points are simple bifurcation
points (fold bifurcation and cusp bifurcation in
\cite{CLBichir_bib_Brezzi_Rappaz_Raviart3_1981}, transcritical in
\cite{CLBichir_bib_Cliffe_Spence_Tavener2000}, transcritical and
pitchfork (subcritical and supercritical) in
\cite{CLBichir_bib_Drazin_Reid2004} and
\cite{CLBichir_bib_Kielhofer2012}) or double limit points.

We cite, among other references,
\cite{CLBichir_bib_Cliffe_Spence_Tavener2000,
CLBichir_bib_Foias_Temam1978, CLBichir_bib_AG_Mo_O1999,
CLBichir_bib_Golubitsky_Schaeffer1985, CLBichir_bib_Govaerts2000,
CLBichir_bib_Kielhofer2012, CLBichir_bib_Temam1995,
CLBichir_bib_E_Zeidler_NFA_I, CLBichir_bib_E_Zeidler_NFA_III,
CLBichir_bib_E_Zeidler_NFA_IV} for a discussion about the
genericity of bifurcation points. For the exact stationary Navier
- Stokes equations, there exist bifurcation points that are not
generic \cite{CLBichir_bib_Foias_Temam1978,
CLBichir_bib_Temam1995}. For the bifurcation of the solutions of
the stationary Navier - Stokes equations, we mention
\cite{CLBichir_bib_Bohmer2001, CLBichir_bib_Bohmer2010,
CLBichir_bib_Brezzi_Rappaz_Raviart2_1981,
CLBichir_bib_Cliffe_Spence_Tavener2000,
CLBichir_bib_Elman_Silvester_Wathen2005,
CLBichir_bib_Foias_Temam1978, CLBichir_bib_AG1985,
CLBichir_bib_AG_Mo_O1999, CLBichir_bib_Kielhofer2012,
CLBichir_bib_Temam1979, CLBichir_bib_Temam1995,
CLBichir_bib_Temam1999, CLBichir_bib_E_Zeidler_NFA_I,
CLBichir_bib_E_Zeidler_NFA_IV}. For the approximate case of these
equations, studies are performed in, e.g.,
\cite{CLBichir_bib_Brezzi_Rappaz_Raviart2_1981,
CLBichir_bib_Cliffe_Spence_Tavener2000}.

Referring to the results mentioned above, from
\cite{CLBichir_bib_Cr_Ra1990}, concerning imperfect numerical
bifurcation, and to the theory from
\cite{CLBichir_bib_Golubitsky_Schaeffer1985,
CLBichir_bib_Golubitsky_Stewart_Schaeffer1988}, Georgescu
\cite{CLBichir_bib_AG_CLB1997_2002_NS_hs_bt_es} interpreted the
approximate equation (\ref{e5_9_introd}) as a perturbation of the
exact equation (\ref{e5_1}) on different spaces. She also
suggested us \cite{CLBichir_bib_AG_CLB1997_2002_NS_hs_bt_es} that
there probably exists a perturbation of (\ref{e5_9_introd}) which
has the approximate bifurcation point that we sought at
(\ref{e5_9_introd}) in all the situations given by the hypothesis
(\ref{ipotezaHypF}).

As we have seen, in literature, the above question
\textit{($Q_{1}$)} has a positive answer only in some situations.
Moreover, (\ref{e5_9_introd}) cannot be used to study the
qualitative aspects of (\ref{e5_1}). When we use an approximation
method, we expect not only to find some branches of approximate
solutions but also information about the qualitative aspects of
the exact equation. On the other hand, we expect that
(\ref{e5_9_introd}) is the perturbation of an approximate equation
that has a bifurcation point, not only of (\ref{e5_1}) (this
follows e.g. from (the interpretation of)
\cite{CLBichir_bib_Golubitsky_Schaeffer1985,
CLBichir_bib_Golubitsky_Stewart_Schaeffer1988}, from the
discussion \cite{CLBichir_bib_AG_CLB1997_2002_NS_hs_bt_es}
described above and from the interpretation of the approximate
bifurcation equation (3.21), page 11,
\cite{CLBichir_bib_Brezzi_Rappaz_Raviart3_1981}). In order to
obtain a positive answer in all the situations, let us replace the
question \textit{($Q_{1}$)} by the following question

\textit{($Q_{2}$) Does an approximate problem that preserves the
type of $(\lambda_{0},u_{0})$ exist in all the situations given by
the hypothesis (\ref{ipotezaHypF})? If this approximate problem
exists, is the given problem (\ref{e5_9_introd}) a perturbation of
this one?}

We prove an affirmative answer to \textit{($Q_{2}$)}. To the best
of our knowledge, this approach and the results we prove are new.
We do not discuss if the exact bifurcation point is generic or
not. We prove that if an exact bifurcation point exists,
satisfying the hypothesis (\ref{ipotezaHypF}), then, for an
approximation method, there exists an approximate equation that
has a bifurcation point with the same properties (hypotheses).

To be specific, given a function $F_{h}$ that approximates $F$, we
prove that there exists $\varrho_{h}$ such that the equation
\begin{equation}
\label{e5_1_sol_widetilde_x_0h_Inv_Fc_Th_ec_DATA_introd}
      F_{h}(\lambda_{h},u_{h})-\varrho_{h}
      = 0 \, ,
\end{equation}
has a bifurcation point $(\lambda_{0h},u_{0h})$ of the same type
as the bifurcation point $(\lambda_{0},u_{0})$ of (\ref{e5_1}).
$\varrho_{h}$ is a constant. The usual approximate equation
(\ref{e5_9_introd}) is a perturbation of the new approximate
equation (\ref{e5_1_sol_widetilde_x_0h_Inv_Fc_Th_ec_DATA_introd}).
The result for
(\ref{e5_1_sol_widetilde_x_0h_Inv_Fc_Th_ec_DATA_introd}) is local.
Equation (\ref{e5_1_sol_widetilde_x_0h_Inv_Fc_Th_ec_DATA_introd})
can be used in order to study the qualitative aspects of
(\ref{e5_1}). Moreover, there exists a class of maps $C^{p}$ -
equivalent (right equivalent) at $(\lambda_{0h},u_{0h})$ to
$F_{h}(\lambda_{h},u_{h})-\varrho_{h}$ at $(\lambda_{0h},u_{0h})$
and that satisfies the hypothesis (\ref{ipotezaHypF}) in
$(\lambda_{0h},u_{0h})$. The problem can be formulated as an
inverse problem: given $F_{h}$, there exists $\varrho_{h}$ and it
must be determined  such that
(\ref{e5_1_sol_widetilde_x_0h_Inv_Fc_Th_ec_DATA_introd}) has a
bifurcation point of the same type as the exact equation
(\ref{e5_1}).
(\ref{e5_1_sol_widetilde_x_0h_Inv_Fc_Th_ec_DATA_introd})
approximates (\ref{e5_1}). Not every approximate equation of
(\ref{e5_1}) has a bifurcation point. Equation
(\ref{e5_1_sol_widetilde_x_0h_Inv_Fc_Th_ec_DATA_introd}) is a
particular form of (\ref{e5_9_introd}), obtained by replacing
$F_{h}(\lambda_{h},u_{h})$ with
$F_{h}(\lambda_{h},u_{h})-\varrho_{h}$. If (\ref{e5_1}) has two
bifurcation points satisfying hypothesis (\ref{ipotezaHypF}), it
is possible that the corresponding two $\varrho_{h}$ are not
equal. These results do not contradict the present literature
results.

The equation
(\ref{e5_1_sol_widetilde_x_0h_Inv_Fc_Th_ec_DATA_introd}) and the
conclusion for this are the consequences of the formulation of two
main results that we introduce: (i) Theorem
\ref{teorema_principala_parteaMAIN} about the equivalence between
the properties of a bifurcation point that satisfies
(\ref{ipotezaHypF}) and the existence of the solution of an
overdetermined extended system; (ii) Theorem
\ref{teorema_principala_spatii_infinit_dimensionale_widetilde_s_3_0_exact_inf}
where we formulate some sufficient conditions and we establish, on
general Banach spaces, the existence of an equation that has a
bifurcation point of certain type. The reasoning we use is the
following: if the exact problem (\ref{e5_1}) has a bifurcation
point, we construct an adequate extended system applying the
direct implication of the first theorem. This system is
approximated and the proof of the second theorem furnishes an
extended system that satisfies the hypotheses of the converse
implication of the first theorem. In this way, we obtain
(\ref{e5_1_sol_widetilde_x_0h_Inv_Fc_Th_ec_DATA_introd}) and the
fact that (\ref{e5_1_sol_widetilde_x_0h_Inv_Fc_Th_ec_DATA_introd})
has a bifurcation point of the same type as the bifurcation point
of (\ref{e5_1}). Practically, we use Theorem
\ref{teorema_principala_spatii_infinit_dimensionale_widetilde_s_3_0_exact_inf}.
This second theorem is generally valid and regards not only the
approximate equations, but also the exact equations. Theorem
\ref{teorema_principala_spatii_infinit_dimensionale_widetilde_s_3_0_exact_inf}
is a result in its own right. Theorem
\ref{teorema_principala_spatii_infinit_dimensionale_widetilde_s_3_0_exact_inf}
allows us to obtain the existence of a bifurcation problem for
which a given problem is a perturbation. Theorem
\ref{teorema_principala_spatii_infinit_dimensionale_widetilde_s_3_0_exact_h}
and Theorem
\ref{teorema_principala_spatii_infinit_dimensionale_widetilde_s_3_0_lim}
give the affirmative answer to the question \textit{($Q_{2}$)}. We
also give some conditions that relate the exact and the
approximate functions in Theorem \ref{teorema_lema5_8p}. In
Corollary
\ref{corolarul_doi_2_teorema_principala_spatii_infinit_dimensionale_widetilde_s_3_0_exact_fi3_DIF},
we obtain $\varrho$ in the form of a function of $(\lambda,u)$,
where $\varrho$ is the corresponding form of $\varrho_{h}$, from
(\ref{e5_1_sol_widetilde_x_0h_Inv_Fc_Th_ec_DATA_introd}), in the
infinite-dimensional case.

In our approach, the numerical analysis and the numerical
experiments must be performed using the inverse problem attached
to (\ref{e5_1_sol_widetilde_x_0h_Inv_Fc_Th_ec_DATA_introd}) and
not, as usual, using (\ref{e5_9_introd}).

Our results can be applied to the particular case of the exact
simple bifurcation point of (\ref{e5_1}), in the general case, in
the hyperbolic case, studied in
\cite{CLBichir_bib_Brezzi_Rappaz_Raviart3_1981,
CLBichir_bib_CalozRappaz1997, CLBichir_bib_Cr_Ra1990}, mentioned
above. Let us consider (\ref{e5_1}), (\ref{e5_9_introd}) and
(\ref{e5_1_sol_widetilde_x_0h_Inv_Fc_Th_ec_DATA_introd}) in this
case. We obtain the approximate equation
(\ref{e5_1_sol_widetilde_x_0h_Inv_Fc_Th_ec_DATA_introd}) that has
a (an approximate) simple bifurcation point in this case. This
(\ref{e5_1_sol_widetilde_x_0h_Inv_Fc_Th_ec_DATA_introd})
approximates (\ref{e5_1}). The equation (\ref{e5_9_introd}) used
in \cite{CLBichir_bib_Brezzi_Rappaz_Raviart3_1981,
CLBichir_bib_CalozRappaz1997, CLBichir_bib_Cr_Ra1990} is
impractical in order to regain the qualitative aspects of
(\ref{e5_1}), as we saw above; the equation
(\ref{e5_1_sol_widetilde_x_0h_Inv_Fc_Th_ec_DATA_introd}) maintains
the qualitative aspects of (\ref{e5_1}). The Liapunov - Schmidt
method or the alternate equivalent method of Crouzeix and Rappaz
can be applied to
(\ref{e5_1_sol_widetilde_x_0h_Inv_Fc_Th_ec_DATA_introd}) as to
(\ref{e5_1}). There results a (classical) (approximate)
bifurcation equation and there is a diffeomorphism between the
solution set of this one and a degenerate hyperbola. The solution
set of (\ref{e5_1_sol_widetilde_x_0h_Inv_Fc_Th_ec_DATA_introd}) is
composed of two branches that intersect in the simple bifurcation
point. The approximate bifurcation equation from
\cite{CLBichir_bib_Brezzi_Rappaz_Raviart3_1981,
CLBichir_bib_CalozRappaz1997, CLBichir_bib_Cr_Ra1990} is obtained
using mathematical entities related to (\ref{e5_1}); the
(approximate) bifurcation equation for
(\ref{e5_1_sol_widetilde_x_0h_Inv_Fc_Th_ec_DATA_introd}) can be
constructed using only mathematical entities related to
(\ref{e5_1_sol_widetilde_x_0h_Inv_Fc_Th_ec_DATA_introd}). We have
three other remarks related to the results from the literature: 1.
Recall the perturbed approximate bifurcation equation whose
branches intersect transversally in a point (equation (3.21), page
11, \cite{CLBichir_bib_Brezzi_Rappaz_Raviart3_1981}), mentioned
above, made evident by Brezzi, Rappaz and Raviart. This
approximate bifurcation equation is not related to any approximate
equation in \cite{CLBichir_bib_Brezzi_Rappaz_Raviart3_1981}. We do
not perform a study to answer if the approximate equation
(\ref{e5_1_sol_widetilde_x_0h_Inv_Fc_Th_ec_DATA_introd})
corresponds to this perturbed approximate bifurcation equation,
but it seems that this is the case. 2. We can interpret that the
approximation of the exact solution calculated by Weber
\cite{CLBichir_bib_Weber1981}, cited above, is the solution of an
approximate equation of the form
(\ref{e5_1_sol_widetilde_x_0h_Inv_Fc_Th_ec_DATA_introd}). 3. In
each of the generic cases of simple limit point, bifurcation from
the trivial branch and symmetry-breaking bifurcation, in
\cite{CLBichir_bib_Brezzi_Rappaz_Raviart2_1981,
CLBichir_bib_Brezzi_Rappaz_Raviart3_1981,
CLBichir_bib_CalozRappaz1997, CLBichir_bib_Cr_Ra1990}, an estimate
$|\lambda_{0h}$ $-$ $\lambda_{0}|$ is given and it is not proven
that $\lambda_{0h}$ equals $\lambda_{0}$. The bifurcation point
$(\lambda_{0h},u_{0h})$ of
(\ref{e5_1_sol_widetilde_x_0h_Inv_Fc_Th_ec_DATA_introd}) has also
this limitation.

The results can be applied in nonlinear functional analysis
(bifurcation theory, nonlinear Fredholm operators), singularity
theory, analysis on manifolds, modelling, hydrodynamic stability
and bifurcation, solid mechanics, PDEs, other mathematical models
where bifurcation is present, infinite-dimensional and
finite-dimensional dynamical systems, numerical methods.

Let us observe that if some numerical algorithms are implemented
in order to determine the bifurcation point of
(\ref{e5_1_sol_widetilde_x_0h_Inv_Fc_Th_ec_DATA_introd}), then, on
a computer, it is obtained an approximation
$(\lambda_{0\epsilon},u_{0\epsilon})$ of $(\lambda_{0h},u_{0h})$
which is a solution of an equation of the form of
(\ref{e5_1_sol_widetilde_x_0h_Inv_Fc_Th_ec_DATA_introd}),
\begin{equation}
\label{e5_1_sol_widetilde_x_0h_Inv_Fc_Th_ec_DATA_introd_epsilon}
      F_{\epsilon}(\lambda_{\epsilon},u_{\epsilon})-\varrho_{\epsilon}
      = 0 \, ,
\end{equation}
where $F_{\epsilon}$ is an approximation of $F_{h}$ in the
computer's arithmetics.

Under the conditions of the above discussion, the equilibria
(stationary) solutions, at least locally, for an approximate study
of an evolution equation, are given by
(\ref{e5_1_sol_widetilde_x_0h_Inv_Fc_Th_ec_DATA_introd}) and not
by (\ref{e5_9_introd}). In other words, at least locally (related
to equilibria), the approximation of
\begin{equation}
\label{NESTATIONAR_e5_1}
   \frac{\partial u}{\partial t} - F(\lambda,u)=0
\end{equation}
is
\begin{equation}
\label{NESTATIONAR_e5_1_sol_widetilde_x_0h_Inv_Fc_Th_ec_DATA_introd}
      \frac{\partial u_{h}}{\partial t} - F_{h}(\lambda_{h},u_{h})+ \varrho_{h}
      = 0 \, .
\end{equation}

The text is organized as follows.

In our work, we use the method of Crouzeix and Rappaz
\cite{CLBichir_bib_Cr_Ra1990}, so we remind it briefly in Section
\ref{sectiunea_1_preliminaries_Cr_Rap}.

In literature \cite{CLBichir_bib_Bohmer2001,
CLBichir_bib_Bohmer_Dahlke2003, CLBichir_bib_Cliffe_Spence1984,
CLBichir_bib_Cliffe_Spence_Tavener2000, CLBichir_bib_Cr_Ra1990,
CLBichir_bib_Glashoff_Allgower_Peitgen1981,
CLBichir_bib_Govaerts2000, CLBichir_bib_Griewank_Reddien1984,
CLBichir_bib_Griewank_Reddien1986,
CLBichir_bib_Hermann_Middelmann_Kunkel1998,
CLBichir_bib_Hermann_Middelmann1998,
CLBichir_bib_Jepson_Spence1984, CLBichir_bib_Jepson_Spence1989,
CLBichir_bib_Jepson_Spence_Cliffe1991, CLBichir_bib_Keller_1987,
CLBichir_bib_Ku_Mitt_Web1984, CLBichir_bib_Kuznetsov1998,
CLBichir_bib_Menzel1984, CLBichir_bib_Moore1980,
CLBichir_bib_Moore_Spence1980, CLBichir_bib_Sey32_1979,
CLBichir_bib_Sey33_1979, CLBichir_bib_Sey1996,
CLBichir_bib_Weber1981}, an extended system is used in order to
reduce a problem that presents a bifurcation to a problem without
a bifurcation. We develop the work on the basis of a connection
between the properties of a bifurcation point of a nonlinear
equation on Banach spaces and a somewhat new extended system (this
one is constructed based on a local $C^{p}$ - diffeomorphism
related to the bifurcation point). Sections
\ref{sectiunea_1_main_theorem_on_extended_systems} and
\ref{sectiunea_I_main_th_partea1_PROOF} are devoted to this
subject.

In Section
\ref{sectiunea_01_O_formulare_pe_spatii_infinit_dimensionale}, we
formulate and we prove the main result about the existence of an
equation of form
(\ref{e5_1_sol_widetilde_x_0h_Inv_Fc_Th_ec_DATA_introd}), on
infinite-dimensional Banach spaces, which has a solution
$(\lambda_{0},u_{0})$ that satisfies the hypothesis
(\ref{ipotezaHypF}). These developments are based on the methods
presented in the monograph
\cite{CLBichir_bib_Dontchev_Rockafellar2009} of Dontchev and
Rockafellar. Between them, there are the Graves' theorem, one of
its consequences and the contraction mapping principle for
set-valued mappings. These are reminded in Section
\ref{sectiunea_1_preliminaries_Dont_Rock}. The results that we
obtain in Section
\ref{sectiunea_01_O_formulare_pe_spatii_infinit_dimensionale} are
placed in the formalism of
\cite{CLBichir_bib_Brezzi_Rappaz_Raviart1_1980,
CLBichir_bib_Brezzi_Rappaz_Raviart2_1981,
CLBichir_bib_Brezzi_Rappaz_Raviart3_1981,
CLBichir_bib_CalozRappaz1997, CLBichir_bib_Cr_Ra1990,
CLBichir_bib_Dontchev_Rockafellar2009, CLBichir_bib_Gir_Rav1986}
and of Graves' theorem
\cite{CLBichir_bib_Dontchev_Rockafellar2009}.

In Section \ref{sectiunea05_class}, the existence of a class of
maps equivalent to $F_{h}(\lambda_{h},u_{h})-\varrho_{h}$ is made
evident.

In Section \ref{sectiunea06}, the case of the approximate equation
is studied. If the exact problem has a bifurcation point, a
theorem that connects the exact and the approximate problems is
formulated. Then, the main result from Section
\ref{sectiunea_01_O_formulare_pe_spatii_infinit_dimensionale} is
formulated for the approximate case.

In Section \ref{sectiunea04_cazul_ecNS}, we relate the exact and
the finite element formulations from
\cite{CLBichir_bib_Gir_Rav1986}, for the Dirichlet problem for the
stationary Navier-Stokes equations, to the framework of Section
\ref{sectiunea06}.

In Section
\ref{sectiunea_01_O_formulare_pe_spatii_infinit_dimensionale_COMPLEMENTE_NOU},
a complement to Theorem
\ref{teorema_principala_spatii_infinit_dimensionale_widetilde_s_3_0_exact_inf}
is formulated. Instead of a constant $\varrho$ in the equation
(\ref{e5_1_sol_widetilde_x_0h_Inv_Fc_Th_ec_DATA_introd_exact}), we
obtain $\varrho$ in the form of a function of $(\lambda,u)$.

In Section \ref{sectiunea_CONCLUZII_0}, an intended further
research is presented.

\section{Preliminaries}
\label{sectiunea_1_preliminaries}

\subsection{The setting of Crouzeix and Rappaz
\cite{CLBichir_bib_Cr_Ra1990}}
\label{sectiunea_1_preliminaries_Cr_Rap}

Let us retain the equation (\ref{e5_1}) and the hypothesis
(\ref{ipotezaHypF}) considered above following the work of
Crouzeix and Rappaz \cite{CLBichir_bib_Cr_Ra1990}. First, Crouzeix
and Rappaz treat the case $q = 0$ where they reduce a problem that
has bifurcation to a problem without bifurcation (in Chapter 4,
\cite{CLBichir_bib_Cr_Ra1990}). Second, for the case $q \geq 1$,
they reduce the study to the case $q = 0$ (in Chapter 6,
\cite{CLBichir_bib_Cr_Ra1990}). In this subsection, for the sake
of brevity, we remind these in the inverse order, by modifying the
presentation of Crouzeix and Rappaz \cite{CLBichir_bib_Cr_Ra1990}.
The problem without bifurcation is obtained directly for the
reduced problem.

Under the hypothesis that $(\lambda_{0},u_{0})$ is a bifurcation
solution satisfying hypothesis (\ref{ipotezaHypF}) (recall that $q
\geq 1$), $Z_{2}$ $=$ $Range(DF(\lambda_{0},u_{0}))$ is closed in
Z, $DF(\lambda_{0},u_{0})$ is a Fredholm operator of
$\mathbb{R}^{m}\times W$ onto $Z$ with index $m$ and $Z$ $=$
$Z_{1}$ $\oplus$ $Z_{2}$, where $Z_{1}=sp \
\{\bar{a}_{1},\ldots,\bar{a}_{q} \}$ and
$\bar{a}_{1},\ldots,\bar{a}_{q}$ are some linearly independent
some linearly independent elements from $Z$. Let $X$ $=$
$\mathbb{R}^{q+m}$ $\times$ $W$, $Y$ $=$ $\mathbb{R}^{q+m}$
$\times$ $Z$, $f$ $=$ $(f^{1},\ldots,f^{q})$ $\in$
$\mathbb{R}^{q}$, $f_{0}$ $=$ $0$ $\in$ $\mathbb{R}^{q}$, $x$ $=$
$(f,\lambda,u)$, $x_{0}$ $=$ $(f_{0},\lambda_{0},u_{0})$ $\in$
$X$. Crouzeix and Rappaz define the function
\begin{equation}
\label{e5_2}
   G:X \rightarrow Z, \ G(x)=F(\lambda,u)-\sum_{i=1}^{q}f^{i}\bar{a}_{i} \, ,
\end{equation}
and they replace the study of (\ref{e5_1}) around
$(\lambda_{0},u_{0})$ by the study of the case
$Range(DG(x_{0}))=Z$ (corresponding to the case $q = 0$ in Chapter
4, \cite{CLBichir_bib_Cr_Ra1990}) for the problem
\begin{equation}
\label{e5_3}
   G(x)=0 \, ,
\end{equation}
around the solution $x_{0}$ of it. For this, they use the remark
that $(\lambda,u)$ is a solution of (\ref{e5_1}) is equivalent to
the fact that $x$ is a solution of (\ref{e5_3}) satisfying $f=0$.

$G(x_{0})=0$ and $G(x)=F(\lambda,u)$ if $f=0$. $DG(x)y$ $=$
$DG(f,\lambda,u)(g,\mu,w)$ $=$
$DF(\lambda,u)(\mu,w)-\sum_{i=1}^{q}g^{i}\bar{a}_{i}$ for $y$ $=$
$(g,\mu,w)$. $dim \ Ker(DG(x_{0}))=q+m$,
\begin{equation}
\label{e5_2_B_int_ker}
   Ker(DG(x_{0}))=
   \{ y = (g,\mu,w) \in X; g=0,
          (\mu,w) \in Ker(DF(\lambda_{0},u_{0})) \} \, .
\end{equation}
$Range(DG(x_{0}))=Z$ and $DG(x_{0}) \in L(X,Z)$ is a Fredholm
operator with index $q+m$. This is equivalent to
$D_{u}G(0,\lambda_{0},u_{0})$ is a Fredholm operator of $W$ onto
$Z$ with index zero and $Range(DG(x_{0}))=Z$ (according to Chapter
4, \cite{CLBichir_bib_Cr_Ra1990}).

In order to reduce the problem (\ref{e5_3}) to a problem without
bifurcation, Crouzeix and Rappaz justify the introduction of an
operator $B$ and the reduction of the problem of the study of the
solutions of (\ref{e5_3}), in a neighborhood of $x_{0}$, to the
study of the solutions of the extended system in $(\theta,x) \in
\mathbb{R}^{q+m} \times X$
\begin{equation}
\label{e5_1_introduction_2}
   \mathcal{F}(\theta,x) = 0 \, ,
\end{equation}
in a neighborhood of its regular solution $(0,x_{0})$, where
\begin{equation}
\label{e5_1_introduction_2_0}
   \mathcal{F}:\mathbb{R}^{q+m} \times X
      \rightarrow Y \, , \
   \mathcal{F}(\theta,x) = \left[\begin{array}{l}
      B(x)-B(x_{0})-\theta \\
      G(x)
      \end{array}\right].
\end{equation}
The continuous linear operator $B$ $\in$ $L(X,\mathbb{R}^{q+m})$
is introduced such that it satisfies $Ker(DG(x_{0})) \cap Ker(B) =
\{ 0 \}$. $B$ is an isomorphism of $Ker(DG(x_{0}))$ onto
$\mathbb{R}^{q+m}$. To choose $B$ is equivalent to choose $q+m$
linear forms $\chi_{i}$, $i=1,\ldots,q+m$, on $X$ that are
linearly independent on $Ker(DG(x_{0}))$. By identifying
$Ker(DG(x_{0}))$ to $\mathbb{R}^{q+m}$, $B(x)$ can be considered
the component of $x \in X$ on $Ker(DG(x_{0}))$ with respect to the
decomposition $X$ $=$ $Ker(DG(x_{0})) \oplus Ker(B)$. The function
$\mathcal{F}$ is of class $C^{p}$.
\begin{equation}
\label{e5_1_introduction_2_1}
   D_{x}\mathcal{F}(0,x_{0})y= \left[\begin{array}{l}
      B(y) \\
      DG(x_{0})y
      \end{array}\right].
\end{equation}
Since $Ker(D_{x}\mathcal{F}(0,x_{0}))=\{ 0 \}$ and
$Range(D_{x}\mathcal{F}(0,x_{0}))=Y$, there results that
$D_{x}\mathcal{F}(0,x_{0})$ is an isomorphism of $X$ onto $Y$.
Implicit functions theorem leads to

\begin{lem}
\label{lema_A2_7_th_fc_impl} (Lemma 4.1 and Lemma 6.1,
\cite{CLBichir_bib_Cr_Ra1990}). Under the above hypotheses, there
exists a neighborhood $\mathcal{V}$ of $0$ in $\mathbb{R}^{q+m}$,
a neighborhood $\mathcal{U}$ of $x_{0}$ in $X$ and a unique
$C^{p}$ - mapping $x : \theta \in \mathcal{V} \rightarrow
x(\theta)=(f(\theta),\lambda(\theta),u(\theta)) \in \mathcal{U}$
satisfying: a) $G(x(\theta)) = 0$, $\theta =
B(f(\theta)-f_{0},\lambda(\theta)-\lambda_{0},u(\theta)-u_{0})$,
$\forall \theta \in \mathcal{V}$; b) $x(0) = x_{0}$; c)
$(\theta,x(\theta))$ is a regular point of
(\ref{e5_1_introduction_2}), $\forall \theta \in \mathcal{V}$; d)
if $y = (g,\mu,v) \in \mathcal{U}$ is such that $G(y) = 0$, then
$g = f(\theta)$, $\mu = \lambda(\theta)$, $v = u(\theta)$, with
$\theta = B(g-f_{0},\mu-\lambda_{0},v-u_{0})$.
\end{lem}

Equation $f(\theta)=0$ is the bifurcation equation of problem
(\ref{e5_1}) (\cite{CLBichir_bib_Cr_Ra1990}). Crouzeix and Rappaz
specify that their method is equivalent to the Lyapunov-Schmidt
method. They call "classical" the bifurcation equation related to
the Lyapunov-Schmidt method.

For the case $q \geq 1$, Crouzeix and Rappaz use $B$ reduced to
$\mathbb{R}^{m} \times W$, $B$ $\in$ $L(\mathbb{R}^{m} \times
W,\mathbb{R}^{q+m})$. In the above presentation and in Lemma
\ref{lema_A2_7_th_fc_impl}, we maintain the operator $B$ $\in$
$L(X,\mathbb{R}^{q+m})$ corresponding to the case in Chapter 4,
\cite{CLBichir_bib_Cr_Ra1990}.

The problem (\ref{e5_1_introduction_2}), in the case $q = 0$, is
also formulated in a different context in
\cite{CLBichir_bib_Dontchev_Rockafellar2009}. Related to the study
of a linear Fredholm operator and of the system
$D_{x}\mathcal{F}(0,x)y-[\theta',0]^{T}=0$, where $x$ and
$\theta'$ $\in$ $\mathbb{R}^{q+m}$ are fixed, the use of an
operator similar to $B$ (from \cite{CLBichir_bib_Cr_Ra1990}) and
of some elements similar to $\bar{a}_{1},\ldots,\bar{a}_{q}$ (from
\cite{CLBichir_bib_Cr_Ra1990}) is found in
\cite{CLBichir_bib_Bochev_GunzburgerLSFEM2009,
CLBichir_bib_Bohmer2001, CLBichir_bib_Bohmer_Dahlke2003,
CLBichir_bib_Dontchev_Rockafellar2009,CLBichir_bib_Griewank_Reddien1984,
CLBichir_bib_Griewank_Reddien1986, CLBichir_bib_Jepson_Spence1989,
CLBichir_bib_Keller_1987}.

Given two normed (linear) spaces $E$ and $F$, $L(E,F)$ is the
space of all continuous linear mappings (operators) $K:E
\rightarrow F$. An isomorphism of $E$ onto $F$ is a linear,
continuous and bijective mapping $K:E \rightarrow F$ whose inverse
$K^{-1}$ is continuous
\cite{CLBichir_bib_Bochev_GunzburgerLSFEM2009,
CLBichir_bib_Cr_Ra1990, CLBichir_bib_Dontchev_Rockafellar2009,
CLBichir_bib_Fabian_Habala_Hajek_Montesinos_Zizler2011,
CLBichir_bib_Gir_Rav1986, CLBichir_bib_V_L_Hansen1999,
CLBichir_bib_Sir_AF1982, CLBichir_bib_Sir_Spaces_AF1982,
CLBichir_bib_Sir_AM1983}.

\begin{lem}
\label{lema_A2_7_th_fc_impl_th_apl_inv} (The hypotheses of a
formulation of the inverse function theorem, Theorem I.2.2,
\cite{CLBichir_bib_CalozRappaz1997} and a partial result from the
proof of this one). For $v$ $\in$ $X$ and the function
$\widehat{G}:X \rightarrow Z$ of class $C^{p}$, $p \geq 1$, we
assume that $D\widehat{G}(v)$ $ \in$ $L(X,Z)$ is an isomorphism
and that $\beta$ satisfies $2 \gamma L_{\widehat{G}}(\beta)$
$\leq$ $1$, with $\gamma$ $=$
$\widetilde{\gamma}(\widehat{G},v,X,Z)$ and
$L_{\widehat{G}}(\beta)$ $=$
$\widetilde{L}(\widehat{G},v,x,\beta,X,Z)$, where we use the
notations (\ref{e_A2_17_TEXT_gamma_bar_gen}) and
(\ref{e_A2_17_TEXT_mu_bar_gen}) below. Then, for any $z$ $\in$
$\textrm{int} \, \mathbb{B}_{\frac{\beta}{2
\gamma}}(\widehat{G}(v))$, the equation $\widehat{G}(x)$ $=$ $z$
has a unique solution $x$ in $\mathbb{B}_{\beta_{1}}(v)$, where
$\beta_{1}$ $=$ $2 \gamma\| \widehat{G}(v) - z \|_{Z}$ $\leq$
$\beta$.
\end{lem}

\subsection{Contraction mapping principle for set-valued mappings
and Graves' theorem \cite{CLBichir_bib_Dontchev_Rockafellar2009}}
\label{sectiunea_1_preliminaries_Dont_Rock}

Let us present the following mathematical entities as they are
introduced in the monograph
\cite{CLBichir_bib_Dontchev_Rockafellar2009} of Dontchev and
Rockafellar.

Let $(X,\rho)$ be a metric space. For the sets $C$ and $D$ in $X$
and $x$ $\in$ $C$,

$d(x,C)$ $=$ $\inf_{x' \in C} \rho(x,x')$, $e(C,D)$ $=$ $\sup_{x
\in C} d(x,D)$, \\
with the convention $e(\emptyset,D)$ $=$ $0$ when $D \neq
\emptyset$ and $e(\emptyset,D)$ $=$ $\infty$ otherwise.

$\rho(x,y) = \| x-y \|$ if $(X,\| \cdot \|)$ is a normed space.

Let $X$, $Y$ be Banach spaces. Let $F:X \rightrightarrows Y$ be a
set-valued mapping, that is, for $x$ $\in$ $X$, $F$ assigns a set
$F(x)$ that contains one or more elements of $Y$ or it is empty.
The graph of $F$ is the set $gph \ F$ $=$ $\{ (x,y) \in X \times Y
| y \in F(x) \}$. The domain of $F$ is the set $dom \ F$ $=$ $\{ x
| F(x) \neq \emptyset \}$. The range of $F$ is the set $rge \ F$
$=$ $\{ y | y \in F(x) \ \textrm{for some} \ x \}$. A set-valued
mapping $F:X \rightrightarrows Y$ has an inverse $F^{-1}:Y
\rightrightarrows X$, $F^{-1}(y)$ $=$ $\{$ $x$ $|$ $y$ $\in$
$F(x)$ $\}$. Let us retain that $x$ $\in$ $F^{-1}(y)$
$\Leftrightarrow$ $y$ $\in$ $F(x)$.

Let $A$ $\in$ $L(X,Y)$ be a surjective mapping. A consequence of
Theorem 5A.1 (Banach open mapping theorem), page 253,
\cite{CLBichir_bib_Dontchev_Rockafellar2009}, is the existence of
a $\kappa > 0$ such that $d(0,A^{-1}(y))$ $\leq$ $\kappa \| y \|$,
for all $y$. The regularity modulus is defined by
\begin{equation}
\label{e_A2_17_TEXT_reg}
   reg \, A = \sup_{\| y \| \leq 1} d(0,A^{-1}(y))
   \, ,
\end{equation}
$reg \, A < \infty$. For a Banach space $\mathcal{X}$ with norm
$\| \cdot \|_{\mathcal{X}}$, let $\mathbb{B}_{a}(\tilde{s})$ $=$
$\{ s \in \mathcal{X} ; \| \tilde{s}-s \|_{\mathcal{X}} \leq a \}$
be the closed ball with center $\tilde{s}$ and radius $a$.

\begin{thm}
\label{teorema_contraction_mapping_principle_Dontchev_Rockafellar2009}
(Contraction mapping principle for set-valued mappings, Theorem
5E.2, page 284, \cite{CLBichir_bib_Dontchev_Rockafellar2009}) Let
$(X,\rho)$ be a complete metric space, and consider a set-valued
mapping $T:X \rightrightarrows X$ and a point $\bar{x}$ $\in$ $X$.
Suppose that there exist scalars $a > 0$ and $\lambda \in (0,1)$
such that the set $gph T \cap (\mathbb{B}_{a}(\bar{x}) \times
\mathbb{B}_{a}(\bar{x}))$ is closed and

I. $d(\bar{x},T(\bar{x}))$ $<$ $a(1-\lambda)$;

II. $e(T(u) \cap \mathbb{B}_{a}(\bar{x}), T(v))$ $\leq$ $\lambda
\rho(u,v)$ for all $u$, $v$ $\in$ $\mathbb{B}_{a}(\bar{x})$.

Then, $T$ has a fixed point in $\mathbb{B}_{a}(\bar{x})$; that is,
there exists $x$ $\in$ $\mathbb{B}_{a}(\bar{x})$ such that $x$
$\in$ $T(x)$.
\end{thm}

\begin{thm}
\label{teorema_principala_spatii_infinit_dimensionale_th_Graves}
(Graves, Theorem 5D.2, page 276,
\cite{CLBichir_bib_Dontchev_Rockafellar2009}) Let $X$ and $Y$ be
Banach spaces. Consider a function $f:X \rightarrow Y$ and a point
$\bar{x}$ $\in$ $int \, dom \, f$ and let $f$ be continuous in
$\mathbb{B}_{\varepsilon}(\bar{x})$ for some $\varepsilon > 0$.
Let $A$ $\in$ $L(X,Y)$ be surjective and let $\kappa \geq reg \,
A$. Suppose there is a nonnegative $\mu$ such that $\kappa \mu <
1$ and
\begin{equation}
\label{e5_57_conditie_th_Graves_dem}
   \| f(x) - f(x') - A(x - x') \| \leq \mu \| x - x' \|
   \ \textrm{whenever} \
   x, x' \in \mathbb{B}_{\varepsilon}(\bar{x})
   \, .
\end{equation}
Then, in terms of $\bar{y} = f(\bar{x})$ and $c = \kappa^{-1} -
\mu$, if $y$ is such that $ \| y - \bar{y} \| \leq c
\varepsilon$, then the equation $y = f(x)$ has a solution $x$
$\in$ $\mathbb{B}_{\varepsilon}(\bar{x})$.
\end{thm}

\begin{cor}
\label{teorema_principala_spatii_infinit_dimensionale_cor_th_Graves}
(A consequence of Graves' theorem, pages 277-278,
\cite{CLBichir_bib_Dontchev_Rockafellar2009}) Let
(\ref{e5_57_conditie_th_Graves_dem}) hold for $x$, $x'$ $\in$
$\mathbb{B}_{\varepsilon}(\bar{x})$ and choose a positive $\tau <
\varepsilon$. Then, there is a neighborhood $U$ of $\bar{x}$ such
that $\mathbb{B}_{\tau}(x)$ $\subset$
$\mathbb{B}_{\varepsilon}(\bar{x})$ for all $x \in U$. Make $U$
smaller if necessary so that $\| f(x) - f(\bar{x}) \|$ $<$ $c
\tau$ for $x \in U$. Pick $x \in U$ and a neighborhood $V$ of
$\bar{y}$ such that $\| y - f(x) \|$ $\leq$ $c \tau$ for $y \in
V$. Then,
\begin{equation}
\label{e5_57_consecinta_th_Graves_dem_carte}
   d(x,f^{-1}(y)) \leq \frac{\kappa}{1-\kappa \mu} \|y-f(x)\|
   \ \textrm{for} \
   (x,y) \in U \times V
   \, .
\end{equation}
This is the metric regularity property of the function $f$ at
$\bar{x}$ for $\bar{y}$.
\end{cor}

If for every $\mu > 0$ there exists $\varepsilon > 0$ such that
(\ref{e5_57_conditie_th_Graves_dem}) holds for every $x$, $x'$
$\in$ $\mathbb{B}_{\varepsilon}(\bar{x})$, then $A$ is, by
definition (\cite{CLBichir_bib_Dontchev1996},
\cite{CLBichir_bib_Dontchev_Rockafellar2009}), the strict
derivative of $f$ at $\bar{x}$, $A$ $=$ $Df(\bar{x})$.

\begin{thm}
\label{teorema_principala_spatii_infinit_dimensionale_th_Graves_revisited}
(Theorem 1.3, \cite{CLBichir_bib_Dontchev1996}) Let $X$ and $Y$ be
Banach spaces, let $\bar{x}$ $\in$ $X$ and let $f:X \rightarrow Y$
be a function which is strictly differentiable at $\bar{x}$ (see
the definition at page 31,
\cite{CLBichir_bib_Dontchev_Rockafellar2009}). Suppose that
$Df(\bar{x})$ is onto. Then, there exist a neighborhood $W$ of
$\bar{x}$ and a constant $c > 0$ such that for every $x \in W$ and
$\tau > 0$ with $\mathbb{B}_{\tau}(x)$ $\subset$ $W$,
\begin{equation}
\label{e5_57_consecinta_th_Graves_revisited}
   \mathbb{B}_{c \tau}(f(x))
   \subset
   f(\mathbb{B}_{\tau}(x))
   \, .
\end{equation}
\end{thm}

Let us consider the notations $\gamma$ and $L(\varepsilon)$ from
\cite{CLBichir_bib_CalozRappaz1997, CLBichir_bib_Cr_Ra1990,
CLBichir_bib_Gir_Rav1986} in the following way:
\begin{equation}
\label{e_A2_17_TEXT_gamma_bar_gen}
   \widetilde{\gamma}(f,\bar{x},X,Y) = \| Df(\bar{x})^{-1} \|_{L(Y,X)} \, ,
\end{equation}
\begin{equation}
\label{e_A2_17_TEXT_mu_bar_gen}
   \widetilde{L}(f,\bar{x},x,\varepsilon,X,Y) = \sup_{x \in \mathbb{B}_{\varepsilon}(\bar{x})}
   \| Df(\bar{x}) - Df(x) \|_{L(X,Y)} \, .
\end{equation}

\section{The main theorem on the extended systems considered in the work}
\label{sectiunea_1_main_theorem_on_extended_systems}

We formulate a theorem for the equivalence between the properties
of the bifurcation point $(\lambda_{0},u_{0})$ of (\ref{e5_1}),
satisfying hypothesis (\ref{ipotezaHypF}), and the existence of
the solution of an (overdetermined) extended system. This Section
and Section \ref{sectiunea_I_main_th_partea1_PROOF} are devoted to
this subject.

Let us consider $\theta=0$ in equation (\ref{e5_1_introduction_2})
and conclude that we can introduce an equation based on the value
$B(x_{0})$. Let us denote
\begin{equation}
\label{e5_21_BIS}
   \theta_{\ast}=B(x_{0}) \, .
\end{equation}
It follows that the equation
\begin{eqnarray}
   & \ & B(x)-\theta_{\ast}=0 \, ,
         \label{e5_1_introduction_3_teta0} \\
   & \ & G(x)=0
         \nonumber
\end{eqnarray}
has an unique solution $x$ $=$ $x_{0}$ $=$
$(f_{0},\lambda_{0},u_{0})$.

There results the existence of a $\theta_{\ast} \in
\mathbb{R}^{q+m}$ such that we can consider equation
(\ref{e5_1_introduction_3_teta0}) on its own and we can formulate
questions about the existence, the uniqueness and the properties
of the solution $(\lambda_{0},u_{0})$ of (\ref{e5_1}). Since our
results are results of existence, it is not necessary to know (to
have explicitely) the value of $\theta_{\ast}$. In the sequel, we
use (\ref{e5_21_BIS}) in two respects: 1) Given $B$ and $x_{0}$,
$\theta_{\ast}$ results from (\ref{e5_21_BIS}). 2) Given $B$ and
$\theta_{\ast}$, $x_{0}$ is a solution of (\ref{e5_21_BIS}).

Take some fixed $\theta_{0} \in \mathbb{R}^{q+m}$. Consider $p
\geq 2$. Let us introduce the function
\begin{equation}
\label{e5_23}
   \Psi:X \rightarrow Y , \
   \Psi(x)=\left[\begin{array}{l}
      B(x)-\theta_{0}\\
      G(x)
      \end{array}\right] \, ,
\end{equation}
and the equation in $x \in X$
\begin{equation}
\label{e5_22}
   \Psi(x)=0 \, .
\end{equation}
We have $\Psi(x_{0})$ $=$ $\mathcal{F}(0,x_{0})$ and
$D\Psi(x_{0})$ $=$ $D_{x}\mathcal{F}(0,x_{0})$, where
$D\Psi(x)\overline{x}$ $=$
$[B(\overline{x}),DG(x)\overline{x}]^{T}$ and $[\cdot]^{T}$
denotes the transpose of the row vector $[\cdot]$.

As $DG(x_{0})$ is associated to the Fredholm operator
$DF(\lambda_{0},u_{0})$, let us introduce a function $H$, of class
$C^{p-1}$, associated to $D_{u}F(\lambda_{0},u_{0})$. The space
$Z_{4}$ $=$ $Range(D_{u}F(\lambda_{0},u_{0}))$ is closed in $Z$,
$Z$ $=$ $Z_{3}$ $\oplus$ $Z_{4}$, where $Z_{3}$ $=$ $sp \
\{\bar{b}_{1},\ldots,\bar{b}_{n} \}$ and
$\bar{b}_{1},\ldots,\bar{b}_{n}$ are some linearly independent
elements from $Z$. Let $\Delta$ $=$ $\mathbb{R}^{n} \times W$,
$\Sigma$ $=$ $\mathbb{R}^{n} \times Z$, $e$ $=$
$(e^{1},\ldots,e^{n})$ $\in$ $\mathbb{R}^{n}$, $\mathrm{z}$ $=$
$(e,v)$ $\in$ $\Delta$. Define the function $H$ by
\begin{equation}
\label{e5_2_H}
   H:\mathbb{R}^{m} \times W \times \Delta \rightarrow Z, \
      H(\lambda,u,\mathrm{z})=D_{u}F(\lambda,u)v
         -\sum_{k=1}^{n}e^{k}\bar{b}_{k} \, .
\end{equation}
\begin{equation}
\label{e5_2_H_B_ker}
   Ker(H(\lambda_{0},u_{0},\cdot))=
   \{ \mathrm{z} = (e,v) \in \Delta; e=0,
          v \in Ker(D_{u}F(\lambda_{0},u_{0})) \} \, .
\end{equation}
It follows that $dim \ Ker(H(\lambda_{0},u_{0},\cdot))=n$,
$Range(H(\lambda_{0},u_{0},\cdot))$ $=Z$ and
$H(\lambda_{0},u_{0},\cdot) \in L(\Delta,Z)$ is a Fredholm
operator with index $n$. We have
$DH(\lambda,u,\mathrm{z})(\overline{\lambda},\overline{u},\overline{\mathrm{z}})$
$=$
$D_{(\lambda,u)}(D_{u}F(\lambda,u)v)(\overline{\lambda},\overline{u})$
$+$ $D_{u}F(\lambda,u)\overline{v}$ $-$
$\sum_{k=1}^{n}\overline{e}^{ \: k}\bar{b}_{k}$.

Consider an operator $\bar{\mathcal{B}}$, similar to $B$,
$\bar{\mathcal{B}} \in L(\Delta,\mathbb{R}^{n})$ such that $
Ker(H(\lambda_{0},u_{0},\cdot)) \cap Ker(\bar{\mathcal{B}}) = \{ 0
\}$. $\bar{\mathcal{B}}$ is an isomorphism of
$Ker(H(\lambda_{0},u_{0},\cdot))$ onto $\mathbb{R}^{n}$ and the
decomposition $\Delta$ $=$ $Ker(H(\lambda_{0},u_{0},\cdot))$
$\oplus$ $Ker(\mathcal{B}_{n})$ holds. To choose
$\bar{\mathcal{B}}$ is equivalent to choose $n$ linear forms
$\bar{\chi}_{k}$, $k=1,\ldots,n$, on $\Delta$ that are linearly
independent on $Ker(H(\lambda_{0},u_{0},\cdot))$.

Define the functions $\Phi_{G}(x,\cdot):X \rightarrow Y$ and
$\Phi_{H}(x,\cdot):\Delta \rightarrow \Sigma$  by
\begin{equation*}
\label{e5_26_T1}
   \Phi_{G}(x,y)=\left[\begin{array}{l}
      B(y) \\
      DG(x)y
      \end{array}\right]
      \, \ \textrm{and} \ \,
   \Phi_{H}(x,\mathrm{z})=\left[\begin{array}{l}
      \bar{\mathcal{B}}(\mathrm{z}) \\
      H(\lambda,u,\mathrm{z})
      \end{array}\right] \, .
\end{equation*}
We have $\Phi_{G}(x,y)$ $=$ $D\Psi(x)y$ and
\begin{equation*}
\label{e5_26_T1_dif}
   D\Phi_{G}(x,y)(\overline{x},\overline{y})=\left[\begin{array}{l}
      B(\overline{y}) \\
      D^{2}F(\lambda,u)
         ((\mu,w),(\overline{\lambda},\overline{u}))
            +DG(x)\overline{y}
      \end{array}\right] \, ,
\end{equation*}
\begin{equation*}
\label{e5_26_T1_dif2}
   D\Phi_{H}(x,\mathrm{z})(\overline{x},\overline{\mathrm{z}})=\left[\begin{array}{l}
      \bar{\mathcal{B}}(\overline{\mathrm{z}}) \\
      D_{(\lambda,u)}(D_{u}F(\lambda,u)v)
         (\overline{\lambda},\overline{u})
            +H(\lambda,u,\overline{\mathrm{z}})
      \end{array}\right] \, .
\end{equation*}

We remember that $X=\mathbb{R}^{q+m} \times W$,
$Y=\mathbb{R}^{q+m} \times Z$, $f=(f^{1},\ldots,f^{q})$ $\in$
$\mathbb{R}^{q}$, $f_{0}=0$ $\in$ $\mathbb{R}^{q}$,
$x=(f,\lambda,u)$, $x_{0}=(f_{0},\lambda_{0},u_{0})$ $\in$ $X$.

Related to $DF(\lambda_{0},u_{0})$, we denote
$g=(g^{1},\ldots,g^{q})$, $g_{i}=(g_{i}^{1},\ldots,g_{i}^{q})$
$\in$ $\mathbb{R}^{q}$, $g_{0}=0$, $g_{i,0}=0$ $\in$
$\mathbb{R}^{q}$ and $y=(g,\mu,w)$, $y_{0}=(g_{0},\mu_{0},w_{0})$,
$y_{i}=(g_{i},\mu_{i},w_{i})$,
$y_{i,0}=(g_{i,0},\mu_{i,0},w_{i,0})$ $\in$ $X$, $i=1,\ldots,q+m$.

Related to $D_{u}F(\lambda_{0},u_{0})$, we denote
$e_{k}=(e_{k}^{1},\ldots,e_{k}^{n})$ $\in$ $\mathbb{R}^{n}$,
$e_{k,0}=0$ $\in$ $\mathbb{R}^{n}$ and $\mathrm{z}_{k} =
(e_{k},v_{k})$, $\mathrm{z}_{k,0} = (e_{k,0}, v_{k,0})$ $\in$
$\Delta$, $k=1,\ldots,n$.

We also denote $\Gamma = X^{1+q+m} \times \Delta^{n}$, $\Sigma =
(\mathbb{R}^{q+m} \times Z)^{1+q+m} \times (\mathbb{R}^{n} \times
Z)^{n}$ and
$s=(x,y_{1},\ldots,y_{q+m},\mathrm{z}_{1},\ldots,\mathrm{z}_{n})$,
$s_{0}=(x_{0},y_{1,0},\ldots,y_{q+m,0},\mathrm{z}_{1,0},\ldots,\mathrm{z}_{n,0})$
$\in \Gamma$.

On the product space $Y_{0}=\prod_{\jmath \: = 1}^{N}Y_{\jmath}$,
we use the norm $\| \cdot \|_{(1)}$, $\| \kappa_{0} \|_{(1)}$ $=$
$\| \kappa_{0} \|_{Y}=\sum_{\jmath \: = 1}^{N}\| \kappa_{\jmath}
\|_{Y_{\jmath}}$, where $Y_{\jmath}$ is a normed space with norm
$\| \cdot \|_{Y_{\jmath}}$, for $\jmath=1,\ldots,N$, and
$\kappa_{0}=(\kappa_{1},\ldots,\kappa_{N}) \in Y$. We write $\|
\cdot \|$ instead of $\| \cdot \|_{E}$ or $\| \cdot \|_{L(E,F)}$
when the spaces $E$ and $F$ are clear. Let $I_{N}$ be the identity
operator on $\mathbb{R}^{N}$, $N \geq 1$. Let $I_{W}$ be the
identity operator on $W$.

$\delta_{ij}$ is the Kronecker delta and $\delta_{k}^{N}$ = $\{
\delta_{k1},\ldots,\delta_{kk},\ldots,\delta_{kN} \}$, for $1 \leq
k \leq N$, but $\{ \delta_{1}^{N}$, $\ldots$, $\delta_{N}^{N} \}$
can be any fixed basis of $\mathbb{R}^{N}$.

For the solution $x_{0}$ of (\ref{e5_1_introduction_3_teta0}), a
basis of $Ker(DF(\lambda_{0},u_{0}))$ and a basis of
$Ker(D_{u}F(\lambda_{0},u_{0}))$ can be determined by the solution
$s_{0}$ of (\ref{e5_56}) (where $\theta_{0}$ $=$ $\theta_{\ast}$)
below.

Together with $\theta_{0}$, let us consider some elements
$\bar{a}_{1},\ldots,\bar{a}_{q}$, $\bar{b}_{1},\ldots,\bar{b}_{n}$
$\in$ $Z$. Having in mind the local $C^{p}$ - diffeomorphism
$\Psi$ at $x_{0}$, let us introduce the function $S:\Gamma
\rightarrow \Sigma$ (\cite{CLBichir_bib_PhDThesis2002}),
\begin{equation}
\label{e5_57}
   S(s)=\left[\begin{array}{l}
      \Psi(x)\\
      B(y_{i})-\delta_{i}^{q+m}\\
      DG(x)y_{i} \\
      \bar{\mathcal{B}}(\mathrm{z}_{k})-\delta_{k}^{n} \\
      H(\lambda,u,\mathrm{z}_{k})
      \end{array}\right]
    \ \textrm{or} \
   S(s)=\left[\begin{array}{l}
      \Psi(x) \\
      \Phi_{G}(x,y_{i})
         - \left[\begin{array}{l}
           \delta_{i}^{q+m} \\
           0
           \end{array}\right] \\
      \Phi_{H}(x,\mathrm{z}_{k})
         - \left[\begin{array}{l}
           \delta_{k}^{n} \\
           0
           \end{array}\right]
      \end{array}\right] \ ,
\end{equation}
for all $i=1,\ldots,q+m$, $k=1,\ldots,n$, with the following
convention (in the left formulation of $S(s)$) that we use
throughout the paper: the second and the third rows (components)
are taken $q+m$ times, for all the values of $i$, and the fourth
and the fifth rows (components) are taken $n$ times, for all the
values of $k$. Consider the extended system in $s$
\begin{equation}
\label{e5_56}
   S(s)=0 \ .
\end{equation}
In (\ref{e5_56}), $x$ is determined by equation $\Psi(x)=0$ and
the rest of components of $s$ are determined by the rest of the
equations.

\begin{rem}
\label{observatia_diferentiala_G} Let us observe that
$DG(f,\lambda,u)(g,\mu,w)$ $=$ $DG(f',\lambda,u)(g,\mu,w)$ for $f$
$\neq$ $f'$.
\end{rem}

We have
\begin{equation}
\label{e5_27_S_3}
       DS(s)\overline{s}=\left[\begin{array}{l}
      B(\overline{x}) \\
      DG(x)\overline{x} \\
      B(\overline{y}_{i}) \\
      D^{2}F(\lambda,u)
         ((\mu_{i},w_{i}),(\overline{\lambda},\overline{u}))
            +DG(x)\overline{y}_{i} \\
      \bar{\mathcal{B}}(\overline{\mathrm{z}}_{k}) \\
      D_{(\lambda,u)}(D_{u}F(\lambda,u)v_{k})
         (\overline{\lambda},\overline{u})
            +H(\lambda,u,\overline{\mathrm{z}}_{k})
      \end{array}\right] \, ,
\end{equation}
$i=1,\ldots,q+m$, $k=1,\ldots,n$.

Let us define the function $\Phi:X \times \Gamma \rightarrow
\Sigma$,
\begin{equation}
\label{e5_57_Phi_3}
   \Phi(x,\phi')=\left[\begin{array}{l}
      \Phi_{G}(x,y') \\
      \Phi_{G}(x,y_{i}') \\
      \Phi_{H}(x,\mathrm{z}_{k}')
      \end{array}\right] \ ,
\end{equation}
for all $i=1,\ldots,q+m$, $k=1,\ldots,n$, where
$\phi'=(y',y_{1}',\ldots,y_{q+m}',\mathrm{z}_{1}',\ldots,\mathrm{z}_{n}')$.

We have
\begin{equation}
\label{e5_57_Phi_3_dif}
   D\Phi(x,\phi')(\overline{x},\overline{\phi}')=\left[\begin{array}{l}
      B(\overline{y}') \\
      D^{2}F(\lambda,u)
         ((\mu',w'),(\overline{\lambda},\overline{u}))
            +DG(x)\overline{y}' \\
      B(\overline{y}_{i}') \\
      D^{2}F(\lambda,u)
         ((\mu_{i}',w_{i}'),(\overline{\lambda},\overline{u}))
            +DG(x)\overline{y}_{i}' \\
      \bar{\mathcal{B}}(\overline{\mathrm{z}}_{k}') \\
      D_{(\lambda,u)}(D_{u}F(\lambda,u)v_{k}')
         (\overline{\lambda},\overline{u})
            +H(\lambda,u,\overline{\mathrm{z}}_{k}')
      \end{array}\right] \ ,
\end{equation}

\begin{rem}
\label{observatia_Lema_solutie_01_echivalenta}

Related to the solution $(\lambda_{0},u_{0})$ of (\ref{e5_1}), to
a basis of $Ker(DF(\lambda_{0},u_{0}))$ and to a basis of
$Ker(D_{u}F(\lambda_{0},u_{0}))$, let us observe that the
following statements i) and ii) must be equivalent:

i) $\theta_{0}$ $=$ $\theta_{\ast}$ from (\ref{e5_21_BIS}),
$g_{i}=0$, $e_{k}=0$, $i=1,\ldots,q+m$, $k=1,\ldots,n$.

ii) There exists $\theta_{0}$ $\in$ $\mathbb{R}^{q+m}$ such that
$f=0$, $g_{i}=0$, $e_{k}=0$, $i=1,\ldots,q+m$, $k=1,\ldots,n$.

\end{rem}

\begin{lem}
\label{Lema_solutie_01_DS3} Let $s$ $\in$ $\Gamma$. The following
statements i) and ii) are equivalent:

i) $D\Psi(x)$ is an isomorphism of $X$ onto $Y$, $DS(s)$ is an
isomorphism of $\Gamma$ onto $\Sigma$

ii) $\Phi_{G}(x,\cdot)$ is an isomorphism of $X$ onto $Y$,
$\Phi_{H}(x,\cdot)$ is an isomorphism of $\Delta$ onto $\Sigma$.
\end{lem}

\begin{proof}
For the implication (i) $\Rightarrow$ (ii), we mentions the
followings remarks: $\Phi_{G}(x,\cdot)$ $=$ $D\Psi(x)$. Let us
take the last component of $DS(s)$ and obtain that
$\Phi_{H}(x,\cdot)$ is bijective from the study of the equation in
$\overline{\mathrm{z}}$
\begin{equation*}
\label{e5_27_S_3_UN_Y_UN_Z_ecuatie}
   \Phi_{H}(x,\overline{\mathrm{z}})
      =[\varsigma',\varsigma'']^{T}
      -{[0,D_{(\lambda,u)}(D_{u}F(\lambda,u)v_{n})
         (\overline{\lambda},\overline{u})]}^{T} \, ,
\end{equation*}
for $(\varsigma', \varsigma'')$ $\in$ $Y$ and
$(\overline{\lambda},\overline{u})$ fixed in $\mathbb{R}^{m}
\times W$.

\qquad
\end{proof}

\begin{lem}
\label{Lema_solutie_01_DS3_fi3} Let the assumptions (i) or (ii) of
Lemma \ref{Lema_solutie_01_DS3} hold. Then, $\Phi(x,\cdot)$ is an
isomorphism of $\Gamma$ onto $\Sigma$.
\end{lem}

We study the case $q \geq 1$. The results can be transferred,
easily, to the case $q =0$. We formulate, now, the main theorem on
the extended systems considered above.

\begin{thm}
\label{teorema_principala_parteaMAIN}

Let $F$, $n \geq 1$, $q \geq 1$, $\bar{a}_{1},\ldots,\bar{a}_{q}$,
$\bar{b}_{1},\ldots,\bar{b}_{n}$ $\in$ $Z \backslash \{ 0 \}$ and
$G$, $H$ defined above. Let $B$ $=$ $(\chi_{1},\ldots,\chi_{q+m})$
$\in$ $L(X,\mathbb{R}^{q+m})$ and $\bar{\mathcal{B}}$ $=$
$(\bar{\chi}_{1},\ldots,\bar{\chi}_{n})$ $\in$
$L(\Delta,\mathbb{R}^{n})$, where $\chi_{1}$, $\ldots$,
$\chi_{q+m}$ are $(q+m)$ nonzero linear forms on $X$ and
$\bar{\chi}_{1}$, $\ldots$, $\bar{\chi}_{n}$ are $n$ nonzero
linear forms on $\Delta$. The following statements (i) and (ii)
are equivalent

(i) Let (\ref{e5_1}). $(\lambda_{0},u_{0})$ is a solution of
(\ref{e5_1}). Hypothesis (\ref{ipotezaHypF}) holds for
$(\lambda_{0},u_{0})$. $Z=Z_{1} \oplus Z_{2}$ with
$\bar{a}_{1},\ldots,\bar{a}_{q}$ linearly independent. $Z=Z_{3}
\oplus Z_{4}$ with $\bar{b}_{1},\ldots,\bar{b}_{n}$ linearly
independent. $x_{0}$ is the unique solution of (\ref{e5_3}). $B$
and $\bar{\mathcal{B}}$ are such that
\begin{equation}
\label{e5_2_B_int}
   Ker(DG(x_{0})) \cap Ker(B) = \{ 0 \} \, ,
\end{equation}
\begin{equation}
\label{e5_2_H_B}
   Ker(H(\lambda_{0},u_{0},\cdot)) \cap Ker(\bar{\mathcal{B}})
      = \{ 0 \} \, .
\end{equation}
where these kernels are given by (\ref{e5_2_B_int_ker}) and
(\ref{e5_2_H_B_ker}). We have
\begin{equation}
\label{e5_21}
   \theta_{0}=B(x_{0}) \in \mathbb{R}^{q+m}.
\end{equation}

(ii) Fix $\theta_{0}$ $\in$ $\mathbb{R}^{q+m}$, a basis $\{
\delta_{i}^{q+m} \}_{i=1,\ldots,q+m}$ of $\mathbb{R}^{q+m}$ and a
basis $\{ \delta_{k}^{n} \}_{k=1,\ldots,n}$ of $\mathbb{R}^{n}$.
Consider $\Psi$ and (\ref{e5_22}) together with $S$ and
(\ref{e5_56}). We have:

(a) The system (\ref{e5_56}) has the solution $s_{0}$ $=$ $
(x_{0},$ $y_{1,0},$ $\ldots,$ $y_{q+m,0},$ $\mathrm{z}_{1,0},$
$\ldots,$ $\mathrm{z}_{n,0})$, where
$x_{0}=(f_{0},\lambda_{0},u_{0})$,
$y_{i,0}=(g_{i,0},\mu_{i,0},w_{i,0})$, $\mathrm{z}_{k,0} =
(e_{k,0}, v_{k,0})$, $f_{0}=0$, $g_{i,0}=0$, $e_{k,0}=0$,
$i=1,\ldots,q+m$, $k=1,\ldots,n$ and the component $x_{0}$ of
$s_{0}$ is the solution of (\ref{e5_22}).

(b) $D\Psi(x_{0})$ is an isomorphism of $X$ onto $Y$; $DS(s_{0})$
is an isomorphism of $\Gamma$ onto $\Sigma$.

\end{thm}

\begin{cor}
\label{corolar_teorema_principala_parteaMAIN_neliniar} Theorem
\ref{teorema_principala_parteaMAIN} remains true if we replace $B$
by a nonlinear $B_{N}$ in the definition (\ref{e5_23}) of $\Psi$.
\end{cor}

\begin{proof} We first use Lemma \ref{teorema_principala_partea1}
(i) and (ii) below. Then, we observe that, under the assumptions
of the statement (ii) (b), we obtain that $\Phi_{G}(x,\cdot)$ is
an isomorphism of $X$ onto $Y$ as in the proof of Lemma
\ref{Lema_solutie_01_DS3}.

\qquad
\end{proof}

\begin{rem}
\label{observatia_teorema_principala_parteaMAIN_1} Since we have
$f_{0}=0$, $g_{i,0}=0$ ($i=1,\ldots,q+m$) and $e_{k,0}=0$
($k=1,\ldots,n$), it follows that the results of Theorem
\ref{teorema_principala_parteaMAIN} do not depend on the choice of
$\bar{a}_{1},\ldots,\bar{a}_{q}$,
$\bar{b}_{1},\ldots,\bar{b}_{n}$.
\end{rem}

\begin{rem}
\label{observatia_teorema_principala_parteaMAIN_2} The results of
Theorem \ref{teorema_principala_parteaMAIN} do not depend on the
choice of $B$ $=$ $(\chi_{1},\ldots,\chi_{q+m})$ and
$\bar{\mathcal{B}}$ $=$ $(\bar{\chi}_{1},\ldots,\bar{\chi}_{n})$.
\end{rem}

\begin{cor}
\label{corolar_teorema_principala_parteaMAIN_ii}

Theorem \ref{teorema_principala_parteaMAIN} ii) a) is equivalent
to the following formulation: $s_{0}$ is a solution of the
overdetermined system
\begin{equation}
\label{e5_57_corolar_teorema_principala_parteaMAIN_ii}
   S(s) = 0 \, ,
   f = 0 \, , \ g_{i} = 0 \, , \ e_{k} = 0 \, , \ i=1,\ldots,q+m \, , \ k=1,\ldots,n \, .
\end{equation}
\end{cor}

\begin{cor}
\label{corolar_teorema_principala_parteaMAIN_ii_q_0} In the case
$q = 0$, we replace $G$ by $F$ in the definition (\ref{e5_57}) of
$S$ and (\ref{e5_57_corolar_teorema_principala_parteaMAIN_ii}) is
reduced to
\begin{eqnarray}
   && B(\lambda,u)-\theta_{0} = 0 \, ,
         \nonumber \\
   && F(\lambda,u) = 0 \, ,
         \nonumber \\
   && B(\mu_{i},w_{i})-\delta_{i}^{m} = 0 \, ,
         \nonumber \\
   && DF(\lambda,u)(\mu_{i},w_{i}) = 0 \, ,
         \label{e5_57_corolar_teorema_principala_parteaMAIN_ii_q_0} \\
   && \bar{\mathcal{B}}(\mathrm{z}_{k})-\delta_{k}^{n} = 0 \, ,
         \nonumber \\
   && H(\lambda,u,\mathrm{z}_{k}) = 0 \, ,
         \nonumber \\
   && e_{k} = 0 \, , \ k=1,\ldots,n \, .
         \nonumber
\end{eqnarray}
\end{cor}

Let us remark that the study of a nonlinear Fredholm operator
$\widehat{F}$ can be reduced to the study of the equation
(\ref{e5_1}) satisfying hypothesis (\ref{ipotezaHypF}). Let
$\widehat{U}$ be an open set, $\widehat{U}$ $\subseteq$ $W$, and
$\widehat{F}:\widehat{U} \rightarrow Z$. Let us fix
$\widehat{u}_{0} \in \widehat{U}$ and let $\widehat{\zeta}_{0} =
\widehat{F}(\widehat{u}_{0})$. Then, the equation
\begin{equation}
\label{e5_1_widehat}
   \widehat{F}(\widehat{u}) - \widehat{\zeta}_{0}=0 \,
\end{equation}
has the solution $\widehat{u}_{0}$ that satisfies the hypothesis:
\begin{eqnarray}
   & \ & D\widehat{F}(\widehat{u}_{0}) \
            \textrm{is a Fredholm operator of} \
            W \ \textrm{onto} \ Z \ \textrm{with index} \ m \, ,
         \label{ipotezaHypF_widehat} \\
   & \ & \quad dim \ Ker(D\widehat{F}(\widehat{u}_{0}))=q+m \, .
        \nonumber
\end{eqnarray}
For instance, we can choose $\widehat{u}_{0}$ such that $dim \
Ker(D\widehat{F}(\widehat{u}_{0}))=\max$.

\section{Proof of the main theorem on the extended systems, Theorem
\ref{teorema_principala_parteaMAIN}}
\label{sectiunea_I_main_th_partea1_PROOF}

\subsection{Some properties of the extended systems
associated to $D_{u}F(\lambda,u)$ and $DF(\lambda,u)$}
\label{subsectiunea_I_partea_4_DFu_sis_DF_sis_x}

Throughout this Section, let us consider the hypotheses of Theorem
\ref{teorema_principala_parteaMAIN}.

\begin{lem}
\label{teorema_principala_partea1} The statements (i) and (ii) are
equivalent, and the same is true for (iii) and (iv).

(i) $x$ is the unique solution of (\ref{e5_3}) and
$\theta_{0}=B(x) \in \mathbb{R}^{q+m}$.

(ii) Let $\theta_{0}$ be a fixed element in $\mathbb{R}^{q+m}$.
Let us consider $\Psi$ and (\ref{e5_22}). $x$ is a solution of
(\ref{e5_22}). $D\Psi(x)$ is an isomorphism of $X$ onto $Y$.

(iii) Assume (i). $DG(x)$ is a Fredholm operator with index $q+m$,
$Range(DG(x))=Z$ and $Ker(DG(x)) \cap Ker(B) = \{ 0 \}$.

(iv) Assume (ii). $\Phi_{G}(x,\cdot)$ is an isomorphism of $X$
onto $Y$.

The lemma remains true if we replace $B$ by a nonlinear $B_{N}$ in
the definition (\ref{e5_23}) of $\Psi$ and we keep $B$ in the
definition of $\Phi_{G}(x,\cdot)$.
\end{lem}

\begin{proof} These results are obtained by replacing $x_{0}$ with $x$ in the
formulation of problem (\ref{e5_1_introduction_2}) as it is
introduced by Crouzeix and Rappaz in
\cite{CLBichir_bib_Cr_Ra1990}.

\qquad
\end{proof}

\begin{rem}
\label{observatie_teorema_principala_partea1} We formulate only
the results for $DF(\lambda,u)$ since those for
$D_{u}F(\lambda,u)$ are similar.
\end{rem}

\begin{cor}
\label{corolarul_teorema_principala_partea2} Under the conditions
of Lemma \ref{teorema_principala_partea1} (iv), we have:
\begin{equation}
\label{e5_2_B_dim_S1_ii_P2}
   dim \ Ker(DG(x))=q+m \, .
\end{equation}
\begin{equation}
\label{e5_2_B_plus_S1_i_P2_B0}
   B \ \textrm{is an isomorphism of} \ Ker(DG(x)) \
      \textrm{onto} \ \mathbb{R}^{q+m} \, ,
\end{equation}
\begin{equation}
\label{e5_2_B_plus_S1_i_P2_phi}
   \chi_{1}, \ldots , \chi_{q+m} \ \textrm{are linearly
      independent on} \ Ker(DG(x)) \, .
\end{equation}
\begin{equation}
\label{e5_2_B_plus_S1_i_P2}
   X = Ker(DG(x)) \oplus Ker(B) \, ,
\end{equation}
\end{cor}

\begin{proof}
Let us prove (\ref{e5_2_B_dim_S1_ii_P2}). Consider the elements
$(\delta_{i}^{q+m},0)$ $\in$ $Y$, $i=1,\ldots,q+m$. Since
$\Phi_{G}(x,\cdot)$ is an isomorphism of $X$ onto $Y$, it follows
that there exist $x_{1},\ldots,x_{q+m}$ $\in$ $X$ such that
$\Phi_{G}(x,x_{i})$ $=$ $(\delta_{i}^{q+m},0)$, so $B(x_{i}) =
\delta_{i}^{q+m}$ (otherwise said, $\chi_{j}(x_{i}) =
\delta_{ij}$) and $DG(x)x_{i} = 0$. We have $x_{i} \neq 0$ since
$\Phi_{G}(x,x_{i})$ $=$ $(0,0)$ otherwise. Suppose that there
exist $\beta_{i}$ $\in$ $\mathbb{R}$, $i=1,\ldots,q+m$, where
$\beta_{i} \neq 0$ for at least one $i$, such that
$\sum_{i=1}^{q+m}\beta_{i}x_{i}=0$. For every $j$,
$j=1,\ldots,q+m$, we obtain $0$ $=$ $\chi_{j}(0)$ $=$
$\chi_{j}(\sum_{i=1}^{q+m}\beta_{i}x_{i})$ $=$
$\sum_{i=1}^{q+m}\beta_{i}\chi_{j}(x_{i})$ $=$
$\beta_{j}\chi_{j}(x_{j})$ $=$ $\beta_{j}$. So our supposition is
false. Hence $\{ x_{1},\ldots,x_{q+m} \}$ is a linearly
independent subset of $Ker(DG(x))$. So $dim \ Ker(DG(x)) \geq
q+m$. Suppose that $dim \ Ker(DG(x)) > q+m$. Then, we have another
element $x'$, $x' \neq 0$, such that $x_{1},\ldots,x_{q+m}$, $x'$
are linearly independent and $DG(x)x' = 0$. Let $\varsigma' =
B(x')$. Suppose that $\varsigma' = 0$. Since $\Phi_{G}(x,\cdot)$
is an isomorphism of $X$ onto $Y$, there results that $x' = 0$, so
it remains $\varsigma' \neq 0$. Since $\varsigma'$ $\in$
$\mathbb{R}^{q+m}$, there exist $\beta_{i} \in \mathbb{R}$,
$i=1,\ldots,q+m$, such that $\varsigma'$ $=$
$\sum_{i=1}^{q+m}\beta_{i}\delta_{i}^{q+m}$. We have
$(\varsigma',0)$ $=$
$\sum_{i=1}^{q+m}\beta_{i}(\delta_{i}^{q+m},0)$, $\beta_{i} \in
\mathbb{R}$. We obtain $\Phi_{G}(x,(x' -
\sum_{i=1}^{q+m}\beta_{i}x_{i}))$ $=$ $(\varsigma',0)$ $-$
$\sum_{i=1}^{q+m}\beta_{i}(\delta_{i}^{q+m},0)$ $=$ $0$.
$\Phi_{G}(x,\cdot)$ is an isomorphism of $X$ onto $Y$ so
$Ker(\Phi_{G}(x,\cdot) = \{ 0 \}$. It follows that $x' -
\sum_{i=1}^{q+m}\beta_{i}x_{i} = 0$ which contradicts the
supposition that $x_{1},\ldots,x_{q+m}$, $x'$ are linearly
independent. It remains (\ref{e5_2_B_dim_S1_ii_P2}).

From (\ref{e5_2_B_dim_S1_ii_P2}), it follows that $B$ is an
isomorphism of $Ker(DG(x))$ onto $\mathbb{R}^{q+m}$.

It follows that $\chi_{1}$, $\ldots$, $\chi_{q+m}$ are linearly
independent on $Ker(DG(x))$.

Let us define $X_{1}$ $=$ $Ker(DG(x))$ $\oplus$ $Ker(B)$. Suppose
that there exists $\overline{x}$ $\in$ $X$ and $\overline{x}$
$\notin$ $X_{1}$. Let $(\varsigma',\varsigma'')$ $=$
$\Phi_{G}(x,\overline{x})$ ($\ast$), where $\varsigma'$ $\in$
$\mathbb{R}^{q+m}$ and $\varsigma''$ $\in$ $Z$. Since
$\Phi_{G}(x,\cdot)$ is an isomorphism of $X$ onto $Y$, there
results that, for $(\varsigma',0)$ $\in$ $Y$, there exists $x'$
$\in$ $X$ such that $\Phi_{G}(x,x')$ $=$ $(\varsigma',0)$ and, for
$(0,\varsigma'')$ $\in$ $Y$, there exists $x''$ $\in$ $X$ such
that $\Phi_{G}(x,x'')$ $=$ $(0,\varsigma'')$. We obtain that $x'$
$\in$ $Ker(DG(x))$ and $x''$ $\in$ $Ker(B)$, so $\overline{\xi}$
$=$ $x'+x''$ $\in$ $X_{1}$, and $\Phi_{G}(x,\overline{\xi})$ $=$
$(\varsigma',\varsigma'')$. So we obtain two solutions
$\overline{x}$, $\overline{\xi}$, $\overline{x}$ $\neq$
$\overline{\xi}$, for ($\ast$). This contradicts the uniqueness of
$\overline{x}$ in ($\ast$). It remains
(\ref{e5_2_B_plus_S1_i_P2}).

\qquad
\end{proof}

\begin{lem}
\label{Lema_solutie_01_aaa_BIS_4_baza_introduction} Under the
hypotheses and the conclusions of Lemma
\ref{teorema_principala_partea1} (iii) or (iv), we have:

(i) $y$ $\in$ $Ker(DG(x))$, $y \neq 0$, if and only if $\exists$
$\theta_{q+m} \neq 0$ such that $B(y) = \theta_{q+m}$ and
$DG(x)y=0$.

(ii) Let $q+m = dim \ Ker(DG(x))$. Then, $y_{1}$, $\ldots$,
$y_{q+m}$ $\in$ $Ker(DG(x))$, $y_{i}$ $\neq 0$, $i=1,\ldots,q+m$,
form a basis for $Ker(DG(x))$ if and only if $\exists$
$\hat{\theta}_{1}$, $\ldots$, $\hat{\theta}_{q+m}$,
$\hat{\theta}_{i}$ $\neq 0$, $i=1,\ldots,q+m$, which form a basis
for $\mathbb{R}^{q+m}$, such that $B(y_{i}) = \hat{\theta}_{i}$
and $DG(x)y_{i}=0$.
\end{lem}

\begin{proof}
$\Phi_{G}(x,\cdot)$ is an isomorphism of $X$ onto $Y$.

\qquad
\end{proof}

\begin{lem}
\label{observatia_teorema_principala_partea5} Replacing
$\Phi_{G}(x,\cdot)$ by $\Phi_{H}(x,\cdot)$ in Lemma
\ref{teorema_principala_partea1} (iii) and (iv), Corollary
\ref{corolarul_teorema_principala_partea2}, Lemma
\ref{Lema_solutie_01_aaa_BIS_4_baza_introduction}, we obtain
similar results related to $H(\lambda,u,\cdot)$ and
$\bar{\mathcal{B}}$ instead of $DG(x)$ and $B$.
\end{lem}

\subsection{The properties of kernel of $DF(\lambda_{0},u_{0})$ and of
$D_{u}F(\lambda_{0},u_{0})$}
\label{NOU_sectiunea08_S3}

Let us replace $x$ by $x_{0}=(0,\lambda_{0},u_{0})$ and $y_{i}$ by
$y_{i,0}$ $=$ $(0,\mu_{i,0},w_{i,0})$ in Subsection
\ref{subsectiunea_I_partea_4_DFu_sis_DF_sis_x}.

\begin{lem}
\label{corolarul_teorema_principala_parteaMAIN} Under the
hypotheses of Theorem \ref{teorema_principala_parteaMAIN} (ii), we
have:
\begin{equation}
\label{NOU_e5_62_prim}
   \textrm{A basis of} \ Ker(DG(x_{0})) \ \textrm{is} \
      \{(0,\mu_{1,0},w_{1,0}),
      \ldots,
      (0,\mu_{q+m,0},w_{q+m,0})
      \} \, ,
\end{equation}
\begin{equation}
\label{NOU_e5_61}
   Ker(DG(x_{0}))=
   \{ (f,\lambda,u) \in X; f=0,
          (\lambda,u) \in Ker(DF(\lambda_{0},u_{0})) \} \, ,
\end{equation}
\begin{equation}
\label{NOU_e5_62_prim_prim}
   \textrm{A basis of} \ Ker(DF(\lambda_{0},u_{0})) \ \textrm{is} \
      \{(\mu_{1,0},w_{1,0}),
      \ldots,
      (\mu_{q+m,0},w_{q+m,0})
      \} \, ,
\end{equation}
\begin{equation}
\label{NOU_e5_63}
   dim \ Ker(DF(\lambda_{0},u_{0}))=q+m \, ,
\end{equation}
\begin{equation}
\label{NOU_e5_62_prim_H}
   \textrm{A basis of} \ Ker(H(\lambda_{0},u_{0},\cdot)) \ \textrm{is} \
      \{(0,v_{1,0}),
      \ldots,
      (0,v_{n,0})
      \} \, ,
\end{equation}
\begin{equation}
\label{NOU_e5_61_H}
   Ker(H(\lambda_{0},u_{0},\cdot))=
   \{ (e,v) \in \Delta; e=0,
          v \in Ker(D_{u}F(\lambda_{0},u_{0})) \} \, ,
\end{equation}
\begin{equation}
\label{NOU_e5_62_prim_prim_H}
   \textrm{A basis of} \ Ker(D_{u}F(\lambda_{0},u_{0})) \ \textrm{is} \
      \{v_{1,0},
      \ldots,
      v_{n,0}
      \} \, ,
\end{equation}
\begin{equation}
\label{NOU_e5_63_H}
   dim \ Ker(D_{u}F(\lambda_{0},u_{0}))=n \, ,
\end{equation}

$\chi_{1}$, $\ldots$, $\chi_{q+m}$ are linearly independent on
$Ker(DG(x_{0}))$,

$\bar{\chi}_{1}$, $\ldots$, $\bar{\chi}_{n}$ are linearly
independent on $Ker(H(\lambda_{0},u_{0},\cdot))$.
\end{lem}

\begin{proof}
Observe that $y_{i,0}$ $=$ $(0,\mu_{i,0},w_{i,0})$, where
$(\mu_{i,0},w_{i,0})$ $\in$ $Ker(DF(\lambda_{0},u_{0}))$.

$\Phi_{G}(x_{0},\cdot)$ is an isomorphism of $X$ onto $Y$, so we
have (\ref{e5_2_B_dim_S1_ii_P2}) and
(\ref{e5_2_B_plus_S1_i_P2_B0}). Hence, we obtain
(\ref{NOU_e5_62_prim}) and (\ref{NOU_e5_61}).

There results that the set $\{ (\mu_{1,0},w_{1,0})$, $\ldots$,
$(\mu_{q+m,0},w_{q+m,0}) \}$ is a linearly independent subset of
$Ker(DF(\lambda_{0},u_{0}))$. Suppose that $dim \
Ker(DF(\lambda_{0},u_{0}))$ $>$ $q+m$. So there exists $(\mu',w')$
$\in$ $Ker(DF(\lambda_{0},u_{0}))$, $(\mu',w') \neq 0$, such that
$\{(\mu_{1,0},w_{1,0})$, $\ldots$, $(\mu_{q+m,0},w_{q+m,0})$,
$(\mu',w') \}$ is a linearly independent subset of
$Ker(DF(\lambda_{0},u_{0}))$. Hence $(0,\mu',w')$ $\in$
$Ker(DG(x_{0}))$, so $dim \ Ker(DG(x_{0}))>q+m$, which contradicts
(\ref{e5_2_B_dim_S1_ii_P2}). It remains that
(\ref{NOU_e5_62_prim_prim}) and (\ref{NOU_e5_63}) are true.

\qquad
\end{proof}

\subsection{The decompositions of space $Z$}
\label{NOU_sectiunea09}

\begin{lem}
\label{Lema_Range_01} Under the hypotheses of Theorem
\ref{teorema_principala_parteaMAIN} (ii), the elements
$\bar{a}_{1}$, $\ldots$, $\bar{a}_{q}$ form a linearly independent
set and the elements $\bar{b}_{1},\ldots,\bar{b}_{n}$ form a
linearly independent set.
\end{lem}

\begin{proof}
Suppose that $\bar{a}_{1}$, $\ldots$, $\bar{a}_{q}$ form a
linearly dependent set, therefore there exists
$f'=((f')^{1},\ldots,(f')^{q}) \in \mathbb{R}^{q}$, $f' \neq 0$
and $\sum_{i=1}^{q}(f')^{i}\bar{a}_{i}=0$. Then, for an element
$(\mu,w) \in Ker(DF(\lambda_{0},u_{0}))$, we have $(f',\mu,w) \in
Ker(DG(x_{0}))$, with $f' \neq 0$. This is a contradiction.
\end{proof}

\begin{lem}
\label{Lema_Range_03} Under the hypotheses of Theorem
\ref{teorema_principala_parteaMAIN} (ii), we have:
\begin{equation}
\label{NOU_e5_80}
   Z=Z_{1} \oplus Range(DF(\lambda_{0},u_{0})) \, .
\end{equation}
\begin{equation}
\label{NOU_e5_80_H}
   Z=Z_{3} \oplus Range(D_{u}F(\lambda_{0},u_{0})) \, .
\end{equation}
\end{lem}

\begin{proof} We have $DG(x_{0})y=DF(\lambda_{0},u_{0})(\mu,w)
-\sum_{i=1}^{q}g^{i}\bar{a}_{i}$, where
$-\sum_{i=1}^{q}g^{i}\bar{a}_{i} \in Z_{1}$, $g_{i} \in
\mathbb{R}$, $i=1,\ldots,q$. So
\begin{equation}
\label{NOU_e5_80_3}
   Range(DG(x_{0}))=\{z \in Z ;
      z = z_{1} + z_{2},
      z_{1} \in Z_{1},
      z_{2} \in Range(DF(\lambda_{0},u_{0})) \} \, .
\end{equation}

Let us prove that
\begin{equation}
\label{NOU_e5_80_1}
   Z_{1} \cap Range(DF(\lambda_{0},u_{0})) = \{ 0 \} \, .
\end{equation}

Suppose that (\ref{NOU_e5_80_1}) is not true, so there exists $a
\in Z$, $a \neq 0$, such that $a \in Z_{1}$ and $a \in
Range(DF(\lambda_{0},u_{0}))$. So there exits
$g=(g^{1},\ldots,g^{q})$ $\in$ $\mathbb{R}^{q}$ and $(\lambda,u)$
$\in$ $\mathbb{R}^{m} \times W$ such that $a =
\sum_{i=1}^{q}g^{i}\bar{a}_{i}$ and $a$ $=$
$DF(\lambda_{0},u_{0})(\lambda,u)$. It follows that
$DF(\lambda_{0},u_{0})(\lambda,u)$ $-$
$\sum_{i=1}^{q}g^{i}\bar{a}_{i}$ $= 0$, i.e. $(g,\lambda,u) \in
Ker(DG(x_{0}))$. From (\ref{NOU_e5_61}), there results that $g =
0$, so $a = 0$. This contradicts the supposition that $a \neq 0$,
hence (\ref{NOU_e5_80_1}).

From (\ref{NOU_e5_80_3}) and (\ref{NOU_e5_80_1}), it follows that
\begin{equation}
\label{NOU_e5_80_2}
   Range(DG(x_{0}))=Z_{1} \oplus Range(DF(\lambda_{0},u_{0})) \, .
\end{equation}

Since $\Phi_{G}(x_{0},\cdot)$ is an isomorphism of $X$ onto $Y$,
it follows follows that $Range(\Phi_{G}(x_{0},\cdot))$ $=$ $Y$ ,
therefore
\begin{equation}
\label{NOU_e5_81}
   Range(DG(x_{0}))=Z \, .
\end{equation}

From (\ref{NOU_e5_81}) and (\ref{NOU_e5_80_2}), there results
(\ref{NOU_e5_80}).

\qquad
\end{proof}

\begin{lem}
\label{Lema_Range_02} Under the hypotheses of Theorem
\ref{teorema_principala_parteaMAIN} (ii), we have:
\begin{equation}
\label{NOU_e5_70}
   codim \ Z_{2}=q \, ,
      \quad Z_{2} \ \textrm{is closed in} \ Z \, ,
\end{equation}
\begin{equation}
\label{NOU_e5_70_H}
   codim \ Z_{4}=n \, ,
      \quad Z_{4} \ \textrm{is closed in} \ Z \, .
\end{equation}
\end{lem}

\begin{proof}
From (\ref{NOU_e5_80}) and Lemma \ref{Lema_Range_01}, we have
$codim \ Z_{2}=dim \ Z_{1}=q$. We obtain (\ref{NOU_e5_70}).

\qquad
\end{proof}

\subsection{Properties of the Fréchet derivatives and the consequences
on the qualitative aspects}
\label{NOU_sectiunea12}

\begin{lem}
\label{Lema_Frechet_01} Under the hypotheses of Theorem
\ref{teorema_principala_parteaMAIN} (ii), we
have:  \\
(a) $DF(\lambda_{0},u_{0})$ is a Fredholm operator of
$\mathbb{R}^{m} \times W$ onto $Z$ with index $m$;  \\
(b) $D_{u}F(\lambda_{0},u_{0})$ is a Fredholm operator of $W$ onto
$Z$ with index zero.
\end{lem}

\begin{proof}
The proof follows from (\ref{NOU_e5_63}), (\ref{NOU_e5_70}) and
(\ref{NOU_e5_63_H}), (\ref{NOU_e5_70_H}).

\qquad
\end{proof}

\subsection{The proof of implication (i) $\Rightarrow$ (ii)
of Theorem \ref{teorema_principala_parteaMAIN}}
\label{subcap_dem_implicatia_directa}

From the implications (i) $\Rightarrow$ (ii) and (iii)
$\Rightarrow$ (iv) of Lemma \ref{teorema_principala_partea1} and
from Remark \ref{observatie_teorema_principala_partea1}, we deduce
that $x_{0}$ is a solution of (\ref{e5_22}),
$\Phi_{G}(x_{0},\cdot)$ is an isomorphism of $X$ onto $Y$ and
$\Phi_{H}(x_{0},\cdot)$ is an isomorphism of $\Delta$ onto
$\Sigma$. From (\ref{e5_56}), we have the equations
$\Phi_{G}(x_{0},y_{i,0})-[\delta_{i}^{q+m},0]^{T}=0$,
$\Phi_{H}(x_{0},\mathrm{z}_{k,0})-[\delta_{k}^{n},0]^{T}=0$. From
(\ref{e5_2_B_int}), (\ref{e5_2_H_B}), (\ref{e5_2_B_int_ker}) and
(\ref{e5_2_H_B_ker}), we deduce $g_{i,0}=0$, $e_{k,0}=0$. The
proof is complete by applying Lemma \ref{Lema_solutie_01_DS3}.

\subsection{The proof of implication (ii) $\Rightarrow$ (i)
of Theorem \ref{teorema_principala_parteaMAIN}}
\label{subcap_dem_implicatia_reciproca} The proof follows from the
implications (ii) $\Rightarrow$ (i) and (iv) $\Rightarrow$ (iii)
of Lemma \ref{teorema_principala_partea1} where $x$ is $x_{0}$,
from Remark \ref{observatie_teorema_principala_partea1}, from
Corollary \ref{corolarul_teorema_principala_partea2} with $x$ is
$x_{0}$, from Lemma
\ref{Lema_solutie_01_aaa_BIS_4_baza_introduction} to Lemma
\ref{Lema_Frechet_01}.

\section{The main result on the existence of a bifurcation problem}
\label{sectiunea_01_O_formulare_pe_spatii_infinit_dimensionale}

\subsection{Introduction}
\label{sectiunea_01_O_formulare_pe_spatii_infinit_dimensionale_1}

We formulate some sufficient conditions for the existence of an
equation which has a bifurcation point that satisfies the
hypothesis (\ref{ipotezaHypF}). We also obtain the existence of a
bifurcation problem for which a given problem is a perturbation.

In the sequel, in Sections
\ref{sectiunea_01_O_formulare_pe_spatii_infinit_dimensionale} -
\ref{sectiunea_CONCLUZII_0}, we keep the notations from the
preceding sections, but we do not consider the equations they
define. Equations (\ref{e5_1}), (\ref{e5_3}) and (\ref{e5_56}),
hypothesis (\ref{ipotezaHypF}) and Theorem
\ref{teorema_principala_parteaMAIN} will be related to a function
$S_{0}$ obtained as a perturbation of $S$. Let us fix these:

Let $W$ and $Z$ be real Banach spaces. They are both
infinite-dimensional spaces or they are both finite-dimensional
spaces with $dim \ W$ $=$ $dim \ Z$. Let $m \geq 1$, $p \geq 2$.
Let $F:\mathbb{R}^{m} \times W \rightarrow Z$ be a nonlinear
function of class $C^{p}$. Let $q \geq 1$, $n \geq 1$,
$\bar{a}_{1},\ldots,\bar{a}_{q}$, $\bar{b}_{1},\ldots,\bar{b}_{n}$
$\in$ $Z \backslash \{ 0 \}$. Let us consider $G$, $H$ defined in
(\ref{e5_2}), (\ref{e5_2_H}) and $B$ $\in$
$L(X,\mathbb{R}^{q+m})$, $\bar{\mathcal{B}}$ $\in$
$L(\Delta,\mathbb{R}^{n})$. Let us take some points
$\widetilde{s}_{0}$, $\widetilde{\phi}_{0}'$ $\in$ $\Gamma$. Let
$S$ and $\Phi$ be the functions defined in (\ref{e5_57}), for some
$\theta_{0}$ $\in$ $\mathbb{R}^{q+m}$, and (\ref{e5_57_Phi_3}),
respectively.

Let us construct the function $S$ from (\ref{e5_57}), assuming the
existence of the elements that allow this construction. We seek a
function $S_{0}$ of the same form (\ref{e5_57}) such that the
equation $S_{0}(s)$ $=$ $0$ has a solution $s_{0}$ and both
$S_{0}$ and $s_{0}$ satisfy the statement (ii) of Theorem
\ref{teorema_principala_parteaMAIN}.

We started with the following analysis. For the function $S$ from
(\ref{e5_57}), consider a point $\widetilde{s}_{0}$ of the form of
$s_{0}$, that is, $\widetilde{f}_{0}=0$, $\widetilde{g}_{i,0}=0$,
$\widetilde{e}_{k,0}=0$. Assume that $DS(\widetilde{s}_{0})$ is an
isomorphism of $\Gamma$ onto $\Sigma$. Under an additional
hypothesis, according to Theorem 3.1
\cite{CLBichir_bib_Cr_Ra1990}, Theorem IV.3.1
\cite{CLBichir_bib_Gir_Rav1986} and Theorem I.2.1
\cite{CLBichir_bib_CalozRappaz1997}, equation $S(s)$ $=$ $0$ (that
is, (\ref{e5_56})) has a solution $s$ in a neighborhood of
$\widetilde{s}_{0}$. We are interested in the case that the
solution $s$ is of the form of $s_{0}$ from Theorem
\ref{teorema_principala_parteaMAIN} (ii) (a). Let
$\widetilde{\varrho}_{0}$ $=$ $S(\widetilde{s}_{0})$. Applying the
inverse function theorem, $S$ is a local $C^{p}$-diffeomorphism at
$\widetilde{s}_{0}$ from $\widetilde{\mathcal{U}}$ onto
$\widetilde{\mathcal{V}}$, where $\widetilde{\mathcal{U}}$ is a
neighborhood of $\widetilde{s}_{0}$ and $\widetilde{\mathcal{V}}$
is a neighborhood of $\widetilde{\varrho}_{0}$. Let us take $\Phi$
defined in (\ref{e5_57_Phi_3}). $\Phi(x,\cdot)$ is an isomorphism
of $\Gamma$ onto $\Sigma$ for $x$ from some neighborhood of
$\widetilde{x}_{0}$. For each $s$ $\in$ $\widetilde{\mathcal{U}}$,
we can consider the equation in $\phi'$ $\in$ $\Gamma$,
\begin{equation}
\label{corolarul_trei_3_fi_3_ec_e5_57_forma2_sistem_REG_2_GGGvarianta_infinit_dimensionale_DEM_pc_fix_cont_EC}
   S(s)
   - \Phi(x,\phi')
   = 0 \, .
\end{equation}
Let us consider now this equation for $(s,\phi')$ $\in$
$\widetilde{\mathcal{U}} \times \Gamma$. Let us observe that the
form of the left hand side of
(\ref{corolarul_trei_3_fi_3_ec_e5_57_forma2_sistem_REG_2_GGGvarianta_infinit_dimensionale_DEM_pc_fix_cont_EC})
can be compared, especially the rows that correspond to
$\Phi_{G}(x,y_{i})$ and $\Phi_{H}(x,\mathrm{z}_{k})$, with the
form of the function $S(s)$ defined in (\ref{e5_57}). We then ask
if there exists a solution $(s,\phi')$ of
(\ref{corolarul_trei_3_fi_3_ec_e5_57_forma2_sistem_REG_2_GGGvarianta_infinit_dimensionale_DEM_pc_fix_cont_EC})
which allows us to find the function $S_{0}$ and the solution
$s_{0}$.

We obtain the function $S_{0}$ as a perturbation of $S$. The main
result is in Theorem
\ref{teorema_principala_spatii_infinit_dimensionale_widetilde_s_3_0_exact_inf}.

The idea of the proof came when we read the proofs of Theorem 6C.1
and Theorem 6D.1, from
\cite{CLBichir_bib_Dontchev_Rockafellar2009}. Here, some mappings
$G_{\textrm{\cite{CLBichir_bib_Dontchev_Rockafellar2009}}}$ and
$g_{\textrm{\cite{CLBichir_bib_Dontchev_Rockafellar2009}}}$ are
constructed and, under adequate hypotheses, the existence of a
solution of the equation
$g_{\textrm{\cite{CLBichir_bib_Dontchev_Rockafellar2009}}}(x) \in
G_{\textrm{\cite{CLBichir_bib_Dontchev_Rockafellar2009}}}(x)$ is
proved (using the contraction mapping principle for set-valued
mappings, Theorem 5E.2
\cite{CLBichir_bib_Dontchev_Rockafellar2009}). For the formulation
of the results and of the proofs, we also take into account
Graves' theorem 5D.2 \cite{CLBichir_bib_Dontchev_Rockafellar2009}
together with one of its consequences and the techniques from the
proofs of Theorem 3.1 \cite{CLBichir_bib_Cr_Ra1990}, Theorem
IV.3.1 \cite{CLBichir_bib_Gir_Rav1986} and Theorem I.2.1
\cite{CLBichir_bib_CalozRappaz1997}.

The analysis of
(\ref{corolarul_trei_3_fi_3_ec_e5_57_forma2_sistem_REG_2_GGGvarianta_infinit_dimensionale_DEM_pc_fix_cont_EC})
and the \cite{CLBichir_bib_Dontchev_Rockafellar2009} method
mentioned above lead us to the construction of the mappings
$\mathcal{G}(s,\phi')$ and $\mathcal{Q}(s,\phi')$ from
(\ref{e5_57_forma2_sistem_REG_2_GGGvarianta_infinit_dimensionale_DEM})
and
(\ref{e5_57_forma2_sistem_REG_2_GGGmic_infinit_dimensionale_DEM}).
The aim of this construction is that equation
$\mathcal{Q}(s,\phi') \in \mathcal{G}(s,\phi')$ and a solution
$(\bar{s},\bar{\phi}')$ of this one generate a function $S_{0}$ of
the form (\ref{e5_57}) and a solution $s_{0}$, with $f_{0}=0$,
$g_{i,0}=0$, $e_{k,0}=0$, $i=1,\ldots,q+m$, $k=1,\ldots,n$, of the
equation $S_{0}(s)$ $=$ $0$. $S_{0}$ and $s_{0}$ satisfy the
statements (a) and (b) of Theorem
\ref{teorema_principala_parteaMAIN} (ii).

We obtain
\begin{equation}
\label{corolarul_trei_3_fi_3_e5_57_forma2_sistem_REG_2_GGGvarianta_infinit_dimensionale_DEM_pc_fix_cont}
   S(s_{0})
   - \Phi(x_{0},\phi_{0}')
   \ni 0 \, ,
\end{equation}
that is, $(s_{0},\phi_{0}')$ is a solution of the equation
\begin{equation}
\label{corolarul_trei_3_fi_3_ec_e5_57_forma2_sistem_REG_2_GGGvarianta_infinit_dimensionale_DEM_pc_fix_cont}
   S(s)
   - \Phi(x,\phi')
   \ni 0 \, .
\end{equation}
This is equation
(\ref{e5_57_forma2_sistem_REG_2_GGGvarianta_infinit_dimensionale_DEM_pc_fix_cont_th_princ}).
Relation
(\ref{corolarul_trei_3_fi_3_e5_57_forma2_sistem_REG_2_GGGvarianta_infinit_dimensionale_DEM_pc_fix_cont})
allows us to get $S_{0}$ and $s_{0}$. Corollary
\ref{corolarul_trei_3_teorema_principala_spatii_infinit_dimensionale_widetilde_s_3_0_exact_zero}
tells us that equation $S_{0}(s)$ $=$ $0$ provides a bifurcation
problem $F(\lambda,u) - \varrho$ $=$ $0$ for which the given
problem (\ref{e5_1}) is a perturbation.

\subsection{The definition of the mappings $\mathcal{G}(s,\phi')$ and
$\mathcal{Q}(s,\phi')$}
\label{sectiunea_01_O_formulare_pe_spatii_infinit_dimensionale_2}

\textbf{- The points $\widetilde{s}_{0}$ and
$\widetilde{\phi}_{0}'$}

We denote $\widetilde{s}_{0}$ $=$
$(\widetilde{x}_{0},\widetilde{y}_{i,0},\widetilde{\mathrm{z}}_{k,0})$
and $\widetilde{\phi}_{0}'$ $=$
$(\widetilde{y}_{0}',\widetilde{y}_{i,0}',\widetilde{\mathrm{z}}_{k,0}')$
$\in$ $\Gamma$.
\\
$\widetilde{s}_{0}$ $=$
$(\widetilde{x}_{0},\widetilde{y}_{1,0},\ldots,\widetilde{y}_{q+m,0},\widetilde{\mathrm{z}}_{1,0},\ldots,\widetilde{\mathrm{z}}_{n,0})$,
$\widetilde{\phi}_{0}'$ $=$
$(\widetilde{y}_{0}',\widetilde{y}_{1,0}',\ldots,\widetilde{y}_{q+m,0}',\widetilde{\mathrm{z}}_{1,0}',\ldots,\widetilde{\mathrm{z}}_{n,0}')$.
\\
$\widetilde{x}_{0}$ $=$
$(\widetilde{f}_{0},\widetilde{\lambda}_{0},\widetilde{u}_{0})$,
$\widetilde{y}_{i,0}$ $=$
$(\widetilde{g}_{i,0},\widetilde{\mu}_{i,0},\widetilde{w}_{i,0})$,
$\widetilde{\mathrm{z}}_{k,0}$ $=$
$(\widetilde{e}_{k,0},\widetilde{v}_{k,0})$, \\
$\widetilde{y}_{0}'$ $=$
$(\widetilde{g}_{0}',\widetilde{\mu}_{0}',\widetilde{w}_{0}')$,
$\widetilde{y}_{i,0}'$ $=$
$(\widetilde{g}_{i,0}',\widetilde{\mu}_{i,0}',\widetilde{w}_{i,0}')$,
$\widetilde{\mathrm{z}}_{k,0}'$ $=$
$(\widetilde{e}_{k,0}',\widetilde{v}_{k,0}')$ \\
and $\widetilde{f}_{0}=0$, $\widetilde{g}_{i,0}=0$,
$\widetilde{e}_{k,0}=0$, $\widetilde{y}_{0}'=0$,
$\widetilde{\phi}_{0}'=0$, $i=1,\ldots,q+m$, $k=1,\ldots,n$.

Let us take the value $\widetilde{\theta}_{0}$ $=$
$B(\widetilde{x}_{0})$ $\in$ $\mathbb{R}^{q+m}$ for $\theta_{0}$
in the definition (\ref{e5_23}) of $\Psi$ and in the definition
(\ref{e5_57}) of $S$. Let $\Phi$ be defined by
(\ref{e5_57_Phi_3}).

Let us denote \\
$\xi(g',g_{i}',e_{k}')$ $=$ $\widetilde{\phi}_{0}'$ $+$
$((g',0,0),(g_{1}',0,0),\ldots,(g_{q+m}',0,0),(e_{1}',0,0),\ldots,(e_{n}',0,0))$,
that is, $\xi(g',g_{i}',e_{k}')$ is $\widetilde{\phi}_{0}'$ where
$\widetilde{g}_{0}'$, $\widetilde{g}_{i,0}'$,
$\widetilde{e}_{k,0}'$ are replaced by $\widetilde{g}_{0}'+g'$,
$\widetilde{g}_{i,0}'+g_{i}'$, $\widetilde{e}_{k,0}'+e_{k}'$
respectively. We have
\begin{equation}
\label{e5_57_Phi_3_dif_0_tilde_0_Phi_prim}
   \Phi(x,\xi(g',g_{i}',e_{k}'))=\left[\begin{array}{l}
      B(g',0,0) \\
      -\sum_{\imath=1}^{q}(g')^{\imath}\bar{a}_{\imath} \\
      B(g_{i}',0,0) \\
      -\sum_{\ell=1}^{q}(g_{i}')^{\ell}\bar{a}_{\ell} \\
      \bar{\mathcal{B}}(e_{k}',0) \\
      -\sum_{\jmath=1}^{n}(e_{k}')^{\jmath}\bar{b}_{\jmath}
      \end{array}\right] \ .
\end{equation}

\textbf{- The mappings $\mathcal{G}$ and $\mathcal{Q}$}

Let us fix a real $\alpha$, $0 < \alpha < 1$, $\alpha \neq
\frac{1}{2}$.

Let us define
\begin{equation}
\label{e5_57_forma2_sistem_REG_2_GGGvarianta_infinit_dimensionale_DEM_mathcal_H_S_3}
   \mathcal{H}_{S}(s)
   = S(\widetilde{s}_{0})
   + DS(\widetilde{s}_{0})(s-\widetilde{s}_{0})
   \, ,
\end{equation}

\begin{equation}
\label{corlar_e5_57_forma2_sistem_REG_2_GGGvarianta_infinit_dimensionale_DEM_mathcal_H_Phi_3}
   \mathcal{H}_{\Phi}(x,\phi')
   = \Phi(\widetilde{x}_{0},\widetilde{\phi}_{0}')
   + D\Phi(\widetilde{x}_{0},\widetilde{\phi}_{0}')((x,\phi')-(\widetilde{x}_{0},\widetilde{\phi}_{0}'))
   \, .
\end{equation}

$\mathcal{H}_{S}$ is a strict first-order approximation of $S$ at
$\widetilde{s}_{0}$.

$\mathcal{H}_{\Phi}$ is a strict first-order approximation of
$\Phi$ at $(\widetilde{x}_{0},\widetilde{\phi}_{0}')$.

$\mathcal{G}:\Gamma \times \Gamma \rightarrow \Sigma$,
\begin{eqnarray}
   && \mathcal{G}(s,\phi')
   = \frac{1}{2}S(s)
   - \frac{1}{2}\Phi(x,\phi')
   + \frac{1}{2}\mathcal{H}_{S}(s)
   - \frac{1}{2}\mathcal{H}_{\Phi}(x,\phi')
          \label{e5_57_forma2_sistem_REG_2_GGGvarianta_infinit_dimensionale_DEM} \\
   && - (1-\alpha) \Phi(\widetilde{x}_{0},\xi(f,g_{i},e_{k}))
   + (1-\alpha) \Phi(\widetilde{x}_{0},\xi(g',g_{i}',e_{k}'))
   \, .
         \nonumber
\end{eqnarray}

$\mathcal{Q}:\Gamma \times \Gamma \rightarrow \Sigma$,
\begin{eqnarray}
   && \mathcal{Q}(s,\phi')
   = - \frac{1}{2}S(s)
   + \frac{1}{2}\Phi(x,\phi')
   + \frac{1}{2}\mathcal{H}_{S}(s)
   - \frac{1}{2}\mathcal{H}_{\Phi}(x,\phi')
          \label{e5_57_forma2_sistem_REG_2_GGGmic_infinit_dimensionale_DEM} \\
   && + \alpha \Phi(\widetilde{x}_{0},\xi(f,g_{i},e_{k}))
   - \alpha \Phi(\widetilde{x}_{0},\xi(g',g_{i}',e_{k}'))
   \, .
         \nonumber
\end{eqnarray}

Observe that we have
\begin{equation*}
\label{corlar_e5_57_forma2_sistem_REG_2_GGGvarianta_infinit_dimensionale_DEM_DIFERENTA}
   \Phi(x,\xi(f,g_{i},e_{k}))
   - \Phi(x,\xi(g',g_{i}',e_{k}'))
   =
   \Phi(\widetilde{x}_{0},\xi(f,g_{i},e_{k}))
   - \Phi(\widetilde{x}_{0},\xi(g',g_{i}',e_{k}'))
   \, ,
\end{equation*}
\begin{equation}
\label{dem_e5_57_forma2_sistem_REG_2_GGGvarianta_infinit_dimensionale_DEM_val}
   \mathcal{G}(\widetilde{s}_{0},\widetilde{\phi}_{0}')
   = S(\widetilde{s}_{0}) \, , \
   \mathcal{Q}(\widetilde{s}_{0},\widetilde{\phi}_{0}')
   = 0 \, .
\end{equation}

\begin{lem}
\label{Lema_mathcal_G_surjective} $\mathcal{G}$ is surjective.
\end{lem}

\begin{proof} Using
(\ref{e5_57_forma2_sistem_REG_2_GGGvarianta_infinit_dimensionale_DEM})
and (\ref{e5_57_Phi_3_dif_0_tilde_obs}), we have
\begin{eqnarray*}
   && \mathcal{G}(\widetilde{s}_{0},\phi')
   = S(\widetilde{s}_{0})
   - \Phi(\widetilde{x}_{0},\phi')
   + (1-\alpha) \Phi(\widetilde{x}_{0},\xi(g',g_{i}',e_{k}'))
   \, .
             \label{e5_57_forma2_sistem_REG_2_GGGvarianta_infinit_dimensionale_DEM_tilde}
\end{eqnarray*}

From Lemma \ref{Lema_solutie_01_DS3_fi3}, we have that
$\Phi(\widetilde{x}_{0},\cdot)$ is an isomorphism of $\Gamma$ onto
$\Sigma$. Hence $\forall$ $\zeta$ $\in$ $\Sigma$, $\exists$
$(\widetilde{s}_{0},\phi')$ $\in$ $\Gamma \times \Gamma$ such that
$\mathcal{G}(\widetilde{s}_{0},\phi')$ $=$ $\zeta$ and
$\mathcal{G}$ is surjective.

\qquad
\end{proof}

We use the surjectivity of $\mathcal{G}$ in the Proof of Theorem
\ref{teorema_principala_spatii_infinit_dimensionale_widetilde_s_3_0_exact_inf},
Subsection
\ref{sectiunea_01_O_formulare_pe_spatii_infinit_dimensionale_5}.
We mention that we can skip this condition since we work with
set-valued mappings; for details for the methodology, see
\cite{CLBichir_bib_Dontchev_Rockafellar2009}.

\textbf{- The operator $A$ $=$ $\mathcal{A}$}

\begin{equation}
\label{e5_57_forma2_sistem_REG_2_GGGvarianta_infinit_dimensionale_DEM_dif_def}
   \mathcal{A}
   = D\mathcal{G}(\widetilde{s}_{0},\widetilde{\phi}_{0}')
   \, .
\end{equation}

We write $\mathcal{A}(\alpha)$ when it is necessary to consider
$\alpha$ as a variable.

\begin{eqnarray}
   && D\mathcal{G}(s,\phi')(\overline{s},\overline{\phi}')
   = \frac{1}{2}DS(s)\overline{s}
   - \frac{1}{2}D\Phi(x,\phi')(\overline{x},\overline{\phi}')
   - \frac{1}{2}D\Phi(\widetilde{x}_{0},\widetilde{\phi}_{0}')(\overline{x},\overline{\phi}')
          \label{e5_57_forma2_sistem_REG_2_GGGvarianta_infinit_dimensionale_DEM_dif} \\
   && + \frac{1}{2}DS(\widetilde{s}_{0})\overline{s}
   - (1-\alpha) \Phi(\widetilde{x}_{0},\xi(\overline{f},\overline{g}_{i},\overline{e}_{k}))
   + (1-\alpha) \Phi(\widetilde{x}_{0},\xi(\overline{g}',\overline{g}_{i}',\overline{e}_{k}'))
   \, .
         \nonumber
\end{eqnarray}

\begin{eqnarray}
   && \mathcal{A}(\overline{s},\overline{y}')
   = DS(\widetilde{s}_{0})\overline{s}
   - D\Phi(\widetilde{x}_{0},\widetilde{\phi}_{0}')(\overline{x},\overline{\phi}')
          \label{e5_57_forma2_sistem_REG_2_GGGvarianta_infinit_dimensionale_DEM_dif_tilde} \\
   && - (1-\alpha) \Phi(\widetilde{x}_{0},\xi(\overline{f},\overline{g}_{i},\overline{e}_{k}))
   + (1-\alpha) \Phi(\widetilde{x}_{0},\xi(\overline{g}',\overline{g}_{i}',\overline{e}_{k}'))
   \, ,
         \nonumber
\end{eqnarray}
where
\begin{equation}
\label{e5_27_S_3_0_tilde}
       DS(\widetilde{s}_{0})\overline{s}=\left[\begin{array}{l}
      B(\overline{x}) \\
      DG(\widetilde{x}_{0})\overline{x} \\
      B(\overline{y}_{i}) \\
      D^{2}F(\widetilde{\lambda}_{0},\widetilde{u}_{0})
         ((\widetilde{\mu}_{i,0},\widetilde{w}_{i,0}),(\overline{\lambda},\overline{u}))
            +DG(\widetilde{x}_{0})\overline{y}_{i} \\
      \bar{\mathcal{B}}(\overline{\mathrm{z}}_{k}) \\
      D_{(\lambda,u)}(D_{u}F(\widetilde{\lambda}_{0},\widetilde{u}_{0})\widetilde{v}_{k,0})
         (\overline{\lambda},\overline{u})
            +H(\widetilde{\lambda}_{0},\widetilde{u}_{0},\overline{\mathrm{z}}_{k})
      \end{array}\right] \, ,
\end{equation}
and
\begin{equation}
\label{e5_57_Phi_3_dif_0_tilde}
   D\Phi(\widetilde{x}_{0},\widetilde{\phi}_{0}')(\overline{x},\overline{\phi}')=\left[\begin{array}{l}
      B(\overline{y}') \\
      D^{2}F(\widetilde{\lambda}_{0},\widetilde{u}_{0})
         ((\widetilde{\mu}_{0}',\widetilde{w}_{0}'),(\overline{\lambda},\overline{u}))
            +DG(\widetilde{x}_{0})\overline{y}' \\
      B(\overline{y}_{i}') \\
      D^{2}F(\widetilde{\lambda}_{0},\widetilde{u}_{0})
         ((\widetilde{\mu}_{i,0}',\widetilde{w}_{i,0}'),(\overline{\lambda},\overline{u}))
            +DG(\widetilde{x}_{0})\overline{y}_{i}' \\
      \bar{\mathcal{B}}(\overline{\mathrm{z}}_{k}') \\
      D_{(\lambda,u)}(D_{u}F(\widetilde{\lambda}_{0},\widetilde{u}_{0})\widetilde{v}_{k,0}')
         (\overline{\lambda},\overline{u})
            +H(\widetilde{\lambda}_{0},\widetilde{u}_{0},\overline{\mathrm{z}}_{k}')
      \end{array}\right] \ ,
\end{equation}
for all $i=1,\ldots,q+m$, $k=1,\ldots,n$. We have
$D^{2}F(\widetilde{\lambda}_{0},\widetilde{u}_{0})
         ((\widetilde{\mu}_{0}',\widetilde{w}_{0}'),(\overline{\lambda},\overline{u})) = 0$,
$D^{2}F(\widetilde{\lambda}_{0},\widetilde{u}_{0})
         ((\widetilde{\mu}_{i,0}',\widetilde{w}_{i,0}'),(\overline{\lambda},\overline{u})) = 0$,
$D_{(\lambda,u)}(D_{u}F(\widetilde{\lambda}_{0},\widetilde{u}_{0})\widetilde{v}_{k,0}')
         (\overline{\lambda},\overline{u}) = 0$.

\textbf{- Condition $A$ $=$ $\mathcal{A}$ is surjective} By fixing
$\overline{y}'$, the condition is verified immediately.

\textbf{- The definitions of $reg \, \mathcal{A}$, $\kappa$ and
$\gamma$}

$reg \, \mathcal{A}$ is defined as in (\ref{e_A2_17_TEXT_reg}). We
take $\kappa$ $=$ $reg \, \mathcal{A}$. Define $\gamma_{\Psi}$ $=$
$\widetilde{\gamma}(\Psi,\widetilde{x}_{0},X,Y)$, $\gamma$ $=$
$\gamma_{S}$ $=$
$\widetilde{\gamma}(S,\widetilde{s}_{0},\Gamma,\Sigma)$,
$\gamma_{S_{0}}$ $=$
$\widetilde{\gamma}(S_{0},\widetilde{s}_{0},\Gamma,\Sigma)$, where
we use (\ref{e_A2_17_TEXT_gamma_bar_gen}). Observe that $\gamma$
do not depend on $\alpha$.

\begin{lem}
\label{Lema_spatii_infinit_dimensionale_reg}
\begin{equation}
\label{e_A2_16_TEXT}
   \kappa \leq \gamma \, .
\end{equation}
\end{lem}

\begin{proof}
We have $\widetilde{\phi}_{0}'=0$. (\ref{e5_57_Phi_3_dif_0_tilde})
gives
\begin{equation}
\label{e5_57_Phi_3_dif_0_tilde_obs}
   \Phi(\widetilde{x}_{0},\overline{\phi}')
   = D\Phi(\widetilde{x}_{0},\widetilde{\phi}_{0}')(\overline{x},\overline{\phi}') \ ,
\end{equation}
and we replace
$D\Phi(\widetilde{x}_{0},\widetilde{\phi}_{0}')(\overline{x},\overline{\phi}')$
by $\Phi(\widetilde{x}_{0},\overline{\phi}')$ in the expression
(\ref{e5_57_forma2_sistem_REG_2_GGGvarianta_infinit_dimensionale_DEM_dif_tilde})
of $\mathcal{A}$.

We denote
\begin{equation}
\label{abc_e5_57_forma2_sistem_REG_2_GGGvarianta_infinit_dimensionale_DEM_dif_tilde}
   DS_{\alpha}(\widetilde{s}_{0})\overline{s}
   =DS(\widetilde{s}_{0})\overline{s}
   - (1-\alpha) \Phi(\widetilde{x}_{0},\xi(\overline{f},\overline{g}_{i},\overline{e}_{k}))
   \, ,
\end{equation}
\begin{equation}
\label{abc_e5_57_forma2_sistem_REG_2_GGGvarianta_infinit_dimensionale_DEM_dif_tilde}
   \Phi_{\alpha}(\widetilde{x}_{0},\overline{\phi}')
   =\Phi(\widetilde{x}_{0},\overline{\phi}')
   - (1-\alpha) \Phi(\widetilde{x}_{0},\xi(\overline{g}',\overline{g}_{i}',\overline{e}_{k}'))
   \, .
\end{equation}
Hence
\begin{equation}
\label{abc_e5_57_forma2_sistem_REG_2_GGGvarianta_infinit_dimensionale_DEM_dif_tilde}
   \mathcal{A}(\overline{s},\overline{y}')
   = DS_{\alpha}(\widetilde{s}_{0})\overline{s}
   - \Phi_{\alpha}(\widetilde{x}_{0},\overline{\phi}')
   \, .
\end{equation}

For $y$ $\in$ $\Sigma$, we have the sets $\mathcal{E}_{y}$,
$\mathcal{E}_{y}^{0}$ such that

$\mathcal{E}_{y}$ $=$ $\mathcal{A}^{-1}(y)$ $=$
$(DS_{\alpha}(\widetilde{s}_{0}) -
\Phi_{\alpha}(\widetilde{x}_{0},\cdot))^{-1}(y)$ $=$ $\{
(\overline{s},\overline{\phi}') |
DS_{\alpha}(\widetilde{s}_{0})\overline{s}
   - \Phi_{\alpha}(\widetilde{x}_{0},\overline{\phi}') = y \}$ $\supseteq$ $\{
(\overline{s},\overline{\phi}') |
DS_{\alpha}(\widetilde{s}_{0})\overline{s} = y, \,
\overline{\phi}' = 0 \}$ $=$ $\{ (\overline{s},\overline{\phi}') |
\overline{s} = DS_{\alpha}(\widetilde{s}_{0})^{-1}(y), \,
\overline{\phi}' = 0 \}$ $=$ $\mathcal{E}_{y}^{0}$.

\begin{equation}
\label{e_A2_17_TEXT_reg_lema}
   reg \, \mathcal{A} = \sup_{\| y \| \leq 1} d(0,\mathcal{A}^{-1}(y))
   = \sup_{\| y \| \leq 1} d(0,\mathcal{E}_{y})
\end{equation}
\begin{equation*}
\label{abc_e_A2_17_TEXT_reg_lema}
   \leq \sup_{\| y \| \leq 1} d(0,\mathcal{E}_{y}^{0})
   = \sup_{\| y \| \leq 1} d(0,
      \{ \overline{s} | \overline{s} =
      DS_{\alpha}(\widetilde{s}_{0})^{-1}(y) \})
\end{equation*}
\begin{equation*}
\label{abc_e_A2_17_TEXT_reg_lema}
   = \sup_{\| y \| \leq 1} \| DS_{\alpha}(\widetilde{s}_{0})^{-1}(y) \|
   = \| DS(\widetilde{s}_{0})^{-1} \|_{L(\Sigma,\Gamma)}
   \, .
\end{equation*}

\qquad
\end{proof}

\textbf{- The definition of $\mu$ $=$ $L(\varepsilon)$}

We denote $\widetilde{\mathcal{G}}_{3}(s,\phi')
   = \frac{1}{2}S(s)
   - \frac{1}{2}\Phi(x,\phi')$.

$D\widetilde{\mathcal{G}}_{3}(s,\phi')(\overline{s},\overline{\phi}')
   = \frac{1}{2}DS(s)\overline{s}
   - \frac{1}{2}D\Phi(x,\phi')(\overline{x},\overline{\phi}')$

$\psi_{1}((s,\phi'),(\overline{s},\overline{\phi}'))$

$=$

$\| \widetilde{\mathcal{G}}_{3}(s,\phi')
   - \widetilde{\mathcal{G}}_{3}(\overline{s},\overline{\phi}')
   - D\widetilde{\mathcal{G}}_{3}(\widetilde{s}_{0},\widetilde{\phi}_{0}')
      ((s,\phi') - (\overline{s},\overline{\phi}')) \|$

$=$

$\|
\int_{0}^{1}[D\widetilde{\mathcal{G}}_{3}((\overline{s},\overline{\phi}')+t((s,\phi')-(\overline{s},\overline{\phi}')))
-
D\widetilde{\mathcal{G}}_{3}(\widetilde{s}_{0},\widetilde{\phi}_{0}')]
\cdot ((s,\phi') - (\overline{s},\overline{\phi}')) dt \|$,

where we use some standard techniques from
\cite{CLBichir_bib_CalozRappaz1997, CLBichir_bib_Cr_Ra1990,
CLBichir_bib_Gir_Rav1986, CLBichir_bib_S_Lang1993}.

$\psi_{2}((s,\phi'),(\overline{s},\overline{\phi}'))$ $=$

$\sum_{\imath=1}^{q}|f^{\imath}-\overline{f}^{\, \imath}|
\|\bar{a}_{\imath}\|$ $+$ $\sum_{i =
1}^{q+m}\sum_{\ell=1}^{q}|g_{i}^{\ell}-\overline{g}_{i}^{\, \ell}|
\|\bar{a}_{\ell}\|$ $+$ $\sum_{k =
1}^{n}\sum_{\jmath=1}^{n}|e_{k}^{\jmath}-\overline{e}_{k}^{\,
\jmath}| \|\bar{b}_{\jmath}\|$

$+$ $\sum_{\imath=1}^{q}|(g')^{\imath}-(\overline{g}')^{\imath}|
\|\bar{a}_{\imath}\|$ $+$ $\sum_{i =
1}^{q+m}\sum_{\ell=1}^{q}|(g_{i}')^{\ell}-(\overline{g}_{i}')^{\ell}|
\|\bar{a}_{\ell}\|$ $+$ $\sum_{k =
1}^{n}\sum_{\jmath=1}^{n}|(e_{k}')^{\jmath}-(\overline{e}_{k}')^{\jmath}|
\|\bar{b}_{\jmath}\|$.

Take $\mu$ $=$ $L(\varepsilon)$ $=$
$\widetilde{L}(\mathcal{G},(\widetilde{s}_{0},\widetilde{\phi}_{0}'),(s,\phi'),\varepsilon,\Gamma
\times \Gamma,\Sigma)$, where we use
(\ref{e_A2_17_TEXT_mu_bar_gen}). We obtain $\mu$ $=$
$L(\varepsilon)$ $=$
$\widetilde{L}(\widetilde{\mathcal{G}}_{3},(\widetilde{s}_{0},\widetilde{\phi}_{0}'),(s,\phi'),\varepsilon,\Gamma
\times \Gamma,\Sigma)$.

Define $L_{\Psi}(\varepsilon)$ $=$
$\widetilde{L}(\Psi,\widetilde{x}_{0},x,\varepsilon,X,Y)$,
$L_{S}(\varepsilon)$ $=$
$\widetilde{L}(S,\widetilde{s}_{0},s,\varepsilon,\Gamma,\Sigma)$
and $L_{S_{0}}(\varepsilon)$ $=$
$\widetilde{L}(S_{0},\widetilde{s}_{0},s,\varepsilon,\Gamma,\Sigma)$.
We have
\begin{equation}
\label{abc_e_A2_17_TEXT_mu_LS}
   L_{S}(\varepsilon)
   =
   \sup_{(s,\phi') \in \mathbb{B}_{\varepsilon}(\widetilde{s}_{0},\widetilde{\phi}_{0}'), \, \phi' = 0}
   \|DS(\widetilde{s}_{0})
   -DS(s)\|_{L(\Gamma,\Sigma)} \ .
\end{equation}

\begin{lem}
\label{Lema_spatii_infinit_dimensionale_Lh}
\begin{equation}
\label{e_Lema_spatii_infinit_dimensionale_Lh}
   \frac{1}{2}L_{S}(\varepsilon) \leq L(\varepsilon) \ .
\end{equation}
\end{lem}

\begin{proof}

See Appendix \ref{sectiunea_evaluarea lui_mu_L_h}. Let us use
(\ref{appendix_e_A2_17_TEXT_mu}). We have
$\widetilde{\phi}_{0}'=0$. Take $\phi'=0$. We have

\begin{equation*}
\label{abc_e5_57_forma2_sistem_REG_2_GGGvarianta_infinit_dimensionale_DEM_dif_tilde}
   \frac{1}{2}\|DS(\widetilde{s}_{0})
   -DS(s)\|
\end{equation*}
\begin{equation*}
\label{abc_e5_57_forma2_sistem_REG_2_GGGvarianta_infinit_dimensionale_DEM_dif_tilde}
   = \sup_{\|(\overline{s},\overline{\phi}')\| \leq 1, \, \overline{\phi}' = 0}
   \|D\mathcal{G}(\widetilde{s}_{0},\widetilde{\phi}_{0}')(\overline{s},\overline{\phi}')
   -D\mathcal{G}(s,0)(\overline{s},\overline{\phi}')\|
\end{equation*}
\begin{equation*}
\label{abc_e5_57_forma2_sistem_REG_2_GGGvarianta_infinit_dimensionale_DEM_dif_tilde}
   \leq \sup_{\|(\overline{s},\overline{\phi}')\| \leq 1}
   \|D\mathcal{G}(\widetilde{s}_{0},\widetilde{\phi}_{0}')(\overline{s},\overline{\phi}')
   -D\mathcal{G}(s,0)(\overline{s},\overline{\phi}')\|
\end{equation*}
\begin{equation*}
\label{abc_appendix_e_A2_17_TEXT_mu}
   =\|D\mathcal{G}(\widetilde{s}_{0},\widetilde{\phi}_{0}')
   -D\mathcal{G}(s,0)\| \ .
\end{equation*}

\begin{equation*}
\label{abc_e_A2_17_TEXT_mu}
   \frac{1}{2}L_{S}(\varepsilon)
   \leq \sup_{(s,\phi') \in \mathbb{B}_{\varepsilon}(\widetilde{s}_{0},\widetilde{\phi}_{0}'), \, \phi' = 0}
   \|
   D\mathcal{G}(\widetilde{s}_{0},\widetilde{\phi}_{0}') - D\mathcal{G}(s,0)
   \|_{L(\Gamma \times \Gamma,\Sigma)}
\end{equation*}
\begin{equation*}
\label{abc_e_A2_17_TEXT_mu}
   \leq \sup_{(s,\phi') \in \mathbb{B}_{\varepsilon}(\widetilde{s}_{0},\widetilde{\phi}_{0}')}
   \|
   D\mathcal{G}(\widetilde{s}_{0},\widetilde{\phi}_{0}') - D\mathcal{G}(s,\phi')
   \|_{L(\Gamma \times \Gamma,\Sigma)}
   = \mu = L(\varepsilon) \ .
\end{equation*}

\qquad
\end{proof}

\subsection{The theorem on the existence of a bifurcation problem}
\label{sectiunea_01_O_formulare_pe_spatii_infinit_dimensionale_3}

\begin{thm}
\label{teorema_principala_spatii_infinit_dimensionale_widetilde_s_3_0_exact_inf}

Let $W$ and $Z$ be real Banach spaces. They are both
infinite-dimensional spaces or they are both finite-dimensional
spaces with $dim \ W$ $=$ $dim \ Z$. Let $m \geq 1$, $p \geq 2$.
Let $F:\mathbb{R}^{m} \times W \rightarrow Z$ be a nonlinear
function of class $C^{p}$. Let $q \geq 1$, $n \geq 1$,
$\bar{a}_{1},\ldots,\bar{a}_{q}$, $\bar{b}_{1},\ldots,\bar{b}_{n}$
$\in$ $Z \backslash \{ 0 \}$. Let us consider $G$, $H$ defined in
(\ref{e5_2}), (\ref{e5_2_H}) and $B$ $\in$
$L(X,\mathbb{R}^{q+m})$, $\bar{\mathcal{B}}$ $\in$
$L(\Delta,\mathbb{R}^{n})$. Let us take the points
$\widetilde{s}_{0}$ and $\widetilde{\phi}_{0}'$ introduced above.
Let $S$ and $\Phi$ be the functions defined in (\ref{e5_57}),
where $\theta_{0}$ is replaced by $\widetilde{\theta}_{0}$ $=$
$B(\widetilde{x}_{0})$, and (\ref{e5_57_Phi_3}), respectively. Let
us consider the above definitions
(\ref{e5_57_forma2_sistem_REG_2_GGGvarianta_infinit_dimensionale_DEM})
and
(\ref{e5_57_forma2_sistem_REG_2_GGGmic_infinit_dimensionale_DEM})
of $\mathcal{G}$ and $\mathcal{Q}$ and the related entities.

Assume that for some $\varepsilon$ and $\alpha$, $\varepsilon
> 0$ and $0 < \alpha < 1$, $\alpha \neq
\frac{1}{2}$, $\alpha$ arbitrarily small, we have
\begin{equation}
\label{dem_e_A2_37_TEXT_G_prod_alphaastastast_var_exact}
   2 \kappa L(\varepsilon) + 2 \kappa \alpha \widehat{a} < 1 \, ,
\end{equation}
where $\kappa \geq reg \, \mathcal{A}$ and $\widehat{a}$ $=$ $\max
\{$ $\|\bar{a}_{1}\|,\ldots,\|\bar{a}_{q}\|$,
$\|\bar{b}_{1}\|,\ldots,\|\bar{b}_{n}\|$ $\}$.

Let $c = \kappa^{-1} - L(\varepsilon)$ and $M$ $>$ $\|
D\mathcal{G}(\widetilde{s}_{0},\widetilde{\phi}_{0}') \|_{L(\Gamma
\times \Gamma,\Sigma)}$.

Then, there exist positive constants $a^{\ast}$ and $b^{\ast}$ $=$
$c a^{\ast}$ such that
\begin{equation}
\label{e5_57_consecinta_th_Graves_dem}
   d((s,\phi'),\mathcal{G}^{-1}(\zeta)) \leq \frac{\kappa}{1-\kappa \mu} \|\zeta-\mathcal{G}(s,\phi')\|
   \, ,
\end{equation}
for $((s,\phi'),\zeta)$ $\in$
$\mathbb{B}_{a^{\ast}}(\widetilde{s}_{0},\widetilde{\phi}_{0}')$
$\times$
$\mathbb{B}_{b^{\ast}}(\mathcal{G}(\widetilde{s}_{0},\widetilde{\phi}_{0}'))$.

Let $\delta$ $=$ $\|
\mathcal{G}(\widetilde{s}_{0},\widetilde{\phi}_{0}') \|_{\Sigma}$
$=$ $\| S(\widetilde{s}_{0}) \|_{\Sigma}$. Assume
\begin{equation}
\label{dem_e_A2_38_TEXT_G_estimare_M_a_b_h_conditie_inf_th_exact}
   \delta <
   \frac{1}{2} c a^{\ast} \, .
\end{equation}

Assume that $D\Psi(\widetilde{x}_{0})$ is an isomorphism of $X$
onto $Y$ and $DS(\widetilde{s}_{0})$ is an isomorphism of $\Gamma$
onto $\Sigma$.

Then, there exists a solution $(\hat{s}_{0},\hat{\phi}_{0}')$
$\in$
$\mathbb{B}_{a^{\ast}}(\widetilde{s}_{0},\widetilde{\phi}_{0}')$
of the equation
\begin{equation}
\label{e5_57_forma2_sistem_REG_2_GGGvarianta_infinit_dimensionale_DEM_pc_fix_cont_th_princ}
   S(s)
   - \Phi(x,\phi')
   \ni 0 \, ,
\end{equation}
where $\hat{s}_{0}$ $=$
$(\hat{x}_{0},\hat{y}_{1,0},\ldots,\hat{y}_{q+m,0},\hat{\mathrm{z}}_{1,0},\ldots,\hat{\mathrm{z}}_{n,0})$,
\\
$\hat{x}_{0}$ $=$ $(\hat{f}_{0},\hat{\lambda}_{0},\hat{u}_{0})$,
$\hat{y}_{i,0}$ $=$
$(\hat{g}_{i,0},\hat{\mu}_{i,0},\hat{w}_{i,0})$,
$\hat{\mathrm{z}}_{k,0}$ $=$ $(\hat{e}_{k,0},\hat{v}_{k,0})$,
\\
$\hat{f}_{0}=0$, $\hat{g}_{i,0}=0$, $\hat{e}_{k,0}=0$,
$i=1,\ldots,q+m$, $k=1,\ldots,n$. \\
$\hat{\phi}_{0}'$ $=$
$(\hat{y}_{0}',\hat{y}_{1,0}',\ldots,\hat{y}_{q+m,0}',\hat{\mathrm{z}}_{1,0}',\ldots,\hat{\mathrm{z}}_{n,0}')$,
\\
$\hat{y}_{0}'$ $=$ $(\hat{g}_{0}',\hat{\mu}_{0}',\hat{w}_{0}')$,
$\hat{y}_{i,0}'$ $=$
$(\hat{g}_{i,0}',\hat{\mu}_{i,0}',\hat{w}_{i,0}')$,
$\hat{\mathrm{z}}_{k,0}'$ $=$ $(\hat{e}_{k,0}',\hat{v}_{k,0}')$
\\
and $\hat{g}_{0}'=0$, $\hat{g}_{i,0}'$ $=$ $0$, $\hat{e}_{k,0}'$
$=$ $0$, $i=1,\ldots,q+m$, $k=1,\ldots,n$.

Let $s_{0}$ $=$ $(x_{0}, y_{1,0}, \ldots, y_{q+m,0},
\mathrm{z}_{1,0}, \ldots, \mathrm{z}_{n,0})$ with
$(f_{0},\lambda_{0},u_{0})$ $=$ $x_{0}$ $=$ $\hat{x}_{0}$,
$y_{i,0}$ $=$ $\hat{y}_{i,0}-\hat{y}_{i,0}'$, $\mathrm{z}_{k,0}$
$=$ $\hat{\mathrm{z}}_{k,0}-\hat{\mathrm{z}}_{k,0}'$ and
$i=1,\ldots,q+m$, $k=1,\ldots,n$. $s_{0}$ $\in$
$\mathbb{B}_{a^{\ast}}(\widetilde{s}_{0})$.

Let us fix $\hat{x}_{0}$ and $\hat{y}_{0}'$ (whose existence is
demonstrated) in $\Phi_{G}(x,y')$ from the first line of
$\Phi(x,\phi')$ in
(\ref{e5_57_forma2_sistem_REG_2_GGGvarianta_infinit_dimensionale_DEM_pc_fix_cont_th_princ}).
Let us take $\theta_{0}$ $=$
$\widetilde{\theta}_{0}+B(\hat{y}_{0}')$ and $\varrho$ $=$
$DF(\hat{\lambda}_{0},\hat{u}_{0})(\hat{\mu}_{0}',\hat{w}_{0}')$
and consider them as constants. Let us denote by $S_{0}$ the
function $S$ from (\ref{e5_57}) formulated for this $\theta_{0}$
and for $F(\lambda,u)-\varrho$ instead of $F(\lambda,u)$.

Then, $s_{0}$ is the solution of the equation
\begin{equation}
\label{e5_56_b}
   S_{0}(s)= 0 \ .
\end{equation}
Equation (\ref{e5_56_b}) is of the form of equation (\ref{e5_56}).

Then, the component $(\lambda_{0},u_{0})$ of $s_{0}$ is a solution
of the equation
\begin{equation}
\label{e5_1_sol_widetilde_x_0h_Inv_Fc_Th_ec_DATA_introd_exact}
      F(\lambda,u) - \varrho = 0 \, ,
\end{equation}
$(\lambda_{0},u_{0})$ $\in$
$\mathbb{B}_{a^{\ast}}(\widetilde{\lambda}_{0},\widetilde{u}_{0})$.

Let us replace $\kappa$ by $\gamma$ in
(\ref{dem_e_A2_37_TEXT_G_prod_alphaastastast_var_exact}) and,
instead of
(\ref{dem_e_A2_37_TEXT_G_prod_alphaastastast_var_exact}), consider
\begin{equation}
\label{dem_e_A2_37_TEXT_G_prod_alphaastastast_var_exact_gamma}
   2 \gamma L(\varepsilon) + 2 \kappa \alpha \widehat{a} < 1 \, .
\end{equation}

Then, $s_{0}$ is the unique solution of the equation
(\ref{e5_56_b}) in $\mathbb{B}_{a}(\widetilde{s}_{0})$ for every
$a$ $\geq$ $a^{\ast}$ that satisfies $\gamma L_{S}(a) < 1$. The
system (\ref{e5_56_b}) and its solution $s_{0}$ verify the
assertions (a) and (b) of the statement (ii) of Theorem
\ref{teorema_principala_parteaMAIN}.

Then, the component $(\lambda_{0},u_{0})$ of $s_{0}$ is the unique
solution of the equation
(\ref{e5_1_sol_widetilde_x_0h_Inv_Fc_Th_ec_DATA_introd_exact})
that satisfies hypothesis (\ref{ipotezaHypF}) and the rest of the
hypotheses of the statement (i) of Theorem
\ref{teorema_principala_parteaMAIN} in
$\mathbb{B}_{a}(\widetilde{\lambda}_{0},\widetilde{u}_{0})$ for
every $a$ $\geq$ $a^{\ast}$ that satisfies $\gamma L_{S}(a) < 1$.
The solution $(\lambda_{0},u_{0})$ is a bifurcation point of
problem
(\ref{e5_1_sol_widetilde_x_0h_Inv_Fc_Th_ec_DATA_introd_exact}).

We have
\begin{equation}
\label{e5_1_sol_widetilde_x_0h_Inv_Fc_Th_ec_DATA_introd_exact_varrho}
      \| \varrho \|
         \leq \frac{a^{\ast}}{2 \gamma}
         + \| S(\widetilde{s}_{0}) \|
         < (\frac{1}{\gamma} - \frac{1}{2}L(\varepsilon))a^{\ast} \, ,
\end{equation}
\begin{equation}
\label{e5_1_sol_widetilde_x_0h_Inv_Fc_Th_ec_DATA_introd_exact_varrho_var}
      \| \varrho \|
         \leq
         (\| DF(\widetilde{\lambda}_{0},\widetilde{u}_{0}) \|
         + L_{F}(a^{\ast})) a^{\ast} \, ,
\end{equation}
\begin{equation}
\label{e5_49_th}
   \| s_{0} - r \|_{\Gamma} \leq
      [\gamma /(1-\gamma L_{S}(a)] \cdot
      \| S_{0}(r) \|_{\Sigma} \, ,
   \forall r \in
      \mathbb{B}_{a}(\widetilde{s}_{0}) \, .
\end{equation}
\end{thm}

The proof of Theorem
\ref{teorema_principala_spatii_infinit_dimensionale_widetilde_s_3_0_exact_inf}
is given after some remarks.

The following conditions
(\ref{dem_e_A2_37_TEXT_G_prod_alphaastastast_var_exact_fara_alfa2})
and
(\ref{dem_e_A2_37_TEXT_G_prod_alphaastastast_var_exact_fara_alfa2_1})
do not depend on $\alpha$.
\begin{cor}
\label{teorema_principala_spatii_infinit_dimensionale_widetilde_s_3_0_exact_cor1}
(i) Let $\alpha_{0}$ be an arbitrarily fixed positive number, $0 <
\alpha_{0} < 0.1$. Let us take
\begin{equation*}
\label{e_A2_17_TEXT_reg_alfa_kapa_exact_alpha_0}
   \kappa = \sup_{0 < \alpha \leq \alpha_{0}} reg \, \mathcal{A}(\alpha)
   \, , \
   M > \sup_{0 < \alpha \leq \alpha_{0}} \|
      D\mathcal{G}(\widetilde{s}_{0},\widetilde{y}_{0}') \|
      = \sup_{0 < \alpha \leq \alpha_{0}} \| \mathcal{A}(\alpha) \| \, .
\end{equation*}
Instead of
(\ref{dem_e_A2_37_TEXT_G_prod_alphaastastast_var_exact}), we can
assume that for some $\varepsilon$, $\varepsilon > 0$, we have
\begin{equation}
\label{dem_e_A2_37_TEXT_G_prod_alphaastastast_var_exact_fara_alfa2}
   2 \kappa L(\varepsilon) < 1 \, .
\end{equation}

(ii) Instead of
(\ref{dem_e_A2_37_TEXT_G_prod_alphaastastast_var_exact}), we can
assume that for some $\varepsilon$, $\varepsilon > 0$, we have
\begin{equation}
\label{dem_e_A2_37_TEXT_G_prod_alphaastastast_var_exact_fara_alfa2_1}
   2 \gamma L(\varepsilon) < 1 \, .
\end{equation}
\end{cor}

\begin{proof} Let us observe that $L(\varepsilon)$ do not depend on $\alpha$. We have
(\ref{dem_e_A2_37_TEXT_G_prod_alphaastastast_var_exact_fara_alfa2}).
There exists  $0 < \alpha < \alpha_{0}$ ($0 < \alpha < 1$, $\alpha
\neq \frac{1}{2}$), such that
(\ref{dem_e_A2_37_TEXT_G_prod_alphaastastast_var_exact}) holds.

\qquad
\end{proof}

\subsection{Preliminaries related to the metric regularity property}
\label{sectiunea_01_O_formulare_pe_spatii_infinit_dimensionale_4}

We now formulate some estimates of the radii of some balls
included in the neighborhoods $U$ and $V$ from Corollary
\ref{teorema_principala_spatii_infinit_dimensionale_cor_th_Graves}
related to the metric regularity property of the function $f$ at
$\bar{x}$ for $\bar{y}$.

\begin{lem}
\label{Lema_spatii_infinit_dimensionale_th_Graves_estimare_raze_vecinatati}
Let us retain the hypotheses of Graves' theorem
\ref{teorema_principala_spatii_infinit_dimensionale_th_Graves}.

Let $\widetilde{a}$ be such that $0<\widetilde{a}<\varepsilon$ and
$int \mathbb{B}_{\widetilde{a}}(\bar{x})$ is the maximal (for the
inclusion) open ball, with center $\bar{x}$ and radius
$\widetilde{a}$, contained in $U$.

Let $a$ be fixed close to $\widetilde{a}$, $0<a<\widetilde{a}$.

Let us fix $b$, $0 < b \leq c \tau$.

Let us take $a^{\ast}$ $=$ $\min \{ a, \frac{b}{c} \}$ and
$b^{\ast}$ $=$ $c a^{\ast}$. We denote $U^{\ast}$ $=$
$\mathbb{B}_{a^{\ast}}(\bar{x})$, $V^{\ast}$ $=$
$\mathbb{B}_{b^{\ast}}(f(\bar{x}))$.

The relation (\ref{e5_57_consecinta_th_Graves_dem_carte}) holds
for all $(x,y)$ $\in$ $U^{\ast} \times V^{\ast}$.
\end{lem}

\begin{proof}

From the definition of $U$, we deduce that $U$ $\subset$
$\mathbb{B}_{\varepsilon}(\bar{x})$, so
$0<\widetilde{a}<\varepsilon$.

From the definition of $V$, there results that we can choose $V$
of the form $\mathbb{B}_{c \tau}(f(\bar{x}))$, $int \mathbb{B}_{c
\tau}(f(\overline{x}))$ or $\mathbb{B}_{b}(f(\overline{x}))$, with
$0 < b < c \tau$.

The value $\frac{b}{c}$ is taken related to the conclusion of
Graves' theorem (Theorem
\ref{teorema_principala_spatii_infinit_dimensionale_th_Graves}),
to the relation (\ref{e5_57_consecinta_th_Graves_revisited}) in
Theorem
\ref{teorema_principala_spatii_infinit_dimensionale_th_Graves_revisited}
and  to linear openess
(\cite{CLBichir_bib_Dontchev_Rockafellar2009}).

For any $U^{\ast}$ and $V^{\ast}$ such that $U^{\ast} \subseteq
U$, $V^{\ast} \subseteq V$, we have
(\ref{e5_57_consecinta_th_Graves_dem_carte}) for $(x,y)$ $\in$
$U^{\ast} \times V^{\ast}$.

\qquad
\end{proof}

Consider $L(\varepsilon)$ as in
\cite{CLBichir_bib_CalozRappaz1997, CLBichir_bib_Cr_Ra1990,
CLBichir_bib_Gir_Rav1986} but with another radius.

\begin{lem}
\label{Lema_spatii_infinit_dimensionale_th_Graves_estimare_raze_vecinatati_ast}
Let us retain the hypotheses of Graves' theorem
\ref{teorema_principala_spatii_infinit_dimensionale_th_Graves}.
Take $\mu$ $=$ $L(\varepsilon)$ $=$
$\widetilde{L}(f,\bar{x},x,\varepsilon,X,Y)$ (where we use
(\ref{e_A2_17_TEXT_mu_bar_gen})). Let $M$ $>$ $\|Df(\bar{x})\|$.
Let us consider the neighborhoods $U$ and $V$ from Corollary
\ref{teorema_principala_spatii_infinit_dimensionale_cor_th_Graves}.

(i) Relation (\ref{e5_57_consecinta_th_Graves_dem_carte}) is
satisfied for any $(x,y) \in \mathbb{B}_{a}(\bar{x}) \times V$,
where $\mathbb{B}_{a}(\bar{x})$ $\subset$ $U$ and $V$ $\supseteq$
$\mathbb{B}_{b}(\bar{y})$, with $a$ $\geq$ $a^{\ast}$ $>$
$a_{h}^{\ast}$ and $b$ $\geq$ $b^{\ast}$ $>$ $b_{h}^{\ast}$.
$0<a<\varepsilon$. $\tau$, $a^{\ast}$, $b^{\ast}$ are given by
\begin{equation}
\label{dem_e_A2_38_TEXT_G_estimare_exact_tau_DOI_inf}
   \tau < \frac{L(\varepsilon) + M}{L(\varepsilon) + M + c} \cdot \varepsilon \, ,
\end{equation}
(We can impose even $\leq$ because of the increase with $M$.)
\begin{equation}
\label{dem_e_A2_38_TEXT_G_estimare_exact_a_h_ast_h_DOI_inf_th_exact}
   a_{U}^{\ast} = \frac{c \tau}{L(\varepsilon) + M} \, ,
   \quad
   b_{V}^{\ast} = c \tau \, ,
\end{equation}
\begin{equation}
\label{dem_e_A2_38_TEXT_G_estimare_exact_a_h_ast_h_DOI_inf_min}
   a^{\ast} = \min \{ a_{U}^{\ast}, \frac{b_{V}^{\ast}}{c} \} \, ,
   \quad
   b^{\ast} = c a^{\ast} \, .
\end{equation}

(ii) Assume that $2 \kappa L(\varepsilon) < 1$. Relation
(\ref{e5_57_consecinta_th_Graves_dem_carte}) holds for any $(x,y)
\in \mathbb{B}_{a_{h}^{\ast}}(\bar{x}) \times V$, where
$\mathbb{B}_{a_{h}^{\ast}}(\bar{x})$ $\subset$ $U$ and $V$ $=$
$\mathbb{B}_{b_{h}^{\ast}}(\bar{y})$. $\tau$, $a_{h}^{\ast}$,
$b_{h}^{\ast}$ are given by
\begin{equation}
\label{dem_e_A2_38_TEXT_G_estimare_exact_tau_DOI}
   \tau < \frac{1 + 2 \kappa M}{2 + 2 \kappa M} \cdot \varepsilon \, ,
\end{equation}
\begin{equation}
\label{dem_e_A2_38_TEXT_G_estimare_exact_a_h_ast_h_DOI_inf_th_exact_h}
   a_{U}^{\ast} = \frac{\tau}{1 + 2 \kappa M} \, ,
   \quad
   b_{V}^{\ast} = \frac{\tau}{2 \kappa} < c \tau \, ,
\end{equation}
\begin{equation}
\label{dem_e_A2_38_TEXT_G_estimare_exact_a_h_ast_h_DOI_inf_min_h}
   a_{h}^{\ast} = \min \{ a_{U}^{\ast}, \kappa b_{V}^{\ast} \} \, ,
   \quad
   b_{h}^{\ast} = c a_{h}^{\ast} \, .
\end{equation}
\end{lem}

\begin{proof}

\textbf{Proof of (i)}

We consider $b_{V}^{\ast}$ from
(\ref{dem_e_A2_38_TEXT_G_estimare_exact_a_h_ast_h_DOI_inf_th_exact})
and $V$ $=$ $\mathbb{B}_{b_{V}^{\ast}}(\bar{y})$. We seek
$a_{U}^{\ast}$ such that

$\| f(x) - f(\bar{x}) \|$ $\leq$ $\| f(x) - f(\bar{x}) -
Df(\bar{x})(x - \bar{x}) \|$ $+$ $\| Df(\bar{x})(x - \bar{x}) \|$
$\leq$ $L(\varepsilon) \|x - \bar{x}\|$ $+$ $\|Df(\bar{x})\| \|x -
\bar{x}\|$ $\leq$ $(L(\varepsilon) + \|Df(\bar{x})\|)a_{U}^{\ast}$
$\leq$ $(L(\varepsilon) + M)a_{U}^{\ast}$ and we impose the
condition for $a_{U}^{\ast}$
\begin{equation}
\label{dem_e_A2_38_TEXT_G_estimare_exact_a_h_ast_h_DOI_conditie_inf}
   (L(\varepsilon) + M)a_{U}^{\ast}
   \leq c \tau \, .
\end{equation}

We take $a_{U}^{\ast}$ from
(\ref{dem_e_A2_38_TEXT_G_estimare_exact_a_h_ast_h_DOI_inf_th_exact}).
We have $\mathbb{B}_{a_{U}^{\ast}}(\bar{x})$ $\subset$ $U$.

We must verify the condition $\mathbb{B}_{\tau}(x)$ $\subset$
$\mathbb{B}_{\varepsilon}(\bar{x})$ for all $x \in U$.

If  $U$ $=$ $\mathbb{B}_{a_{U}^{\ast}}(\bar{x})$, we have: let
$x'$ $\in$ $\mathbb{B}_{\tau}(x)$. $\| x' - \bar{x} \|$ $\leq$ $\|
x' - x \|$ $+$ $\| x - \bar{x} \|$ $\leq$ $\tau$ $+$
$a_{U}^{\ast}$ $=$ $\tau$ $+$ $\frac{c \tau}{L(\varepsilon) + M}$
$<$ $\varepsilon$. So the condition
(\ref{dem_e_A2_38_TEXT_G_estimare_exact_tau_DOI_inf}) for $\tau$.

Instead of $a_{U}^{\ast}$ and
$\mathbb{B}_{a_{U}^{\ast}}(\bar{x})$, let us consider the value
$\frac{b_{V}^{\ast}}{c}$ and
$\mathbb{B}_{\frac{b_{V}^{\ast}}{c}}(\bar{x})$ related to the
conclusion of Graves' theorem (Theorem
\ref{teorema_principala_spatii_infinit_dimensionale_th_Graves}),
to the relation (\ref{e5_57_consecinta_th_Graves_revisited}) in
Theorem
\ref{teorema_principala_spatii_infinit_dimensionale_th_Graves_revisited}
and to linear openess
(\cite{CLBichir_bib_Dontchev_Rockafellar2009}).

We take $a^{\ast}$ from
(\ref{dem_e_A2_38_TEXT_G_estimare_exact_a_h_ast_h_DOI_inf_min}).
We have $\mathbb{B}_{a^{\ast}}(\bar{x})$ $\subset$ $U$.

If $a^{\ast}$ $=$ $\frac{b_{V}^{\ast}}{c}$, it is not necessary to
modify $\tau$ from
(\ref{dem_e_A2_38_TEXT_G_estimare_exact_tau_DOI_inf}).

From the definition of $U$, we deduce that $U$ $\subset$
$\mathbb{B}_{\varepsilon}(\bar{x})$, so $0<a^{\ast}<\varepsilon$.

Linear openness gives us $\mathbb{B}_{b^{\ast}}(\bar{y})$ $=$ $[
f(\bar{x})+c a^{\ast} int \mathbb{B} ] \cap V$ (see also Theorem
\ref{teorema_principala_spatii_infinit_dimensionale_th_Graves_revisited}).
Hence
(\ref{dem_e_A2_38_TEXT_G_estimare_exact_a_h_ast_h_DOI_inf_min})
for $b^{\ast}$.

\textbf{Proof of (ii)}

We have $2 \kappa L(\varepsilon) < 1$, so
\begin{equation}
\label{dem_e_A2_37_TEXT_G_prod_estimare_lemma}
   c = \frac{1-\kappa L(\varepsilon)}{\kappa} > \frac{1-\frac{1}{2}}{\kappa}
   = \frac{1}{2 \kappa} \, .
\end{equation}
or $\frac{1}{c}$ $<$ $2 \kappa$. We also have $\kappa$ $<$
$\frac{1}{c}$ (\cite{CLBichir_bib_Dontchev_Rockafellar2009}).

We also consider the following estimates:

Using
(\ref{dem_e_A2_38_TEXT_G_estimare_exact_a_h_ast_h_DOI_inf_th_exact}),
we take $b_{V}^{\ast}$ from
(\ref{dem_e_A2_38_TEXT_G_estimare_exact_a_h_ast_h_DOI_inf_th_exact_h}).
We take $V$ $=$ $\mathbb{B}_{b_{V}^{\ast}}(\bar{y})$.

The relation
(\ref{dem_e_A2_38_TEXT_G_estimare_exact_a_h_ast_h_DOI_conditie_inf})
gives

$\| f(x) - f(\bar{x}) \|$ $\leq$ $\| f(x) - f(\bar{x}) -
Df(\bar{x})(x - \bar{x}) \|$ $+$ $\| Df(\bar{x})(x - \bar{x}) \|$
$\leq$ $L(\varepsilon) \|x - \bar{x}\|$ $+$ $\|Df(\bar{x})\| \|x -
\bar{x}\|$ $\leq$ $(L(\varepsilon) + \|Df(\bar{x})\|)a_{U}^{\ast}$
$\leq$ $(L(\varepsilon) + M)a_{U}^{\ast}$ and we impose the
condition for $a_{U}^{\ast}$
\begin{equation}
\label{dem_e_A2_38_TEXT_G_estimare_exact_a_h_ast_h_DOI_conditie}
   (L(\varepsilon) + M)a_{U}^{\ast}
   < (\frac{1}{2 \kappa} + M)a_{U}^{\ast}
   \leq \frac{\tau}{2 \kappa}
   < c \tau \, .
\end{equation}

We take $a_{U}^{\ast}$ from
(\ref{dem_e_A2_38_TEXT_G_estimare_exact_a_h_ast_h_DOI_inf_th_exact_h}).
We have $\mathbb{B}_{a_{U}^{\ast}}(\bar{x})$ $\subset$ $U$.

If  $U$ $=$ $\mathbb{B}_{a_{U}^{\ast}}(\bar{x})$, we have: using
(\ref{dem_e_A2_38_TEXT_G_estimare_exact_a_h_ast_h_DOI_inf_th_exact_h}),
we obtain $\tau$ from
(\ref{dem_e_A2_38_TEXT_G_estimare_exact_tau_DOI}).

We have $\kappa$ $<$ $\frac{1}{c}$. In order to avoid the
dependence on the index "$h$" in Theorem
\ref{teorema_principala_spatii_infinit_dimensionale_widetilde_s_3_0_lim},
we take $\kappa$ instead of $\frac{1}{c}$. More precisely, we take
$\kappa b_{V}^{\ast}$ instead of $\frac{b_{V}^{\ast}}{c}$. $\kappa
b_{V}^{\ast}$ do not depend on "$h$". $\kappa b_{V}^{\ast}$ $<$
$\frac{b_{V}^{\ast}}{c}$.

We take $a_{h}^{\ast}$ from
(\ref{dem_e_A2_38_TEXT_G_estimare_exact_a_h_ast_h_DOI_inf_min_h}).
We have $\mathbb{B}_{a_{h}^{\ast}}(\bar{x})$ $\subset$ $U$.

If $a_{h}^{\ast}$ $=$ $\kappa b_{V}^{\ast}$, it is not necessary
to modify $\tau$ from
(\ref{dem_e_A2_38_TEXT_G_estimare_exact_tau_DOI}).

Linear openness (\cite{CLBichir_bib_Dontchev_Rockafellar2009})
gives us $\mathbb{B}_{b_{h}^{\ast}}(\bar{y})$ $=$ $[ f(\bar{x})+c
a_{h}^{\ast} int \mathbb{B} ] \cap V$ (see also Theorem
\ref{teorema_principala_spatii_infinit_dimensionale_th_Graves_revisited}).
Hence
(\ref{dem_e_A2_38_TEXT_G_estimare_exact_a_h_ast_h_DOI_inf_min_h})
for $b_{h}^{\ast}$.

\qquad
\end{proof}

\begin{lem}
\label{Lema_spatii_infinit_dimensionale_th_Graves_estimare_raze_vecinatati_general}

In Theorems
\ref{teorema_principala_spatii_infinit_dimensionale_widetilde_s_3_0_exact_inf},
\ref{teorema_principala_spatii_infinit_dimensionale_widetilde_s_3_0_exact_h} and
\ref{teorema_principala_spatii_infinit_dimensionale_widetilde_s_3_0_lim},
we can consider $\tau$, $a^{\ast}$ and $b^{\ast}$ from the general
case as in Lemmas
\ref{Lema_spatii_infinit_dimensionale_th_Graves_estimare_raze_vecinatati}
and
\ref{Lema_spatii_infinit_dimensionale_th_Graves_estimare_raze_vecinatati_ast}
by replacing $f$ by $\mathcal{G}$. In Theorem
\ref{teorema_principala_spatii_infinit_dimensionale_widetilde_s_3_0_lim},
as in Lemma
\ref{Lema_spatii_infinit_dimensionale_th_Graves_estimare_raze_vecinatati_ast}
(ii), $\varepsilon$, $\tau$ and $a_{h}^{\ast}$ do not depend on
the parameter $h$ (they are constants).

\end{lem}

\subsection{Proof of Theorem
\ref{teorema_principala_spatii_infinit_dimensionale_widetilde_s_3_0_exact_inf}}
\label{sectiunea_01_O_formulare_pe_spatii_infinit_dimensionale_5}

Let us observe first that
(\ref{dem_e_A2_37_TEXT_G_prod_alphaastastast_var_exact}) is
equivalent to
\begin{equation}
\label{dem_e_A2_37_TEXT_G_prod_alphaastastast_var_exact_equivalent}
   \frac{1}{c}(\frac{1}{2} L(\varepsilon)
      + \alpha \widehat{a}) < \frac{1}{2} \, .
\end{equation}

Since $\mathcal{G}$ is surjective (see Lemma
\ref{Lema_mathcal_G_surjective}), $\mathcal{G}^{-1}$ is a
set-valued mapping, so we can apply the framework of Subsection
\ref{sectiunea_1_preliminaries_Dont_Rock}.

\textbf{- (i) The verification of the assumptions of Graves'
theorem
\ref{teorema_principala_spatii_infinit_dimensionale_th_Graves}}

In Theorem
\ref{teorema_principala_spatii_infinit_dimensionale_th_Graves}, in
Corollary
\ref{teorema_principala_spatii_infinit_dimensionale_cor_th_Graves},
in Lemma
\ref{Lema_spatii_infinit_dimensionale_th_Graves_estimare_raze_vecinatati}
and in Lemma
\ref{Lema_spatii_infinit_dimensionale_th_Graves_estimare_raze_vecinatati_ast},
let us replace $X$, $Y$, $f$, $x$, $\bar{x}$, $y$, $\bar{y} =
f(\bar{x})$, $\varepsilon$, $A$, $\kappa$, $\mu$, $c$, $\tau$,
$U$, $V$, $d(\cdot,\cdot)$, $\|\cdot\|$ by $\Gamma \times \Gamma$,
$\Sigma$, $\mathcal{G}$, $(s,\phi')$,
$(\widetilde{s}_{0},\widetilde{\phi}_{0}')$, $\zeta$,
$\widetilde{\zeta}_{0}$ $=$
$\mathcal{G}(\widetilde{s}_{0},\widetilde{\phi}_{0}')$,
$\varepsilon$, $\mathcal{A}$, $\kappa$, $L(\varepsilon)$, $c$,
$\tau$, $U$, $V$, $d(\cdot,\cdot)$, $\|\cdot\|$, respectively.

Here and in Lemmas
\ref{Lema_spatii_infinit_dimensionale_th_Graves_estimare_raze_vecinatati},
\ref{Lema_spatii_infinit_dimensionale_th_Graves_estimare_raze_vecinatati_ast},
we use the same notations $M$, $\tau$, $a^{\ast}$, $b^{\ast}$.

- Condition $A$ $=$ $\mathcal{A}$ is surjective is verified above.

- Condition $\kappa \geq reg \, A$ is verified since we take
$\kappa = reg \, \mathcal{A}$.

- Let us introduce $\mu$ $=$ $L(\varepsilon)$ $=$
$\widetilde{L}(\mathcal{G},(\widetilde{s}_{0},\widetilde{\phi}_{0}'),(s,\phi'),\varepsilon,\Gamma
\times \Gamma,\Sigma)$.

- Condition $\kappa \mu < 1$ results from
(\ref{dem_e_A2_37_TEXT_G_prod_alphaastastast_var_exact}).

- Let us verify condition (\ref{e5_57_conditie_th_Graves_dem}) of
Graves' theorem
\ref{teorema_principala_spatii_infinit_dimensionale_th_Graves}.

We have

$\| \mathcal{G}(s,\phi') -
\mathcal{G}(\overline{s},\overline{\phi}')
       - \mathcal{A}((s,\phi') - (\overline{s},\overline{\phi}'))
       \|$ $=$
       $\psi_{1}((s,\phi'),(\overline{s},\overline{\phi}'))$.

As in the proofs of Theorem 3.1 \cite{CLBichir_bib_Cr_Ra1990},
Theorem IV.3.1 \cite{CLBichir_bib_Gir_Rav1986} and Theorem I.2.1
\cite{CLBichir_bib_CalozRappaz1997}, we obtain:

$\mathcal{G}(s,\phi') - \mathcal{G}(\overline{s},\overline{\phi}')
- \mathcal{A}((s,\phi') - (\overline{s},\overline{\phi}'))$

$=$

$\int_{0}^{1}[D\mathcal{G}((\overline{s},\overline{\phi}')+t((s,\phi')-(\overline{s},\overline{\phi}')))
- D\mathcal{G}(\widetilde{s}_{0},\widetilde{\phi}_{0}')] \cdot
((s,\phi') - (\overline{s},\overline{\phi}')) dt$

Relation (\ref{e5_57_conditie_th_Graves_dem_mathcal_G_DR}) holds
with $\mu$ $=$ $L(\varepsilon)$ $=$
$\widetilde{L}(\mathcal{G},(\widetilde{s}_{0},\widetilde{\phi}_{0}'),(s,\phi'),\varepsilon,\Gamma
\times \Gamma,\Sigma)$.
\begin{equation}
\label{e5_57_conditie_th_Graves_dem_mathcal_G_DR}
   \| \mathcal{G}(s,\phi') - \mathcal{G}(\overline{s},\overline{\phi}')
       - \mathcal{A}((s,\phi') - (\overline{s},\overline{\phi}')) \|
       \leq L(\varepsilon) \| (s,\phi') - (\overline{s},\overline{\phi}') \|
   \, .
\end{equation}

($\mathcal{G}$ is strictly differentiable at
$(\widetilde{s}_{0},\widetilde{\phi}_{0}')$ (see definition, pages
31 and 275, \cite{CLBichir_bib_Dontchev_Rockafellar2009})).

Then, the conclusions and the consequences of Graves' theorem
\ref{teorema_principala_spatii_infinit_dimensionale_th_Graves}
hold for $\mathcal{G}$ and
$(\widetilde{s}_{0},\widetilde{\phi}_{0}')$.

\textbf{- (ii) The formulation of the consequence of Graves'
theorem, Corollary
\ref{teorema_principala_spatii_infinit_dimensionale_cor_th_Graves}}

Let us take a positive $\tau < \varepsilon$. From Corollary
\ref{teorema_principala_spatii_infinit_dimensionale_cor_th_Graves}
for $\mathcal{G}$, there results that there exist a neighborhood
$U$ of $(\widetilde{s}_{0},\widetilde{\phi}_{0}')$ and a
neighborhood $V$ of
$\mathcal{G}(\widetilde{s}_{0},\widetilde{\phi}_{0}')$ such that
(i) $\mathbb{B}_{\tau}(s,\phi')$ $\subset$
$\mathbb{B}_{\varepsilon}(\widetilde{s}_{0},\widetilde{\phi}_{0}')$
for all $(s,\phi')$ $\in$ $U$ and $\| \mathcal{G}(s,\phi')$ $-$
$\mathcal{G}(\widetilde{s}_{0},\widetilde{\phi}_{0}') \|$ $<$ $c
\tau$ for $(s,\phi')$ $\in$ $U$, (ii) $\|\zeta -
\mathcal{G}(s,\phi') \|$ $\leq$ $c \tau$ for $\zeta$ $\in$ $V$,
where $(s,\phi')$ $\in$ $U$, (iii) $\mathcal{G}$ is the metrically
regular at $(\widetilde{s}_{0},\widetilde{\phi}_{0}')$ for
$\mathcal{G}(\widetilde{s}_{0},\widetilde{\phi}_{0}')$ with
constant $\frac{\kappa}{1-\kappa \mu}$, that is,
\begin{equation}
\label{e5_57_consecinta_th_Graves_dem}
   d((s,\phi'),\mathcal{G}^{-1}(\zeta)) \leq \frac{\kappa}{1-\kappa \mu} \|\zeta-\mathcal{G}(s,\phi')\|
   \, ,
\end{equation}
for $((s,\phi'),\zeta)$ $\in$ $U \times V$.

We can consider $\tau$, $a^{\ast}$ and $b^{\ast}$ from the general
case as in Lemma
\ref{Lema_spatii_infinit_dimensionale_th_Graves_estimare_raze_vecinatati}
or from the formulas in Lemma
\ref{Lema_spatii_infinit_dimensionale_th_Graves_estimare_raze_vecinatati_ast}
for $\mathcal{G}$. Hence, we have
(\ref{e5_57_consecinta_th_Graves_dem}). The following discussion
is true for every choice.

In order to fix a choice, we use Lemma
\ref{Lema_spatii_infinit_dimensionale_th_Graves_estimare_raze_vecinatati_ast}
(i).
$\mathbb{B}_{a^{\ast}}(\widetilde{s}_{0},\widetilde{\phi}_{0}')$
$\subset$ $U$,
$\mathbb{B}_{b^{\ast}}(\mathcal{G}(\widetilde{s}_{0},\widetilde{\phi}_{0}'))$
$\subseteq$ $V$.

\textbf{- (iii) The imposition of some conditions}

In order to have $\mathcal{Q}(s,\phi')$ $\in$
$\mathbb{B}_{b^{\ast}}(\mathcal{G}(\widetilde{s}_{0},\widetilde{\phi}_{0}'))$,
for $(s,\phi')$ $\in$
$\mathbb{B}_{a^{\ast}}(\widetilde{s}_{0},\widetilde{\phi}_{0}')$,
we impose the following two conditions
(\ref{dem_e_A2_38_TEXT_G_estimare_inf}) and
(\ref{dem_e_A2_38_TEXT_G_estimare_g_DR}).

We first assure that $0$ $\in$
$\mathbb{B}_{b^{\ast}}(\mathcal{G}(\widetilde{s}_{0},\widetilde{\phi}_{0}'))$
and $\|0-\mathcal{G}(\widetilde{s}_{0},\widetilde{\phi}_{0}')\|$
$<$ $\frac{b^{\ast}}{2}$ by imposing
\begin{equation}
\label{dem_e_A2_38_TEXT_G_estimare_inf}
   \delta < \frac{1}{2} \cdot b^{\ast} \, ,
\end{equation}
in other words, the assumption
(\ref{dem_e_A2_38_TEXT_G_estimare_M_a_b_h_conditie_inf_th_exact})
(since $b^{\ast}$ $=$ $c a^{\ast}$).

For $(s,\phi')$ $\in$
$\mathbb{B}_{a^{\ast}}(\widetilde{s}_{0},\widetilde{\phi}_{0}')$,
we have
\begin{eqnarray}
   && \|\mathcal{Q}(s,\phi')\|
         \leq
         \psi_{1}((s,\phi'),(\widetilde{s}_{0},\widetilde{\phi}_{0}'))
         + \alpha \psi_{2}((s,\phi'),(\widetilde{s}_{0},\widetilde{\phi}_{0}'))
          \label{dem_e_A2_38_TEXT_G_estimare_g_DR_ini} \\
   && \leq (\frac{1}{2} L(a^{\ast}) + \alpha \widehat{a})
         \|(\widetilde{s}_{0},\widetilde{\phi}_{0}')-(s,\phi')\| \, .
         \nonumber
\end{eqnarray}

Let us impose the condition
\begin{equation}
\label{dem_e_A2_38_TEXT_G_estimare_g_DR}
   \|\mathcal{Q}(s,\phi')\| < \frac{1}{2} \cdot b^{\ast} \, , \
      \forall (s,\phi') \in \mathbb{B}_{a^{\ast}}(\widetilde{s}_{0},\widetilde{\phi}_{0}') \, .
\end{equation}

If $\frac{a^{\ast}}{b^{\ast}}(\frac{1}{2} L(a^{\ast}) + \alpha
\widehat{a})$ $<$ $\frac{1}{2}$, then we have
(\ref{dem_e_A2_38_TEXT_G_estimare_g_DR}). We have $b^{\ast}$ $=$
$c a^{\ast}$ and $L(a^{\ast})$ $\leq$ $L(\varepsilon)$. There
results that if
(\ref{dem_e_A2_37_TEXT_G_prod_alphaastastast_var_exact_equivalent})
is satisfied, that is, if $\lambda < \frac{1}{2}$ is satisfied,
where $\lambda$ is defined below, in
(\ref{dem_e_A2_37_TEXT_G_prod_alphaastastast_var_lambda_LA_FEL_SI_CEL_INITIAL}),
then we have (\ref{dem_e_A2_38_TEXT_G_estimare_g_DR}).

Let $z$ $\in$ $\mathbb{B}_{\frac{b^{\ast}}{2}}(0)$. Then,
$\|z-\mathcal{G}(\widetilde{s}_{0},\widetilde{\phi}_{0}')\|$
$\leq$ $\|z-0\|$ $+$
$\|0-\mathcal{G}(\widetilde{s}_{0},\widetilde{\phi}_{0}')\|$
$\leq$ $\frac{b^{\ast}}{2}$ $+$
$\|\mathcal{G}(\widetilde{s}_{0},\widetilde{\phi}_{0}')\|$ $<$
$\frac{b^{\ast}}{2}$ $+$ $\frac{b^{\ast}}{2}$ $=$ $b^{\ast}$, so
$z$ $\in$ $int
\mathbb{B}_{b^{\ast}}(\mathcal{G}(\widetilde{s}_{0},\widetilde{\phi}_{0}'))$.

Taking condition (\ref{dem_e_A2_38_TEXT_G_estimare_g_DR}) and
replacing $z$ $=$ $\mathcal{Q}(s,\phi')$, the previous relation
leads us to
\begin{equation}
\label{dem_e_A2_38_TEXT_G_estimare_g_DR_B_b_h}
   \|\mathcal{Q}(s,\phi')-\mathcal{G}(\widetilde{s}_{0},\widetilde{\phi}_{0}')\|
   < b^{\ast} \, ,
\end{equation}
for $(s,\phi')$ $\in$
$\mathbb{B}_{a^{\ast}}(\widetilde{s}_{0},\widetilde{\phi}_{0}')$.
In other words, $\mathcal{Q}(s,\phi')$ $\in$
$\mathbb{B}_{b^{\ast}}(\mathcal{G}(\widetilde{s}_{0},\widetilde{\phi}_{0}'))$,
for $(s,\phi')$ $\in$
$\mathbb{B}_{a^{\ast}}(\widetilde{s}_{0},\widetilde{\phi}_{0}')$.

Then, (\ref{e5_57_consecinta_th_Graves_dem}) holds for
$((s,\phi'),\zeta)$$\in$
$\mathbb{B}_{a^{\ast}}(\widetilde{s}_{0},\widetilde{\phi}_{0}')$
$\times$
$\mathbb{B}_{b^{\ast}}(\mathcal{G}(\widetilde{s}_{0},\widetilde{\phi}_{0}'))$.

We obtain, by replacing $\zeta$ $=$
$\mathcal{Q}(\bar{s},\bar{\phi}')$ in
(\ref{e5_57_consecinta_th_Graves_dem}),
\begin{equation}
\label{e5_57_consecinta_th_Graves_mathcalG_g}
   d((s,\phi'),\mathcal{G}^{-1}(\mathcal{Q}(\bar{s},\bar{\phi}'))
   \leq
   \frac{\kappa}{1-\kappa \mu} \|\mathcal{Q}(\bar{s},\bar{\phi}')-\mathcal{G}(s,\phi')\|
   \, ,
\end{equation}
for $(s,\phi')$, $(\bar{s},\bar{\phi}')$ $\in$
$\mathbb{B}_{a^{\ast}}(\widetilde{s}_{0},\widetilde{\phi}_{0}')$.

\textbf{- (iv) The definition of the mapping $T_{0}$ for the
contraction mapping principle for set-valued mappings}

Let us define the mappings:

$\zeta$ $\in$
$\mathbb{B}_{b^{\ast}}(\mathcal{G}(\widetilde{s}_{0},\widetilde{\phi}_{0}'))$,
$\Theta : \zeta \mapsto \mathcal{G}^{-1}(\zeta)$,

$(s,\phi')$ $\in$
$\mathbb{B}_{a^{\ast}}(\widetilde{s}_{0},\widetilde{\phi}_{0}')$,
$T_{0} : (s,\phi') \mapsto \Theta(\mathcal{Q}(s,\phi'))$.

Observe that $dom \ \mathcal{G}^{-1}$ $=$ $\Sigma$, where we use
Lemma \ref{Lema_mathcal_G_surjective}. We extend the definition of
these mappings to the spaces $\Gamma$ and $\Sigma$.

$T_{0}:\Gamma \times \Gamma \rightrightarrows \Gamma \times
\Gamma$, $T_{0} : (s,\phi') \mapsto
\mathcal{G}^{-1}(\mathcal{Q}(s,\phi'))$.

\textbf{- (v) The verification of the assumptions of the
contraction mapping principle for set-valued mappings, Theorem
\ref{teorema_contraction_mapping_principle_Dontchev_Rockafellar2009}}

In Theorem
\ref{teorema_contraction_mapping_principle_Dontchev_Rockafellar2009},
let us replace $X$, $\rho$, $T$, $x$, $\bar{x}$, $a$, $\lambda$,
$d(\cdot,\cdot)$, $e(\cdot,\cdot)$, $u$, $v$ by $\Gamma \times
\Gamma$, $\rho$, $T_{0}$, $(s,\phi')$,
$(\widetilde{s}_{0},\widetilde{\phi}_{0}')$, $a^{\ast}$,
$\lambda$, $d(\cdot,\cdot)$, $e(\cdot,\cdot)$, $(s,\phi')$,
$(\overline{s},\overline{\phi}')$ respectively.

\textbf{- The verification of the condition $gph T_{0} \cap
(\mathbb{B}_{a^{\ast}}(\widetilde{s}_{0},\widetilde{\phi}_{0}')
\times
\mathbb{B}_{a^{\ast}}(\widetilde{s}_{0},\widetilde{\phi}_{0}'))$
is closed}

For brevity, only for this verification, we keep some notations
from the general case from Subsection
\ref{sectiunea_1_preliminaries_Dont_Rock}. So, instead of
$a^{\ast}$, $(s,\phi')$,
$(\widetilde{s}_{0},\widetilde{\phi}_{0}')$, $\zeta$, $\Gamma
\times \Gamma$, $\Sigma$, we keep $a$, $x$, $\bar{x}$, $y$, $X$,
$Y$ respectively. Let us verify that $gph T_{0} \cap
(\mathbb{B}_{a}(\bar{x}) \times \mathbb{B}_{a}(\bar{x}))$ is
closed.

The graph of the continuous mapping $\mathcal{G}$, $gph \
\mathcal{G}$ $=$ $\{ (x,y) \in X \times Y | y = \mathcal{G}(x) \}$
$=$ $\{ (x,y) \in X \times Y | y \in \mathcal{G}(x) \}$ is closed
in $X \times Y$. Here, we identify the function $\mathcal{G}:X
\rightarrow Y$ with the set-valued mapping $\mathcal{G}:X
\rightrightarrows Y$ such that $\mathcal{G}$ is single-valued at
every point of $dom \ \mathcal{G}$.

$\mathcal{G}^{-1}(y)$ $=$ $\{ x \in X | y \in \mathcal{G}(x) \}$.

$gph \ \mathcal{G}^{-1}$ $=$ $\{ (y,x) \in Y \times X | (x,y) \in
gph \ \mathcal{G} \}$.

Hence the graph of $\mathcal{G}^{-1}$, $gph \ \mathcal{G}^{-1}$,
is closed in $Y \times X$.

$Range(\mathcal{Q})$ $=$ $Y$.

$gph \ \mathcal{Q}$ $=$ $\{ (x,y) \in X \times Y | y =
\mathcal{Q}(x) \}$.

$gph \ \mathcal{G}^{-1}$ $=$ $\{ (y,x) \in Range(\mathcal{Q})
\times X | y \in \mathcal{G}(x) \}$.

$gph \ \mathcal{G}^{-1}$ $=$ $\{ (y,x) \in Range(\mathcal{Q})
\times X | (x',y) \in gph \ \mathcal{Q}, \ y \in \mathcal{G}(x)
\}$.

$gph \ \mathcal{G}^{-1}(\mathcal{Q}(\cdot))$ $=$ $\{ (x',x) \in X
\times X | (x',y) \in gph \ \mathcal{Q}, \ y \in \mathcal{G}(x)
\}$.

\textbf{Proof of $gph T_{0} \cap (\mathbb{B}_{a}(\bar{x}) \times
\mathbb{B}_{a}(\bar{x}))$ is closed}

We prove using sequences for the formulation

$gph \ \mathcal{G}^{-1}(\mathcal{Q}(\cdot))$ $=$ $\{ (x',x) \in X
\times X | y=\mathcal{Q}(x'), \ y \in \mathcal{G}(x) \}$.

Let $(x_{\ell}',x_{\ell})$ $\in$ $gph \
\mathcal{G}^{-1}(\mathcal{Q}(\cdot))$, $(x_{\ell}',x_{\ell})$
$\rightarrow$ $(x',x)$ as $\ell$ $\rightarrow$ $\infty$, $(x',x)$
$\in$ $X \times X$.

Define $y_{\ell}=\mathcal{Q}(x_{\ell}')$. From
$(x_{\ell}',x_{\ell})$ $\in$ $gph \
\mathcal{G}^{-1}(\mathcal{Q}(\cdot))$, we have $y_{\ell} \in
\mathcal{G}(x_{\ell})$.

Since $\mathcal{Q}$ is continuous, $x_{\ell}'$ $\rightarrow$ $x'$
implies $y_{\ell}$ $\rightarrow$ $\mathcal{Q}(x')$, so
$(y_{\ell},x_{\ell})$ $\rightarrow$ $(\mathcal{Q}(x'),x)$ as
$\ell$ $\rightarrow$ $\infty$.

$y_{\ell} \in \mathcal{G}(x_{\ell})$ implies that
$(x_{\ell},y_{\ell})$ $\in$ $gph \ \mathcal{G}$. Since $gph \
\mathcal{G}$ is closed in $X \times Y$, there results that
$(x,\mathcal{Q}(x'))$ $\in$ $gph \ \mathcal{G}$ so
$\mathcal{Q}(x') \in \mathcal{G}(x)$.

Let $y=\mathcal{Q}(x')$. We have $y \in \mathcal{G}(x)$. Hence
$(x',x)$ $\in$ $gph \ \mathcal{G}^{-1}(\mathcal{Q}(\cdot))$ and
$gph \ \mathcal{G}^{-1}(\mathcal{Q}(\cdot))$ is closed in $X
\times X$.

We now take the intersection with $\mathbb{B}_{a}(\bar{x}) \times
\mathbb{B}_{a}(\bar{x})$.

There results that $gph T_{0} \cap (\mathbb{B}_{a}(\bar{x}) \times
\mathbb{B}_{a}(\bar{x}))$ is closed, that is, $gph T_{0} \cap
(\mathbb{B}_{a^{\ast}}(\widetilde{s}_{0},\widetilde{\phi}_{0}')
\times
\mathbb{B}_{a^{\ast}}(\widetilde{s}_{0},\widetilde{\phi}_{0}'))$
is closed.

\textbf{- An estimate for
$d((\widetilde{s}_{0},\widetilde{\phi}_{0}'),T_{0}(\widetilde{s}_{0},\widetilde{\phi}_{0}'))$
from the condition I of Theorem
\ref{teorema_contraction_mapping_principle_Dontchev_Rockafellar2009}}

Using (\ref{e5_57_consecinta_th_Graves_mathcalG_g}), we have
$d((\widetilde{s}_{0},\widetilde{\phi}_{0}'),T_{0}(\widetilde{s}_{0},\widetilde{\phi}_{0}'))$
$=$
$d((\widetilde{s}_{0},\widetilde{\phi}_{0}'),\mathcal{G}^{-1}(\mathcal{Q}(\widetilde{s}_{0},\widetilde{\phi}_{0}')))$
\\
$\leq$ $\frac{\kappa}{1-\kappa \mu}
\|\mathcal{Q}(\widetilde{s}_{0},\widetilde{\phi}_{0}')-\mathcal{G}(\widetilde{s}_{0},\widetilde{\phi}_{0}')\|$
$=$ $\frac{\kappa}{1-\kappa \mu} \|S(\widetilde{s}_{0}) \|$ $=$
$\frac{\kappa}{1-\kappa \mu} \delta$ $=$ $\frac{1}{c}\delta$.

\textbf{- An estimate for $e(T_{0}(s,\phi') \cap
\mathbb{B}_{a^{\ast}}(\widetilde{s}_{0},\widetilde{\phi}_{0}'),
T_{0}(\overline{s},\overline{\phi}'))$ from the condition II of
Theorem
\ref{teorema_contraction_mapping_principle_Dontchev_Rockafellar2009}}

We obtain
\begin{equation}
\label{e5_57_conditie_th_dem_mathcal_dif_g_DR_S3}
   \| \mathcal{Q}(s,\phi') - \mathcal{Q}(\overline{s},\overline{\phi}') \|
   \leq
   \psi_{1}((s,\phi'),(\overline{s},\overline{\phi}'))
   + \alpha \psi_{2}((s,\phi'),(\overline{s},\overline{\phi}'))
   \, .
\end{equation}

We have the equivalence $(\hat{s},\hat{\phi}') \in
\mathcal{G}^{-1}(\mathcal{Q}(s,\phi')) \cap
\mathbb{B}_{a^{\ast}}(\widetilde{s}_{0},\widetilde{\phi}_{0}')$ $\Leftrightarrow$ \\
$\mathcal{Q}(s,\phi') \in \mathcal{G}(\hat{s},\hat{\phi}')$ and
$(\hat{s},\hat{\phi}') \in
\mathbb{B}_{a^{\ast}}(\widetilde{s}_{0},\widetilde{\phi}_{0}')$.
We remember that $\mathcal{Q}$ is a function and
$\frac{\kappa}{1-\kappa \mu}$ $=$ $\frac{1}{c}$.

Using (\ref{e5_57_consecinta_th_Graves_mathcalG_g}), we have
$e(T_{0}(s,\phi') \cap
\mathbb{B}_{a^{\ast}}(\widetilde{s}_{0},\widetilde{\phi}_{0}'),
T_{0}(\overline{s},\overline{\phi}'))$ \\
$\leq$ $\sup \{
d((\hat{s},\hat{\phi}'),T_{0}(\overline{s},\overline{\phi}')) |
(\hat{s},\hat{\phi}') \in T_{0}(s,\phi') \cap
\mathbb{B}_{a^{\ast}}(\widetilde{s}_{0},\widetilde{\phi}_{0}') \}$ \\
$=$ $\sup \{
d((\hat{s},\hat{\phi}'),\mathcal{G}^{-1}(\mathcal{Q}(\overline{s},\overline{\phi}')))
| (\hat{s},\hat{\phi}') \in \mathcal{G}^{-1}(\mathcal{Q}(s,\phi'))
\cap
\mathbb{B}_{a^{\ast}}(\widetilde{s}_{0},\widetilde{\phi}_{0}') \}$ \\
$\leq$ $\sup \{ \frac{\kappa}{1-\kappa \mu}
\|\mathcal{Q}(\overline{s},\overline{\phi}')-\mathcal{G}(\hat{s},\hat{\phi}')\|
\; | \; (\hat{s},\hat{\phi}') \in
\mathcal{G}^{-1}(\mathcal{Q}(s,\phi')) \cap
\mathbb{B}_{a^{\ast}}(\widetilde{s}_{0},\widetilde{\phi}_{0}') \}$ \\
$\leq$ $\sup \{ \frac{\kappa}{1-\kappa \mu}
\|\mathcal{Q}(\overline{s},\overline{\phi}')-\mathcal{G}(\hat{s},\hat{\phi}')\|
\; | \; \mathcal{Q}(s,\phi') \in \mathcal{G}(\hat{s},\hat{\phi}'),
(\hat{s},\hat{\phi}') \in  \mathbb{B}_{a^{\ast}}(\widetilde{s}_{0},\widetilde{\phi}_{0}') \}$ \\
$=$ $\frac{\kappa}{1-\kappa \mu}
\|\mathcal{Q}(\overline{s},\overline{\phi}')-\mathcal{Q}(s,\phi')\|$
$\leq$ $\frac{\kappa}{1-\kappa \mu} (\frac{1}{2} L(\varepsilon) +
\alpha \widehat{a})
\|(\overline{s},\overline{\phi}')-(s,\phi')\|$, \\
for all $(s,\phi')$, $(\overline{s},\overline{\phi}')$ $\in$
$\mathbb{B}_{a^{\ast}}(\widetilde{s}_{0},\widetilde{\phi}_{0}')$.

\textbf{- (vi) The formulation of some conditions related to
conditions I and II from Theorem
\ref{teorema_contraction_mapping_principle_Dontchev_Rockafellar2009}}

Condition II from Theorem
\ref{teorema_contraction_mapping_principle_Dontchev_Rockafellar2009}
leads to
\begin{equation}
\label{dem_e_A2_37_TEXT_G_prod_alphaastastast_Conditia doi}
   \frac{1}{c}(\frac{1}{2} L(\varepsilon) + \alpha \widehat{a})
      \leq \lambda \, .
\end{equation}
We take $\lambda$ $=$ $\frac{1}{c}(\frac{1}{2} L(\varepsilon) +
\alpha \widehat{a})$. From
(\ref{dem_e_A2_37_TEXT_G_prod_alphaastastast_var_exact_equivalent}),
we have $0 < \lambda < \frac{1}{2}$. Hence,
\begin{equation}
\label{dem_e_A2_37_TEXT_G_prod_alphaastastast_var_lambda_LA_FEL_SI_CEL_INITIAL}
   \lambda = \frac{1}{c}(\frac{1}{2} L(\varepsilon)
      + \alpha \widehat{a}) < \frac{1}{2} \, .
\end{equation}

Condition I from Theorem
\ref{teorema_contraction_mapping_principle_Dontchev_Rockafellar2009}
leads to
\begin{equation}
\label{dem_e_A2_37_TEXT_G_prod_alphaastastast_Conditia unu}
   \frac{1}{c}\delta < a^{\ast} (1-\lambda)
   \quad \textrm{or} \quad
   \delta < c a^{\ast} (1-\lambda) \, .
\end{equation}

We have $b^{\ast}$ $=$ $c a^{\ast}$ from
(\ref{dem_e_A2_38_TEXT_G_estimare_exact_a_h_ast_h_DOI_inf_min})
and $\frac{1}{2}$ $<$ $1-\lambda$. Taking into account conditions
(\ref{dem_e_A2_38_TEXT_G_estimare_inf}) and
(\ref{dem_e_A2_37_TEXT_G_prod_alphaastastast_Conditia unu}), we
impose
\begin{equation}
\label{dem_e_A2_38_TEXT_G_estimare_M_a_h_conditie_est2_inf_a_si_b_lambda}
   \delta < \min \{ c a^{\ast} (1-\lambda), \frac{1}{2} b^{\ast} \}
   = \frac{1}{2} b^{\ast}
   = \frac{1}{2} c a^{\ast} \, ,
\end{equation}
that is, assumption
(\ref{dem_e_A2_38_TEXT_G_estimare_M_a_b_h_conditie_inf_th_exact}).

With $\lambda$ given by
(\ref{dem_e_A2_37_TEXT_G_prod_alphaastastast_var_lambda_LA_FEL_SI_CEL_INITIAL}),
conditions I and II, from Theorem
\ref{teorema_contraction_mapping_principle_Dontchev_Rockafellar2009},
are verified.

\textbf{- (vii) The formulation of the conclusion of the
contraction mapping principle for set-valued mappings}

The assumptions of Theorem
\ref{teorema_contraction_mapping_principle_Dontchev_Rockafellar2009}
are verified, hence $T_{0}$ has a fixed point in
$\mathbb{B}_{a^{\ast}}(\widetilde{s}_{0},\widetilde{\phi}_{0}')$;
that is, there exists $(\bar{s},\bar{\phi}')$ $\in$
$\mathbb{B}_{a^{\ast}}(\widetilde{s}_{0},\widetilde{\phi}_{0}')$
such that $(\bar{s},\bar{\phi}') \in T_{0}(\bar{s},\bar{\phi}')$.

$$(\bar{s},\bar{\phi}') \in T_{0}(\bar{s},\bar{\phi}')
\Leftrightarrow (\bar{s},\bar{\phi}') \in
\mathcal{G}^{-1}(\mathcal{Q}(\bar{s},\bar{\phi}'))$$

$$\Leftrightarrow \mathcal{Q}(\bar{s},\bar{\phi}') \in
\mathcal{G}(\bar{s},\bar{\phi}') \Leftrightarrow
\mathcal{G}(\bar{s},\bar{\phi}') -
\mathcal{Q}(\bar{s},\bar{\phi}') \ni 0 \Leftrightarrow$$

\begin{equation}
\label{e5_57_forma2_sistem_REG_2_GGGvarianta_infinit_dimensionale_DEM_pc_fix}
   S(\bar{s})
   - \Phi(\bar{x},\bar{\phi}')
   - \Phi(\widetilde{x}_{0},\xi(\bar{f},\bar{g}_{i},\bar{e}_{k}))
   + \Phi(\widetilde{x}_{0},\xi(\bar{g}',\bar{g}_{i}',\bar{e}_{k}'))
   \ni 0 \, ,
\end{equation}
$$\Leftrightarrow$$
\begin{equation}
\label{e5_57_forma2_sistem_REG_2_GGGvarianta_infinit_dimensionale_DEM_pc_fix_cont}
   S(\hat{s}_{0})
   - \Phi(\hat{x}_{0},\hat{\phi}_{0}')
   \ni 0 \, ,
\end{equation}
that is, $(\hat{s}_{0},\hat{\phi}_{0}')$ is a solution  of the
equation
(\ref{e5_57_forma2_sistem_REG_2_GGGvarianta_infinit_dimensionale_DEM_pc_fix_cont_th_princ}),
where \\ $\hat{s}_{0}$ $=$
$(\hat{x}_{0},\hat{y}_{1,0},\ldots,\hat{y}_{q+m,0},\hat{\mathrm{z}}_{1,0},\ldots,\hat{\mathrm{z}}_{n,0})$,
$\hat{\phi}_{0}'$ $=$
$(\hat{y}_{0}',\hat{y}_{1,0}',\ldots,\hat{y}_{q+m,0}',\hat{\mathrm{z}}_{1,0}',\ldots,\hat{\mathrm{z}}_{n,0}')$
and $\hat{x}_{0}$ $=$
$(\hat{f}_{0},\hat{\lambda}_{0},\hat{u}_{0})$ $=$
$(0,\bar{\lambda},\bar{u})$, $\hat{y}_{i,0}$ $=$
$(\hat{g}_{i,0},\hat{\mu}_{i,0},\hat{w}_{i,0})$ $=$
$(0,\bar{\mu}_{i},\bar{w}_{i})$, $\hat{\mathrm{z}}_{k,0}$ $=$
$(\hat{e}_{k,0},\hat{v}_{k,0})$ $=$ $(0,\bar{v}_{k})$,
$\hat{y}_{0}'$ $=$ $(\hat{g}_{0}',\hat{\mu}_{0}',\hat{w}_{0}')$
$=$ $(0,\bar{\mu}',\bar{w}')$, $\hat{y}_{i,0}'$ $=$
$(\hat{g}_{i,0}',\hat{\mu}_{i,0}',\hat{w}_{i,0}')$ $=$
$(0,\bar{\mu}_{i}',\bar{w}_{i}')$, $\hat{\mathrm{z}}_{k,0}'$ $=$
$(\hat{e}_{k,0}',\hat{v}_{k,0}')$ $=$ $(0,\bar{v}_{k}')$, for
$i=1,\ldots,q+m$, $k=1,\ldots,n$.

It also results $(\hat{s}_{0},\hat{\phi}_{0}')$ $\in$
$\mathbb{B}_{a^{\ast}}(\widetilde{s}_{0},\widetilde{\phi}_{0}')$.
Indeed, we have:

$\| (\widetilde{s}_{0},\widetilde{\phi}_{0}') -
(\hat{s}_{0},\hat{\phi}_{0}') \|$ $\leq$ $\|
(\widetilde{s}_{0},\widetilde{\phi}_{0}') -
(\hat{s}_{0},\hat{\phi}_{0}') \|$ $+$ $\| \widetilde{f}_{0} -
\bar{f} \|$ $+$ $\sum_{i = 1}^{q+m}\| \widetilde{g}_{i,0} -
\bar{g}_{i} \|$ $+$ $\sum_{k = 1}^{n}\| \widetilde{e}_{k,0} -
\bar{e}_{k} \|$ $+$ $\| \widetilde{g}_{0}' - \bar{g}' \|$ $+$
$\sum_{i = 1}^{q+m}\| \widetilde{g}_{i,0}' - \bar{g}_{i}' \|$ $+$
$\sum_{k = 1}^{n}\| \widetilde{e}_{k,0}' - \bar{e}_{k}' \|$ $=$
$\| (\widetilde{s}_{0},\widetilde{\phi}_{0}') -
(\bar{s},\bar{\phi}') \|$ $\leq$ $a^{\ast}$.

(\ref{e5_57_forma2_sistem_REG_2_GGGvarianta_infinit_dimensionale_DEM_pc_fix_cont})
becomes
\begin{equation}
\label{e5_57_forma2_sistem_REG_2_GGGvarianta_infinit_dimensionale_DEM_pc_fix_cont_bar}
   \left[\begin{array}{l}
      B(\hat{x}_{0})-\widetilde{\theta}_{0}-B(\hat{y}_{0}') \\
      G(\hat{x}_{0})-DG(\hat{x}_{0})\hat{y}_{0}' \\
      B(\hat{y}_{i,0}-\hat{y}_{i,0}')-\delta_{i}^{q+m} \\
      DG(\hat{x}_{0})(\hat{y}_{i,0}-\hat{y}_{i,0}') \\
      \bar{\mathcal{B}}(\hat{\mathrm{z}}_{k,0}-\hat{\mathrm{z}}_{k,0}')-\delta_{k}^{n} \\
      H(\lambda_{0},u_{0},(\hat{\mathrm{z}}_{k,0}-\hat{\mathrm{z}}_{k,0}'))
   \end{array}\right]
   \ni 0 \, .
\end{equation}

Let us fix $\hat{x}_{0}$ and $\hat{y}_{0}'$ in $\Phi_{G}(x,y')$
from the first line of $\Phi(x,\phi')$ in
(\ref{e5_57_forma2_sistem_REG_2_GGGvarianta_infinit_dimensionale_DEM_pc_fix_cont_th_princ}).
Let us take $\theta_{0}$ $=$
$\widetilde{\theta}_{0}+B(\hat{y}_{0}')$ and $\varrho$ $=$
$DF(\hat{\lambda}_{0},\hat{u}_{0})(\hat{\mu}_{0}',\hat{w}_{0}')$.
$S_{0}$ denotes the function $S$ from (\ref{e5_57}) formulated for
this $\theta_{0}$ and for $F(\cdot,\cdot)-\varrho$ instead of $F$.
Let us use Remark \ref{observatia_diferentiala_G}.

Then,
(\ref{e5_57_forma2_sistem_REG_2_GGGvarianta_infinit_dimensionale_DEM_pc_fix_cont_bar})
gives
\begin{equation}
\label{e5_57_forma2_sistem_REG_2_GGGvarianta_infinit_dimensionale_DEM_pc_fix_cont_bar_ec}
   \left[\begin{array}{l}
      B(\hat{x}_{0})-\theta_{0} \\
      G(\hat{x}_{0})-\varrho \\
      B(\hat{y}_{i,0}-\hat{y}_{i,0}')-\delta_{i}^{q+m} \\
      D(G(\hat{x}_{0})-\varrho)(\hat{y}_{i,0}-\hat{y}_{i,0}') \\
      \bar{\mathcal{B}}(\hat{\mathrm{z}}_{k,0}-\hat{\mathrm{z}}_{k,0}')-\delta_{k}^{n} \\
      \bar{H}(\varrho,\lambda_{0},u_{0},(\hat{\mathrm{z}}_{k,0}-\hat{\mathrm{z}}_{k,0}'))
   \end{array}\right]
   = 0 \, ,
\end{equation}
where $\bar{H}$ is $H$ for $F(\cdot,\cdot)-\varrho$ instead of
$F(\cdot,\cdot)$,
\begin{equation*}
\label{e5_2_H_bar}
   \bar{H}:Z \times \mathbb{R}^{m} \times W \times \Delta \rightarrow Z, \
      \bar{H}(\varrho,\lambda,u,\mathrm{z})=D_{u}(F(\lambda,u)-\varrho)v
         -\sum_{k=1}^{n}e^{k}\bar{b}_{k} \, .
\end{equation*}

Let $s_{0}$ $=$ $(x_{0}, y_{1,0}, \ldots, y_{q+m,0},
\mathrm{z}_{1,0}, \ldots, \mathrm{z}_{n,0})$ with $x_{0}$ $=$
$\hat{x}_{0}$, $y_{i,0}$ $=$ $\hat{y}_{i,0}-\hat{y}_{i,0}'$,
$\mathrm{z}_{k,0}$ $=$
$\hat{\mathrm{z}}_{k,0}-\hat{\mathrm{z}}_{k,0}'$ and
$i=1,\ldots,q+m$, $k=1,\ldots,n$.

Relation
(\ref{e5_57_forma2_sistem_REG_2_GGGvarianta_infinit_dimensionale_DEM_pc_fix_cont_bar_ec})
means that $s_{0}$ is a solution of the system (\ref{e5_56_b})
verifying the assertion (a) of the statement (ii) of Theorem
\ref{teorema_principala_parteaMAIN}.

It also results $s_{0}$ $\in$
$\mathbb{B}_{a^{\ast}}(\widetilde{s}_{0})$. Indeed, we have:

$\| s_{0} - \widetilde{s}_{0} \|$ $=$ $\| \hat{x}_{0} -
\widetilde{x}_{0} \|$ $+$ $\sum_{i = 1}^{q+m} \|
(\hat{y}_{i,0}-\hat{y}_{i,0}') - \widetilde{y}_{i,0} \|$ $+$
$\sum_{k = 1}^{n} \|
(\hat{\mathrm{z}}_{k,0}-\hat{\mathrm{z}}_{k,0}') -
\widetilde{\mathrm{z}}_{k,0} \|$ $\leq$ $\| \hat{x}_{0} -
\widetilde{x}_{0} \|$ $+$ $\sum_{i = 1}^{q+m} \| \hat{y}_{i,0} -
\widetilde{y}_{i,0} \|$ $+$ $\sum_{k = 1}^{n} \|
\hat{\mathrm{z}}_{k,0} - \widetilde{\mathrm{z}}_{k,0} \|$ $+$
$\sum_{i = 1}^{q+m} \| \hat{y}_{i,0}' - \widetilde{y}_{i,0}' \|$
$+$ $\sum_{k = 1}^{n} \| \hat{\mathrm{z}}_{k,0}' -
\widetilde{\mathrm{z}}_{k,0}' \|$ $\leq$ $\|
(\hat{s}_{0},\hat{\phi}_{0}') -
(\widetilde{s}_{0},\widetilde{\phi}_{0}') \|$ $\leq$ $a^{\ast}$.

\textbf{- (viii) The existence of some isomorphisms}

We have $\gamma_{S_{0}}$ $=$ $\gamma$ , $L_{S}(\varepsilon)$ $=$
$L_{S_{0}}(\varepsilon)$. $DS_{0}(\widetilde{s}_{0})$ $=$
$DS(\widetilde{s}_{0})$ is an isomorphism of $\Gamma$ onto
$\Sigma$. We have $\kappa \leq \gamma$ and
$\frac{1}{2}L_{S}(\varepsilon) \leq L(\varepsilon)$. Then,

$\gamma L_{S}(a^{\ast})$ $\leq$ $2 \gamma
\frac{1}{2}L_{S}(\varepsilon)$ $\leq$ $2 \gamma L(\varepsilon) <
1$.

$L_{S}(a^{\ast})$ $<$ $\frac{1}{\gamma}$, so $DS_{0}(s_{0})$ $=$
$DS(s_{0})$ is an isomorphism of $\Gamma$ onto $\Sigma$.

$\gamma_{\Psi}$ $\leq$ $\gamma$ (this is verified using the
definition), $L_{\Psi}(a^{\ast})$ $\leq$ $L_{S}(a^{\ast})$.

$\gamma_{\Psi} L_{\Psi}(a^{\ast})$ $\leq$ $\gamma L_{S}(a^{\ast})$
$<$ $1$, hence $D\Psi_{0}(x_{0})$ $=$ $D\Psi(x_{0})$ is an
isomorphism of $X$ onto $Y$.

The system (\ref{e5_56_b}) and its solution $s_{0}$ verify the
assertion (b) of the statement (ii) of Theorem
\ref{teorema_principala_parteaMAIN}.

As in the proofs of Theorem 3.1 \cite{CLBichir_bib_Cr_Ra1990},
Theorem IV.3.1 \cite{CLBichir_bib_Gir_Rav1986} and Theorem I.2.1
\cite{CLBichir_bib_CalozRappaz1997}, it follows that $s_{0}$ is
the unique solution of the equation (\ref{e5_56_b}) in
$\mathbb{B}_{a}(\widetilde{s}_{0})$ for every $a$ $\geq$
$a^{\ast}$ that satisfies $\gamma L_{S}(a) < 1$.

We now use Theorem \ref{teorema_principala_parteaMAIN}. Then, the
component $(\lambda_{0},u_{0})$ of $s_{0}$ is a solution of the
equation
(\ref{e5_1_sol_widetilde_x_0h_Inv_Fc_Th_ec_DATA_introd_exact})
that satisfies hypothesis (\ref{ipotezaHypF}) and the rest of the
hypotheses of the statement (i) of Theorem
\ref{teorema_principala_parteaMAIN} in
$\mathbb{B}_{a}(\widetilde{\lambda}_{0},\widetilde{u}_{0})$.

Assume that there exist two different such solutions in
$\mathbb{B}_{a}(\widetilde{\lambda}_{0},\widetilde{u}_{0})$. Then,
equation (\ref{e5_56_b}) has two different solutions in
$\mathbb{B}_{a}(\widetilde{s}_{0})$. This contradicts the
uniqueness of $s_{0}$ in $\mathbb{B}_{a}(\widetilde{s}_{0})$.
Then, $(\lambda_{0},u_{0})$ is the unique solution of the equation
(\ref{e5_1_sol_widetilde_x_0h_Inv_Fc_Th_ec_DATA_introd_exact})
that satisfies hypothesis (\ref{ipotezaHypF}) and the rest of the
hypotheses of the statement (i) of Theorem
\ref{teorema_principala_parteaMAIN} in
$\mathbb{B}_{a}(\widetilde{\lambda}_{0},\widetilde{u}_{0})$ for
every $a$ $\geq$ $a^{\ast}$ that satisfies $\gamma L_{S}(a) < 1$.

\textbf{- (ix) Some estimates}

We have $2 \gamma L_{S_{0}}(a^{\ast})$ $\leq$ $1$ and
$DS_{0}(\widetilde{s}_{0})$ is an isomorphism of $\Gamma$ onto
$\Sigma$. Then, the hypotheses of Lemma
\ref{lema_A2_7_th_fc_impl_th_apl_inv} are satisfied in our
situation. We get: for any $\zeta$ $\in$ $\textrm{int} \,
\mathbb{B}_{\frac{a^{\ast}}{2 \gamma}}(S_{0}(\widetilde{s}_{0}))$,
the equation $S_{0}(s)$ $=$ $\zeta$ has a unique solution $s$ in
$\mathbb{B}_{a_{1}}(\widetilde{s}_{0})$, where $a_{1}$ $=$ $2
\gamma\| S_{0}(\widetilde{s}_{0}) - \zeta \|$ $\leq$ $a^{\ast}$.

For our proof, we take $\zeta = 0$. We obtain $2 \gamma\|
S_{0}(\widetilde{s}_{0}) \|$ $\leq$ $a^{\ast}$. We have $\|
\varrho \|$ $=$ $\| S_{0}(\widetilde{s}_{0}) -
S(\widetilde{s}_{0}) \|$ $\leq$ $\| S_{0}(\widetilde{s}_{0}) \|$
$+$ $\| S(\widetilde{s}_{0}) \|$ $\leq$ $\|
S_{0}(\widetilde{s}_{0}) \|$ $+$ $\| S(\widetilde{s}_{0}) \|$
$\leq$ $\frac{a^{\ast}}{2 \gamma}$ $+$ $\| S(\widetilde{s}_{0})
\|$. Hence,
(\ref{e5_1_sol_widetilde_x_0h_Inv_Fc_Th_ec_DATA_introd_exact_varrho})
holds. We also use
(\ref{dem_e_A2_38_TEXT_G_estimare_M_a_b_h_conditie_inf_th_exact})
and $\frac{a^{\ast}}{2 \gamma}$ $+$ $\frac{1}{2} c a^{\ast}$ $=$
$(\frac{1}{\gamma} - \frac{1}{2}L(\varepsilon))a^{\ast}$.

We also have $\| \varrho \|$ $=$ $\|
DF(\hat{\lambda}_{0},\hat{u}_{0})(\hat{\mu}_{0}',\hat{w}_{0}') \|$
$\leq$ $(\| DF(\widetilde{\lambda}_{0},\widetilde{u}_{0}) \|$ $+$
$L_{F}(a^{\ast})) a^{\ast}$. Then, we have
(\ref{e5_1_sol_widetilde_x_0h_Inv_Fc_Th_ec_DATA_introd_exact_varrho_var}).

Relation (\ref{e5_49_th}) is obtained as in the proofs of Theorem
3.1 \cite{CLBichir_bib_Cr_Ra1990}, Theorem IV.3.1
\cite{CLBichir_bib_Gir_Rav1986} and Theorem I.2.1
\cite{CLBichir_bib_CalozRappaz1997}.

\subsection{Some consequences}
\label{sectiunea_01_O_formulare_pe_spatii_infinit_dimensionale_6}

\begin{cor}
\label{corolarul_raze_teorema_principala_spatii_infinit_dimensionale_widetilde_s_3_0_exact}
Theorem
\ref{teorema_principala_spatii_infinit_dimensionale_widetilde_s_3_0_exact_inf}
holds for those $a^{\ast}$ and $b^{\ast}$ given by Lemma
\ref{Lema_spatii_infinit_dimensionale_th_Graves_estimare_raze_vecinatati},
under the conditions that
(\ref{dem_e_A2_38_TEXT_G_estimare_g_DR_B_b_h}) is verified, that
is, $
\|\mathcal{Q}(s,\phi')-\mathcal{G}(\widetilde{s}_{0},\widetilde{\phi}_{0}')\|
   < b^{\ast}$, for $(s,\phi')$ $\in$
$\mathbb{B}_{a^{\ast}}(\widetilde{s}_{0},\widetilde{\phi}_{0}')$,
and it is satisfied condition I from Theorem
\ref{teorema_contraction_mapping_principle_Dontchev_Rockafellar2009},
that is, (\ref{dem_e_A2_37_TEXT_G_prod_alphaastastast_Conditia
unu}) or $\frac{1}{c}\delta$ $<$ $a^{\ast} (1-\lambda)$, with
$\lambda$ given by
(\ref{dem_e_A2_37_TEXT_G_prod_alphaastastast_var_lambda_LA_FEL_SI_CEL_INITIAL})
or $\lambda$ $=$ $\frac{1}{c}(\frac{1}{2} L(\varepsilon) + \alpha
\widehat{a})$ $<$ $\frac{1}{2}$.
\end{cor}

\begin{cor}
\label{corolarul_trei_3_teorema_principala_spatii_infinit_dimensionale_widetilde_s_3_0_exact_zero}
Assume that $(\widetilde{\lambda}_{0},\widetilde{u}_{0})$ belongs
to a solution branch of equation (\ref{e5_1}).
$(\widetilde{\lambda}_{0},\widetilde{u}_{0})$ can be a regular or
a nonregular solution. Assume the hypotheses of Theorem
\ref{teorema_principala_spatii_infinit_dimensionale_widetilde_s_3_0_exact_inf}.
$\Psi(\widetilde{x}_{0})=0$ in $\delta$. If $\varrho$ $\neq$ $0$,
then, the given problem (\ref{e5_1}) is a perturbation of the
bifurcation problem
(\ref{e5_1_sol_widetilde_x_0h_Inv_Fc_Th_ec_DATA_introd_exact}). If
$\varrho$ $=$ $0$, then, the bifurcation point
$(\lambda_{0},u_{0})$ belongs to the solution branch of equation
(\ref{e5_1}).
\end{cor}

\begin{proof}
Equation (\ref{e5_56_b}) provides a bifurcation problem
(\ref{e5_1_sol_widetilde_x_0h_Inv_Fc_Th_ec_DATA_introd_exact}) for
which the given problem (\ref{e5_1}) is a perturbation.

\qquad
\end{proof}

\begin{cor}
\label{corolarul_trei_3_teorema_principala_spatii_infinit_dimensionale_widetilde_s_3_0_exact_zero_ro}
$\varrho$ $=$ $0$ if and only if $(\hat{\mu}_{0}',\hat{w}_{0}')$
$=$ $\sum_{i=1}^{q+m}\beta_{i}$ $((\hat{\mu}_{i,0},\hat{w}_{i,0})$
$-$ $(\hat{\mu}_{i,0}',\hat{w}_{i,0}'))$, $\beta_{i}$ $\in$
$\mathbb{R}$.
\end{cor}

\subsection{The plan of the study for the next sections}
\label{sectiunea_01_O_formulare_pe_spatii_infinit_dimensionale_7}

In Section
\ref{sectiunea_01_O_formulare_pe_spatii_infinit_dimensionale_COMPLEMENTE_NOU},
we investigate the formulation of $S_{0}$ from (\ref{e5_56_b})
using the function
$F(\lambda,u)-DF(\lambda,u)(\hat{\mu}_{0}',\hat{w}_{0}')$ (where
$(\hat{\mu}_{0}',\hat{w}_{0}')$ is fixed) instead of the function
$F(\lambda,u)-DF(\hat{\lambda}_{0},\hat{u}_{0})(\hat{\mu}_{0}',\hat{w}_{0}')$
(where $(\hat{\lambda}_{0},\hat{u}_{0})$ and
$(\hat{\mu}_{0}',\hat{w}_{0}')$ are fixed and $\varrho$ $=$
$DF(\hat{\lambda}_{0},\hat{u}_{0})(\hat{\mu}_{0}',\hat{w}_{0}')$
$=$ $constant$).

In Sections \ref{sectiunea05_class} -
\ref{sectiunea04_cazul_ecNS}, for simplicity, we study only the
case $\varrho$ $=$
$DF(\hat{\lambda}_{0},\hat{u}_{0})(\hat{\mu}_{0}',\hat{w}_{0}')$
$=$ $constant$ from Theorem
\ref{teorema_principala_spatii_infinit_dimensionale_widetilde_s_3_0_exact_inf}.

\section{The existence of a class of equivalent maps}
\label{sectiunea05_class}

We prove that there exists a class (family) of maps $C^{p}$ -
equivalent (right equivalent) at $(\lambda_{0},u_{0})$ to
$F(\cdot)-\varrho$ at $(\lambda_{0},u_{0})$ and that satisfies
hypothesis (\ref{ipotezaHypF}) in $(\lambda_{0},u_{0})$.

We use the definition of equivalence of maps from
\cite{CLBichir_bib_E_Zeidler_NFA_IV,
CLBichir_bib_E_Zeidler_main1995}. The following discussion has the
references \cite{CLBichir_bib_Ashwin_Bohmer_Mei1995,
CLBichir_bib_Chillingworth1980,
CLBichir_bib_Golubitsky_Schaeffer1985,
CLBichir_bib_Golubitsky_Stewart_Schaeffer1988,
CLBichir_bib_Govaerts2000, CLBichir_bib_Jepson_Spence1989,
CLBichir_bib_Jepson_Spence_Cliffe1991,
CLBichir_bib_E_Zeidler_NFA_I, CLBichir_bib_E_Zeidler_NFA_IV,
CLBichir_bib_E_Zeidler_main1995}.

Consider the equation
(\ref{e5_1_sol_widetilde_x_0h_Inv_Fc_Th_ec_DATA_introd_exact}).

Let $\varphi_{0}$ be a local $C^{p}$ - diffeomorphism from some
open neighborhood $\mathcal{U}(\lambda_{0},u_{0})$ $=$
$\mathcal{U}_{1}(\lambda_{0})$ $\times$ $\mathcal{U}_{2}(u_{0})$
of $(\lambda_{0},u_{0})$ in $\mathbb{R}^{m} \times W$ onto some
open neighborhood
$\widehat{\mathcal{U}}(\widehat{\lambda}_{0},\widehat{u}_{0})$ $=$
$\widehat{\mathcal{U}}_{1}(\widehat{\lambda}_{0})$ $\times$
$\widehat{\mathcal{U}}_{2}(\widehat{u}_{0})$ of
$(\widehat{\lambda}_{0},\widehat{u}_{0})$ $=$
$(\lambda_{0},u_{0})$ in $\mathbb{R}^{m} \times W$ of the form
$(\widehat{\lambda},\widehat{u})$ $=$ $\varphi_{0}(\lambda,u)$ $=$
$(\Lambda(\lambda),\varphi_{\ast}(\lambda,u))$, $\Lambda(\lambda)$
$\in$ $\mathbb{R}^{m}$, $\varphi_{\ast}(\lambda,u)$ $\in$ $W$,
satisfying the following assumptions:
$\varphi_{0}(\lambda_{0},u_{0})$ $=$
$(\widehat{\lambda}_{0},\widehat{u}_{0})$,
$D_{u}\varphi_{\ast}(\lambda_{0},u_{0})$ is bijective.

Let $\Pi$ be a local $C^{p}$ - diffeomorphism  from some open
neighborhood $\mathcal{U}_{3}(0)$ of $0$ in $\mathbb{R}^{q}$ onto
some open neighborhood $\widehat{\mathcal{U}}_{3}(0)$ of $0$ in
$\mathbb{R}^{q}$. $\Pi(0)$ $=$ $0$.

Let $\bar{\Pi}$ be a local $C^{p}$ - diffeomorphism  from some
open neighborhood $\mathcal{U}_{4}(0)$ of $0$ in $\mathbb{R}^{n}$
onto some open neighborhood $\widehat{\mathcal{U}}_{4}(0)$ of $0$
in $\mathbb{R}^{n}$. $\bar{\Pi}(0)$ $=$ $0$.

$(\widehat{\lambda},\widehat{u})$ $=$ $\varphi_{0}(\lambda,u)$ $=$
$(\Lambda(\lambda),\varphi_{\ast}(\lambda,u))$, $(\lambda,u)$ $=$
$\varphi_{0}^{-1}(\widehat{\lambda},\widehat{u})$.

$\widehat{f}$ $=$ $\Pi(f)$, $f$ $=$ $\Pi^{-1}(\widehat{f})$ $=$
$((\Pi^{-1})_{1}(\widehat{f}), \ldots,
(\Pi^{-1})_{q}(\widehat{f}))$.

$\widehat{e}$ $=$ $\bar{\Pi}(e)$, $e$ $=$
$\bar{\Pi}^{-1}(\widehat{e})$ $=$
$((\bar{\Pi}^{-1})_{1}(\widehat{e}), \ldots,
(\bar{\Pi}^{-1})_{n}(\widehat{e}))$.

Let $\varphi$ be the local $C^{p}$ - diffeomorphism from
$\mathcal{U}_{3}(0)$ $\times$ $\mathcal{U}_{1}(\lambda_{0})$
$\times$ $\mathcal{U}_{2}(u_{0})$ onto
$\widehat{\mathcal{U}}_{3}(0)$ $\times$
$\widehat{\mathcal{U}}_{1}(\widehat{\lambda}_{0})$ $\times$
$\widehat{\mathcal{U}}_{2}(\widehat{u}_{0})$ given by
$\varphi(f,\lambda,u)$ $=$ $(\Pi(f),\varphi_{0}(\lambda,u))$.

$(\widehat{f},\widehat{\lambda},\widehat{u})$ $=$
$\varphi(f,\lambda,u)$ $=$ $(\Pi(f),\varphi_{0}(\lambda,u))$ $=$
$(\Pi(f),\Lambda(\lambda),\varphi_{\ast}(\lambda,u))$.

$(f,\lambda,u)$ $=$
$\varphi^{-1}(\widehat{f},\widehat{\lambda},\widehat{u})$ $=$
$(\Pi^{-1}(\widehat{f}),\varphi_{0}^{-1}(\widehat{\lambda},\widehat{u}))$

$D(\varphi_{0}^{-1})(\widehat{\lambda},\widehat{u})$ $=$
$D\varphi_{0}(\lambda,u)^{-1}$, where
$(\widehat{\lambda},\widehat{u})$ $=$ $\varphi_{0}(\lambda,u)$,
for all $(\lambda,u)$ in a small neighborhood of
$(\lambda_{0},u_{0})$.

$D(\Pi^{-1})(\widehat{f})$ $=$ $D\Pi(f)^{-1}$, where $\widehat{f}$
$=$ $\Pi(f)$, for all $f$ in a small neighborhood of $0$.

$D(\bar{\Pi}^{-1})(\widehat{e})$ $=$ $D\bar{\Pi}(e)^{-1}$, where
$\widehat{e}$ $=$ $\bar{\Pi}(e)$, for all $e$ in a small
neighborhood of $0$.

$D(\varphi^{-1})(\widehat{f},\widehat{\lambda},\widehat{u})$ $=$
$D\varphi(f,\lambda,u)^{-1}$, where
$(\widehat{f},\widehat{\lambda},\widehat{u})$ $=$
$\varphi(f,\lambda,u)$, for all $(f,\lambda,u)$ in a small
neighborhood of $(0,\lambda_{0},u_{0})$.

When we write $(\lambda,u)$ $\in$ $\mathcal{U}_{1}(\lambda_{0})$
$\times$ $\mathcal{U}_{2}(u_{0})$, without "$\forall$", related to
a relation, we understand that $(\lambda,u)$ is in the maximal
neighborhood (related to inclusion) contained in
$\mathcal{U}_{1}(\lambda_{0})$ $\times$ $\mathcal{U}_{2}(u_{0})$
such that that relation holds and so on.

Let us retain $F(\lambda,u)- \varrho$ from
(\ref{e5_1_sol_widetilde_x_0h_Inv_Fc_Th_ec_DATA_introd_exact}).

Using the definition for $C^{p}$ equivalence (right equivalence)
of maps on open subsets of Banach spaces, from
\cite{CLBichir_bib_E_Zeidler_NFA_IV,
CLBichir_bib_E_Zeidler_main1995}, related to
\cite{CLBichir_bib_Ashwin_Bohmer_Mei1995,
CLBichir_bib_Chillingworth1980,
CLBichir_bib_Golubitsky_Schaeffer1985,
CLBichir_bib_Golubitsky_Stewart_Schaeffer1988,
CLBichir_bib_Govaerts2000, CLBichir_bib_Jepson_Spence1989,
CLBichir_bib_Jepson_Spence_Cliffe1991,
CLBichir_bib_E_Zeidler_NFA_I}, let us consider a map
$\widehat{F}:\mathbb{R}^{m} \times W \rightarrow Z$, of class
$C^{p}$, defined locally by
\begin{equation}
\label{e5_1_sol_widetilde_x_0h_Inv_Fc_Th_ec_DATA_introd_exact_echiv_F}
      \widehat{F}(\widehat{\lambda},\widehat{u})
      = F(\varphi_{0}^{-1}(\widehat{\lambda},\widehat{u}))-\varrho \, , \
      (\widehat{\lambda},\widehat{u}) \in
      \widehat{\mathcal{U}}_{1}(\widehat{\lambda}_{0}) \times \widehat{\mathcal{U}}_{2}(\widehat{u}_{0}) \, ,
\end{equation}
so $\widehat{F}(\varphi_{0}(\lambda,u))$ $=$
$F(\lambda,u)-\varrho$, $(\lambda,u)$ $\in$
$\mathcal{U}_{1}(\lambda_{0})$ $\times$ $\mathcal{U}_{2}(u_{0})$,
\begin{equation}
\label{e5_1_sol_widetilde_x_0h_Inv_Fc_Th_ec_DATA_introd_exact_echiv_F_cont}
      \widehat{F}(\varphi_{0}(\lambda,u)) = F(\lambda,u)-\varrho \, , \
      (\lambda,u) \in
      \mathcal{U}_{1}(\lambda_{0}) \times \mathcal{U}_{2}(u_{0}) \, ,
\end{equation}
that is, $\widehat{F}$ is $C^{p}$ - equivalent (right equivalent)
at $(\widehat{\lambda}_{0},\widehat{u}_{0})$ to $F-\varrho$ at
$(\lambda_{0},u_{0})$.

We have $\widehat{F}(\varphi_{0}(\lambda_{0},u_{0}))$ $=$
$\widehat{F}(\lambda_{0},u_{0})$ $=$ $0$ so
$(\widehat{\lambda}_{0},\widehat{u}_{0})$ is the solution of
\begin{equation}
\label{e5_1_sol_widetilde_x_0h_Inv_Fc_Th_ec_DATA_introd_exact_echiv_sol_ec}
      \widehat{F}(\widehat{\lambda},\widehat{u}) = 0 \, ,
\end{equation}

Let us prove that $\widehat{F}$ satisfies hypothesis
(\ref{ipotezaHypF}) in $(\lambda_{0},u_{0})$.

Let us consider a map $\widehat{G}:\mathbb{R}^{q} \times
\mathbb{R}^{m} \times W \rightarrow Z$, of class $C^{p}$, defined
locally, for $(\widehat{f},\widehat{\lambda},\widehat{u})$ $\in$
$\widehat{\mathcal{U}}_{3}(0) \times
\widehat{\mathcal{U}}(\widehat{\lambda}_{0},\widehat{u}_{0})$, by
\begin{equation}
\label{e5_1_sol_widetilde_x_0h_Inv_Fc_Th_ec_DATA_introd_exact_echiv_G}
      \widehat{G}(\widehat{f},\widehat{\lambda},\widehat{u})
      = F(\varphi_{0}^{-1}(\widehat{\lambda},\widehat{u}))-\varrho
      -\sum_{i=1}^{q}(\Pi^{-1})_{i}(\widehat{f})\bar{a}_{i} \, ,
\end{equation}
or $\widehat{G}(\varphi(f,\lambda,u))$ $=$
$G(f,\lambda,u)-\varrho$, $(f,\lambda,u)$ $\in$
$\mathcal{U}_{3}(0)$ $\times$ $\mathcal{U}(\lambda_{0},u_{0})$,
that is, $G-\varrho$ is $C^{p}$ - equivalent (right equivalent) at
$(0,\lambda_{0},u_{0})$ to $\widehat{G}$ at
$(0,\widehat{\lambda}_{0},\widehat{u}_{0})$.

We have $\widehat{G}(\widehat{f},\widehat{\lambda},\widehat{u})$
$=$
$G(\varphi^{-1}(\widehat{f},\widehat{\lambda},\widehat{u}))-\varrho$
$=$
$G(\Pi^{-1}(\widehat{f}),\varphi_{0}^{-1}(\widehat{\lambda},\widehat{u}))-\varrho$
$=$
$F(\varphi_{0}^{-1}(\widehat{\lambda},\widehat{u}))-\sum_{i=1}^{q}(\Pi^{-1})_{i}(\widehat{f})\bar{a}_{i}-\varrho$
$=$
$\widehat{F}(\widehat{\lambda},\widehat{u})-\sum_{i=1}^{q}(\Pi^{-1})_{i}(\widehat{f})\bar{a}_{i}$.

Let us introduce some new notations.

$(\widehat{g},\widehat{\mu},\widehat{w})$ $=$
$D\varphi(f,\lambda,u)(g,\mu,w)$ $=$
$(D\Pi(f)g,D\varphi_{0}(\lambda,u)(\mu,w))$.

$(g,\mu,w)$ $=$
$D\varphi(f,\lambda,u)^{-1}(\widehat{g},\widehat{\mu},\widehat{w})$
$=$
$(D\Pi(f)^{-1}\widehat{g},D\varphi_{0}(\lambda,u)^{-1}(\widehat{\mu},\widehat{w}))$.

$D\Pi(f)g$ $=$ $\widehat{g}$, $g$ $=$ $D\Pi(f)^{-1}\widehat{g}$,
with $D\Pi(f)^{-1}\widehat{g}$ $=$
$D(\Pi^{-1})(\widehat{f})\widehat{g}$ if $\widehat{f}$ $=$
$\Pi(f)$, for all $f$ in a small neighborhood of $0$.

$D\varphi_{0}(\lambda,u)(\mu,w)$ $=$
$(D\Lambda(\lambda)\mu,D\varphi_{\ast}(\lambda,u)(\mu,w))$ $=$
$(\widehat{\mu},\widehat{w})$

$(\widehat{g}_{i},\widehat{\mu}_{i},\widehat{w}_{i})$ $=$
$D\varphi(f,\lambda,u)(g_{i},\mu_{i},w_{i})$.

$(0,\widehat{\mu}_{i,0},\widehat{w}_{i,0})$ $=$
$D\varphi(0,\lambda_{0},u_{0})(0,\mu_{i,0},w_{i,0})$ $=$
$(D\Pi(0)0,D\varphi_{0}(\lambda_{0},u_{0})(\mu_{i,0},w_{i,0}))$
$=$
$(0,D\Lambda(\lambda_{0})\mu_{i,0},D\varphi_{\ast}(\lambda_{0},u_{0})(\mu_{i,0},w_{i,0}))$.

$(\mu_{i,0},w_{i,0})$ $=$
$D\varphi_{0}(\lambda_{0},u_{0})^{-1}(\widehat{\mu}_{i,0},\widehat{w}_{i,0})$,
$0$ $=$ $D\Pi(0)^{-1}0$.

$\widehat{v}$ $=$ $D_{u}\varphi_{\ast}(\lambda,u)v$,
$\widehat{v}_{k}$ $=$ $D_{u}\varphi_{\ast}(\lambda,u)v_{k}$,
$\widehat{v}_{k,0}$ $=$
$D_{u}\varphi_{\ast}(\lambda_{0},u_{0})v_{k,0}$, $v_{k,0}$ $=$
$D_{u}\varphi_{\ast}(\lambda_{0},u_{0})^{-1}\widehat{v}_{k,0}$,
$e_{k}$ $=$ $D(\bar{\Pi}^{-1})(0)\widehat{e}_{k}$, $0$ $=$
$D\bar{\Pi}(0)^{-1}0$.

We have

$D(G(f,\lambda,u)-\varrho)(g,\mu,w)$ $=$
$D\widehat{G}(\varphi(f,\lambda,u))D\varphi(f,\lambda,u)(g,\mu,w)$
$=$
$D\widehat{G}(\widehat{f},\widehat{\lambda},\widehat{u})(\widehat{g},\widehat{\mu},\widehat{w})$
$=$
$D\widehat{F}(\widehat{\lambda},\widehat{u})(\widehat{\mu},\widehat{w})$
$-$
$\sum_{i=1}^{q}D(\Pi^{-1})_{i}(\widehat{f})\widehat{g}\bar{a}_{i}$

$D(F(\lambda,u)-\varrho)(\mu,w)$ $=$
$D\widehat{F}(\varphi_{0}(\lambda,u))D\varphi_{0}(\lambda,u)(\mu,w)$
$=$
$D\widehat{F}(\widehat{\lambda},\widehat{u})(\widehat{\mu},\widehat{w})$.

$D_{u}(F(\lambda,u)-\varrho)v$ $=$
$D_{u}\widehat{F}(\varphi_{0}(\lambda,u))v$ $=$
$D_{u}\widehat{F}(\Lambda(\lambda),\varphi_{\ast}(\lambda,u))v$
$=$
$D_{\varphi_{\ast}}\widehat{F}(\varphi_{0}(\lambda,u))D_{u}\varphi_{\ast}(\lambda,u)v$
$=$
$D_{\varphi_{\ast}}\widehat{F}(\widehat{\lambda},\widehat{u})\widehat{v}$
$=$
$D_{\widehat{u}}\widehat{F}(\widehat{\lambda},\widehat{u})\widehat{v}$.

Let us consider a map $\widehat{H}:\mathbb{R}^{m} \times W \times
\Delta \rightarrow Z$, of class $C^{p-1}$, defined locally, for
$(\widehat{e},\widehat{v})$ $\in$ $\widehat{\mathcal{U}}_{4}(0)
\times \widehat{\mathcal{U}}_{2}(\widehat{u}_{0})$, by
\begin{equation}
\label{e5_1_sol_widetilde_x_0h_Inv_Fc_Th_ec_DATA_introd_exact_echiv_H}
   \widehat{H}(\widehat{\lambda},\widehat{u},(\widehat{e},\widehat{v}))
      =D_{\varphi_{\ast}}\widehat{F}(\widehat{\lambda},\widehat{u})\widehat{v}
         -\sum_{k=1}^{n}D(\bar{\Pi}^{-1})_{k}(0)\widehat{e}\bar{b}_{k} \, .
\end{equation}

We have

\begin{equation}
\label{e5_57_echiv_G_var_0_patru}
\begin{split}
 & B(\Pi^{-1}(\widehat{f}),\varphi_{0}^{-1}(\widehat{\lambda},\widehat{u}))-\theta_{0}
      =B(f,\lambda,u)-\theta_{0}   \\
 & \widehat{G}(\widehat{f},\widehat{\lambda},\widehat{u})
      =G(f,\lambda,u)-\varrho   \\
 & B(D\Pi(f)^{-1}\widehat{g}_{i},D\varphi_{0}(\lambda,u)^{-1}(\widehat{\mu}_{i},\widehat{w}_{i}))-\delta_{i}^{q+m}
      = B(g_{i},\mu_{i},w_{i})-\delta_{i}^{q+m}   \\
 & D\widehat{G}(\widehat{f},\widehat{\lambda},\widehat{u})(\widehat{g}_{i},\widehat{\mu}_{i},\widehat{w}_{i})
      = D(G(f,\lambda,u)-\varrho)(g_{i},\mu_{i},w_{i})   \\
 & \bar{\mathcal{B}}(D(\bar{\Pi}^{-1})(0)\widehat{e}_{k},D_{u}\varphi_{\ast}(\lambda_{0},u_{0})^{-1}\widehat{v}_{k})-\delta_{k}^{n}
      = \bar{\mathcal{B}}(e_{k},v_{k})-\delta_{k}^{n}   \\
 & \widehat{H}(\widehat{\lambda},\widehat{u},(\widehat{e}_{k},\widehat{v}_{k}))
      = H(\lambda,u,(e_{k},v_{k}))
\end{split}
\end{equation}

Let us define
\begin{equation}
\label{e5_57_echiv_G_var_0_patru_zero_widehat_def}
\begin{split}
 & \widehat{B}_{N}(\overline{f},\overline{\lambda},\overline{u})
      =B(\Pi^{-1}(\overline{f}),\varphi_{0}^{-1}(\overline{\lambda},\overline{u})) \\
 & \widehat{B}(\overline{g},\overline{\mu},\overline{w})
      =B(D\Pi(0)^{-1}\overline{g},D\varphi_{0}(\lambda_{0},u_{0})^{-1}(\overline{\mu},\overline{w})) \\
 & \widehat{\bar{\mathcal{B}}}(\overline{e},\overline{v})
      =\bar{\mathcal{B}}(D(\bar{\Pi}^{-1})(0)\overline{e},D_{u}\varphi_{\ast}(\lambda_{0},u_{0})^{-1}\overline{v})
\end{split}
\end{equation}

and

\begin{equation}
\label{e5_57_widehat}
   \widehat{S}:\Gamma \rightarrow \Sigma , \
   \widehat{S}(\overline{s})=\left[\begin{array}{l}
      \widehat{B}_{N}(\overline{f},\overline{\lambda},\overline{u})-\theta_{0} \\
      \widehat{G}(\overline{f},\overline{\lambda},\overline{u}) \\
      \widehat{B}(\overline{g}_{i},\overline{\mu}_{i},\overline{w}_{i})-\delta_{i}^{q+m} \\
      D\widehat{G}(\overline{f},\overline{\lambda},\overline{u})(\overline{g}_{i},\overline{\mu}_{i},\overline{w}_{i}) \\
      \widehat{\bar{\mathcal{B}}}(\overline{e}_{k},\overline{v}_{k})-\delta_{k}^{n} \\
      \widehat{H}(\overline{\lambda},\overline{u},(\overline{e}_{k},\overline{v}_{k}))
      \end{array}\right] \ ,
   \ \begin{array}{l}
      i=1,\ldots,q+m \, , \\
      k=1,\ldots,n \, ,
      \end{array}
\end{equation}

Let $\overline{s}_{0}$ $=$
$(\overline{x}_{0},\overline{y}_{1,0},\ldots,\overline{y}_{q+m,0},\overline{\mathrm{z}}_{1,0},\ldots,\overline{\mathrm{z}}_{n,0})$,
$\overline{x}_{0}$ $=$
$(\overline{f}_{0},\overline{\lambda}_{0},\overline{u}_{0})$ $=$
$(0,\widehat{\lambda}_{0},\widehat{u}_{0})$, $\overline{y}_{i,0}$
$=$ $(\overline{g}_{i,0},\overline{\mu}_{i,0},\overline{w}_{i,0})$
$=$ $(0,\widehat{\mu}_{i,0},\widehat{w}_{i,0})$,
$\overline{\mathrm{z}}_{k,0}$ $=$
$(\overline{e}_{k,0},\overline{v}_{k,0})$ $=$
$(0,\widehat{v}_{k,0})$, $i=1,\ldots,q+m$, $k=1,\ldots,n$.

From (\ref{e5_57_echiv_G_var_0_patru}), we obtain
\begin{equation}
\label{e5_57_widehat_patru_s0}
   \widehat{S}(\overline{s}_{0})=S(s_{0})=0 \, .
\end{equation}

\begin{thm}
\label{teorema_principala_spatii_infinit_dimensionale_widetilde_s_3_0_exact_inf_echiv_G}

There exists a class of maps $C^{p}$ - equivalent (right
equivalent) at $(\lambda_{0},u_{0})$ to $F(\cdot)-\varrho$ at
$(\lambda_{0},u_{0})$ and that satisfies hypothesis
(\ref{ipotezaHypF}) in $(\lambda_{0},u_{0})$.

\end{thm}

\begin{proof}

There exists a local $C^{p}$ - diffeomorphism $\varphi_{0}$, as
defined above, such that
(\ref{e5_1_sol_widetilde_x_0h_Inv_Fc_Th_ec_DATA_introd_exact_echiv_F_cont})
takes place. We now apply Corollary
\ref{corolar_teorema_principala_parteaMAIN_neliniar} for the
equation $\widehat{S}(\overline{s})$ $=$ $0$ and for its solution
$\overline{s}_{0}$. Hence, $\widehat{F}$ satisfies hypothesis
(\ref{ipotezaHypF}) in $(\lambda_{0},u_{0})$. Taking all the local
$C^{p}$ - diffeomorphism of the form of $\varphi_{0}$,
$\widehat{F}$ becomes a representative of a class of maps $C^{p}$
- equivalent (right equivalent) at $(\lambda_{0},u_{0})$ to
$F(\cdot)-\varrho$ at $(\lambda_{0},u_{0})$ and that satisfies
hypothesis (\ref{ipotezaHypF}) in $(\lambda_{0},u_{0})$.

\qquad
\end{proof}

\section{Bifurcation of the solutions of approximate equations}
\label{sectiunea06}

In this Section, under the hypothesis that the exact problem has a
bifurcation point, we study the existence of an equation, defined
on closed subspaces of the given spaces, which has a bifurcation
point that approximates the exact one and has the same type as
this one.

\subsection{Approximate spaces and functions}
\label{sectiunea_an_approximate_problem}

Let us consider a closed subspace $W_{h}$ of $W$ and a closed
subspace $Z_{h}$ of $Z$. The spaces $W_{h}$ and $Z_{h}$ are both
finite-dimensional spaces, with $dim \ W_{h}$ $=$ $dim \ Z_{h}$,
or they are both infinite-dimensional spaces. Let $F_{h} \in
C^{p}(\mathbb{R}^{m} \times W_{h},Z_{h})$ be an approximation of
the function $F$ that defines equation (\ref{e5_1}). $h$ is a
positive parameter.

The mathematical entities are the same in the exact case and in
the approximate case, depending on the Banach spaces only. This is
why, in the approximate case, we maintain the notations from the
exact case and adjoin an index $h$. The entities $W$, $Z$, $F$,
$(\lambda,u)$, $X$, $Y$, $x$, $x_{0}=(f_{0},\lambda_{0},u_{0})$,
$\widetilde{x}_{0}$, ..., from the infinite-dimensional case,
become $W_{h}$, $Z_{h}$, $F_{h}$, $(\lambda_{h},u_{h})$, $X_{h}$,
$Y_{h}$, $x_{h}$, $x_{0h}=(f_{0h},\lambda_{0h},u_{0h})$,
$\widetilde{x}_{0h}$, ..., respectively, in the approximate case.

Let $\pi_{h}^{W}:W \rightarrow W_{h}$ and $\pi_{h}^{Z} \in
L(Z,Z_{h})$ be two operators. As a conclusion of the conditions
used to relate the exact spaces and the approximate spaces in the
literature \cite{CLBichir_bib_Bochev_GunzburgerLSFEM2009,
CLBichir_bib_Bohmer2001, CLBichir_bib_Bohmer_Dahlke2003,
CLBichir_bib_Bohmer2010, CLBichir_bib_Brezzi_Rappaz_Raviart1_1980,
CLBichir_bib_Brezzi_Rappaz_Raviart2_1981,
CLBichir_bib_Brezzi_Rappaz_Raviart3_1981,
CLBichir_bib_Cliffe_Spence_Tavener2000, CLBichir_bib_Cr_Ra1990,
CLBichir_bib_Elman_Silvester_Wathen2005, CLBichir_bib_Gir_Rav1986,
CLBichir_bib_Ili1980, CLBichir_bib_Ka_Ak1986,
CLBichir_bib_Quarteroni_Valli2008, CLBichir_bib_Temam1979,
CLBichir_bib_E_Zeidler_NFA_IIA, CLBichir_bib_E_Zeidler_NFA_IIB},
we assume that $\pi_{h}^{Z}$ has the following property
\begin{equation}
\label{e5_8_proprietati}
   \| \bar{a}-\pi_{h}^{Z}\bar{a} \|_{Z} \leq C \| \bar{a} \|_{Z},
      \ \forall \bar{a} \in Z, \ \textrm{where} \ C < 1 \, .
\end{equation}
Let $\bar{a}_{ih}$ $=$ $\pi_{h}^{Z}\bar{a}_{i}$, $i=1,\ldots,q$,
$\bar{b}_{kh}$ $=$ $\pi_{h}^{Z}\bar{b}_{k}$, $k=1,\ldots,n$,
$Z_{1,h}$ $=$ $sp \ \{\bar{a}_{1h},\ldots,\bar{a}_{qh} \}$,
$Z_{3,h}$ $=$ $sp \ \{\bar{b}_{1h},\ldots,\bar{b}_{nh} \}$. Let
$I_{N}$ be the identity operator on $\mathbb{R}^{N}$, $N \geq 1$,
and $\widetilde{f}_{h}$ $=$ $I_{q}f$, $\widetilde{g}_{h}$ $=$
$I_{q}g$ $\in$ $\mathbb{R}^{q}$, $\widetilde{e}_{h}$ $=$ $I_{n}e$
$\in$ $\mathbb{R}^{n}$, $\widetilde{\lambda}_{h}$ $=$
$I_{m}\lambda$, $\widetilde{\mu}_{h}$ $=$ $I_{m}\mu$ $\in$
$\mathbb{R}^{m}$, $\widetilde{u}_{h}$ $=$ $\pi_{h}^{W}u$,
$\widetilde{w}_{h}$ $=$ $\pi_{h}^{W}w$, $\widetilde{v}_{h}$ $=$
$\pi_{h}^{W}v$ $\in$ $W_{h}$. Since $f_{0}=0$, $g_{i,0}=0$,
$e_{k,0}=0$, we take $\widetilde{f}_{0h}=0$,
$\widetilde{g}_{i,0,h}=0$, $\widetilde{e}_{k,0,h}=0$.

\begin{lem}
\label{Lema_aih_linearly_independent_set} (i) For every $\bar{a}
\in Z$, $\bar{a} \neq 0$, we have $\pi_{h}^{Z}\bar{a} \neq 0$.
(ii) Assume that the elements $\bar{a}_{1}$, $\ldots$,
$\bar{a}_{k}$ $\in$ $Z$ form a linearly independent set. Then, the
elements $\bar{a}_{1h}, \ldots, \bar{a}_{kh}$ $\in$ $Z_{h}$ form a
linearly independent set.
\end{lem}

\begin{proof}(i) If there exists $\bar{a} \neq 0$ such that $\pi_{h}^{Z}\bar{a} = 0$,
then we deduce from (\ref{e5_8_proprietati}) that $1 \leq C$. This
contradicts the condition $C < 1$.

(ii) Suppose that $\bar{a}_{1h}$, $\ldots$, $\bar{a}_{qh}$ form a
linearly dependent set, therefore there exists
$f'=((f')^{1},\ldots,(f')^{q}) \in \mathbb{R}^{q}$, $f' \neq 0$
and $\sum_{i=1}^{q}(f')^{i}\bar{a}_{ih}=0$. Then,
$\sum_{i=1}^{q}(f')^{i}\bar{a}_{i} \neq 0$. This contradicts the
condition (i) since
$\pi_{h}^{Z}(\sum_{i=1}^{q}(f')^{i}\bar{a}_{i})$ $=$
$\sum_{i=1}^{q}(f')^{i}\bar{a}_{ih}$ $=$ $0$.

\qquad
\end{proof}

Introduce the function of class $C^{p}$,
\begin{equation*}
\label{e5_10}
   G_{h}:X_{h} \rightarrow Z_{h}, \ G_{h}(x_{h})
      =F_{h}(\lambda_{h},u_{h})-\sum_{i=1}^{q}f_{h}^{i}\bar{a}_{ih} \, ,
\end{equation*}
and the function of class $C^{p-1}$,
\begin{equation*}
\label{e5_10_H}
   H_{h}:\mathbb{R}^{m} \times W_{h} \times \Delta_{h} \rightarrow Z_{h}, \
      H_{h}(\lambda_{h},u_{h},\mathrm{z}_{h})=D_{u}F_{h}(\lambda_{h},u_{h})v_{h}
         -\sum_{k=1}^{n}e_{h}^{k}\bar{b}_{kh} \, .
\end{equation*}

Introduce two operators $B_{h}$, $B_{h} \in
L(X_{h},\mathbb{R}^{q+m})$ and $\bar{\mathcal{B}}_{h}$ $\in$ $
L(\Delta_{h},\mathbb{R}^{n})$ that approximate the operators $B$
and $\bar{\mathcal{B}}$, respectively. $\widetilde{\theta}_{0h}$
$=$ $B_{h}(\widetilde{x}_{0h})$ and
\begin{equation}
\label{e5_23_h}
   \Psi_{h}:X_{h} \rightarrow Y_{h} , \
   \Psi_{h}(x_{h})=\left[\begin{array}{l}
      B_{h}(x_{h})-\widetilde{\theta}_{0h}\\
      G_{h}(x_{h})
      \end{array}\right] \, .
\end{equation}
We have
$D\Psi_{h}(x_{h})y_{h}=[B_{h}(y_{h}),DG_{h}(x_{h})y_{h}]^{T}$. We
also define
\begin{equation}
\label{e5_26_T1_h}
    \Phi_{G,h}(x_{h},\cdot):X_{h}
      \rightarrow Y_{h}, \
    \Phi_{G,h}(x_{h},y_{h})
    =[B_{h}(y_{h}),DG_{h}(x_{h})y_{h}]^{T} \, ,
\end{equation}
\begin{equation}
\label{e5_27_S_5_T2_h}
   \Phi_{H,h}(x_{h},\cdot):\Delta_{h}
      \rightarrow \Sigma_{h}, \
   \Phi_{H,h}(x_{h},\mathrm{z}_{h})
   =[\bar{\mathcal{B}}_{h}(\mathrm{z}_{h}),H_{h}(\lambda_{h},u_{h},\mathrm{z}_{h})]^{T} \, .
\end{equation}

\begin{rem}
\label{observatia5_ecuatia_aprox_Cr_Rap1} In
\cite{CLBichir_bib_Cr_Ra1990}, the approximate equations are
constructed on the same spaces as the exact equation and the exact
operator $B$ is maintained in the approximate case.
\end{rem}

\subsection{Some sufficient conditions for
$\Phi_{G,h}(\widetilde{x}_{0h},\cdot)$ and
$\Phi_{H,h}(\widetilde{x}_{0h},\cdot)$ to be isomorphisms}
\label{sectiunea_conditii_suficiente}

Assume the following hypotheses:
\begin{equation}
\label{e5_45p_conditia_izomorfism} \textrm{there exists an
isomorphism} \ \mathcal{J} \
      \textrm{of} \ W \
      \textrm{onto} \ Z \ \textrm{such that} \ \mathcal{J}(W_{h}) = Z_{h} \, ,
\end{equation}
\begin{eqnarray}
   &  \textrm{there exists} \ \eta_{1} > 0 \
      \textrm{such that, for every} \
      x_{h} \in X_{h}, \
      \textrm{we have}
         \label{e5_45p_conditia_III} \\
   &  \qquad   \| Bx_{h} - B_{h}x_{h}
      \|_{\mathbb{R}^{q+m}}
      \leq
      \eta_{1} \| x_{h} \|_{X_{h}} \, ,
         \nonumber
\end{eqnarray}
\begin{eqnarray}
   &  \textrm{there exists} \ \eta_{2} > 0 \
      \textrm{such that, for every} \
      (\lambda_{h},u_{h}) \in \mathbb{R}^{m} \times W_{h}, \
      \textrm{we have}
         \label{e5_45p_conditia_I} \\
   &  \| \pi_{h}^{Z}DF(\lambda_{0},u_{0})(\lambda_{h},u_{h})
      - DF_{h}(\widetilde{\lambda}_{0h},\widetilde{u}_{0h})(\lambda_{h},u_{h})
      \|_{Z_{h}}
      \leq
      \eta_{2} \| (\lambda_{h},u_{h}) \|_{\mathbb{R}^{m} \times W_{h}} \, ,
         \nonumber
\end{eqnarray}
\begin{eqnarray}
   &  \textrm{there exists} \ \eta_{3} > 0 \
      \textrm{such that, for every} \
      \mathrm{z}_{h} \in \Delta_{h}, \
      \textrm{we have}
         \label{e5_45p_conditia_IV} \\
   &  \| \bar{\mathcal{B}}\mathrm{z}_{h} - \bar{\mathcal{B}}_{h}\mathrm{z}_{h}
      \|_{\mathbb{R}^{n}}
      \leq
      \eta_{3} \| \mathrm{z}_{h} \|_{\Delta_{h}} \, ,
         \nonumber
\end{eqnarray}
\begin{eqnarray}
   & \ & \textrm{there exists} \ \eta_{4} > 0 \
      \textrm{such that, for every} \
      v_{h} \in W_{h}, \
      \textrm{we have}
         \label{e5_45p_conditia_V} \\
   & \ & \| \pi_{h}^{Z}D_{u}F(\lambda_{0},u_{0})v_{h}
      - D_{u}F_{h}(\widetilde{\lambda}_{0h},\widetilde{u}_{0h})v_{h}
      \|_{Z_{h}}
      \leq
      \eta_{4} \| v_{h} \|_{W_{h}} \, .
         \nonumber
\end{eqnarray}

\begin{lem}
\label{lema_Sir_Spaces_AF1982_Cor1}
(\cite{CLBichir_bib_Sir_Spaces_AF1982}) Let $E$, $F$ be two Banach
spaces and let $T,S \in L(E,F)$. If the operator $T$ is bijective
and $\| T^{-1} \|_{L(F,E)} \| T - S \|_{L(E,F)}$ $<$ $1$, then the
operator $S$ is bijective and $\| S^{-1} \|_{L(F,E)}$ $\leq$ $(1 -
q)^{-1}$ $\| T^{-1} \|_{L(F,E)}$, $\forall q \in \mathbb{R}$ that
satisfies $\| T^{-1} \|_{L(F,E)} \| T - S \|_{L(E,F)}$ $\leq$ $q$
$<$ $1$.
\end{lem}

The following Theorem \ref{teorema_lema5_8p} and its proof are the
adaptation, to our conditions, of Theorem XIV.1.1 and of its proof
from \cite{CLBichir_bib_Ka_Ak1986}.

Let us denote $q_{G,1}$ $=$
   $C (\| DF(\lambda_{0},u_{0}) \|_{L(\mathbb{R}^{m} \times W,Z)}$
   $+$ $\sum_{i=1}^{q} \|\bar{a}_{i}\|_{Z}$
   $+$ $\| \mathcal{J} \|_{L(W,Z)})$, $q_{G,2}$ $=$
   $\eta_{1}$
   $+$ $\eta_{2}$
   $+$$C \| \mathcal{J} \|_{L(W,Z)}$, $q_{G}$ $=$ $q_{G,h}$ $=$ $(q_{G,1}$ $+$ $q_{G,2})$ $\|
\Phi_{G}(x_{0},\cdot)^{-1} \|_{L(Y,X)}$,
   $q_{H,1}$ $=$
   $C (\| D_{u}F(\lambda_{0},u_{0}) \|_{L(W,Z)}$
   $+$ $\sum_{k=1}^{n} \|\bar{b}^{k}\|_{Z}$
   $+$ $\| \mathcal{J} \|_{L(W,Z)})$,  $q_{H,2}$ $=$
   $\eta_{3}$
   $+$ $\eta_{4}$
   $+$$C \| \mathcal{J} \|_{L(W,Z)}$, $q_{H}$ $=$ $q_{H,h}$ $=$ $(q_{H,1}$ $+$ $q_{H,2})$ $\|
\Phi_{H}(x_{0},\cdot)^{-1} \|_{L(\Sigma,\Delta)}$.

\begin{thm}
\label{teorema_lema5_8p} (i) Suppose that Hypotheses
(\ref{ipotezaHypF}), (\ref{e5_8_proprietati}),
(\ref{e5_45p_conditia_izomorfism}) - (\ref{e5_45p_conditia_I})
hold. If
\begin{equation}
\label{e5_45p_conditia_q}
   q_{G} = q_{G,h} = (q_{G,1} + q_{G,2})
      \| \Phi_{G}(x_{0},\cdot)^{-1} \|_{L(Y,X)} < 1 \, ,
\end{equation}
then $\Phi_{G,h}(\widetilde{x}_{0h},\cdot)$ is an isomorphism of
$X_{h}$ onto $Y_{h}$ and
\begin{equation}
\label{e5_45p_conditia_q_estimare}
   \| \Phi_{G,h}(\widetilde{x}_{0h},\cdot)^{-1} \|_{L(Y_{h},X_{h})}
   \leq
   (1-q_{G})^{-1}
   \| \Phi_{G}(x_{0},\cdot)^{-1} \|_{L(Y,X)} \, .
\end{equation}

(ii) Suppose that Hypotheses (\ref{ipotezaHypF}),
(\ref{e5_8_proprietati}), (\ref{e5_45p_conditia_izomorfism}),
(\ref{e5_45p_conditia_IV}), (\ref{e5_45p_conditia_V}) hold. If
\begin{equation}
\label{e5_45p_conditia_q_H}
   q_{H} = q_{H,h} = (q_{H,1} + q_{H,2})
      \| \Phi_{H}(x_{0},\cdot)^{-1} \|_{L(\Sigma,\Delta)} < 1 \, ,
\end{equation}
then $\Phi_{H,h}(\widetilde{x}_{0h},\cdot)$  is an isomorphism of
$\Delta_{h}$ onto $\Sigma_{h}$ and
\begin{equation}
\label{e5_45p_conditia_q_H_estimare}
   \| \Phi_{H,h}(\widetilde{x}_{0h},\cdot)^{-1} \|_{L(\Sigma_{h},\Delta_{h})}
   \leq
   (1-q_{H})^{-1}
   \| \Phi_{H}(x_{0},\cdot)^{-1} \|_{L(\Sigma,\Delta)} \, .
\end{equation}
\end{thm}

\begin{proof} Since (\ref{ipotezaHypF}) holds, from the discussion on
(\ref{e5_1_introduction_2_1}), as in
\cite{CLBichir_bib_Cr_Ra1990}, we have that
$\Phi_{G}(x_{0},\cdot)$ is an isomorphism of $X$ onto $Y$. From
Theorem \ref{teorema_principala_parteaMAIN}(ii) and Lemma
\ref{Lema_solutie_01_DS3}, $\Phi_{H}(x_{0},\cdot)$ is an
isomorphism of $\Delta$ onto $\Sigma$.

(i) Define $\widehat{\mathcal{J}}:X \rightarrow Y$,
$\widehat{\mathcal{J}} x$ $=$
$[I_{q+m}(f,\lambda),\mathcal{J}u]^{T}$. From the hypothesis
(\ref{e5_45p_conditia_izomorfism}), there results that
$\widehat{\mathcal{J}}$ is an isomorphism of $X$ onto $Y$ such
that $\widehat{\mathcal{J}}(X_{h})$ $=$ $Y_{h}$.

Let $\mathcal{J}_{h}$ be the restriction of $\mathcal{J}$ to
$W_{h}$ and let $\widehat{\mathcal{J}}_{h}$ be the restriction of
$\widehat{\mathcal{J}}$ to $X_{h}$, $\widehat{\mathcal{J}}_{h} x$
$=$ $[I_{q+m}|_{X_{h}}(f,\lambda),\mathcal{J}_{h}u]^{T}$. We have
$\mathcal{J}_{h}(W_{h})$ $=$ $Z_{h}$ and
$\widehat{\mathcal{J}}_{h}(X_{h})$ $=$ $Y_{h}$. Let $\Pi_{h}$ $=$
$(I_{q+m}, \pi_{h}^{Z})$ $\in$ $L(Y,Y_{h})$.

Let $\mathcal{T}$ $=$ $(\mathcal{T}_{B}$, $\mathcal{T}_{G})$,
$\widetilde{\mathcal{T}}$ $=$ $(\widetilde{\mathcal{T}}_{B}$,
$\widetilde{\mathcal{T}}_{G})$ $\in$ $L(X,Y)$ and
$\widetilde{\mathcal{T}}_{h}$ $=$
$(\widetilde{\mathcal{T}}_{h,B}$,
$\widetilde{\mathcal{T}}_{h,G})$, $\mathcal{T}_{h}$ $=$
$(\mathcal{T}_{h,B}$, $\mathcal{T}_{h,G})$ $\in$ $L(X_{h},Y_{h})$
defined by
\begin{eqnarray}
   && \mathcal{T} =
   \Phi_{G}(x_{0},\cdot) =
   \widehat{\mathcal{J}} + (\Phi_{G}(x_{0},\cdot) - \widehat{\mathcal{J}}) \, ,
         \label{e5_45p_conditia_q_operatorii_T} \\
   && \widetilde{\mathcal{T}} =
   \widehat{\mathcal{J}} + \Pi_{h}(\Phi_{G}(x_{0},\cdot) - \widehat{\mathcal{J}}) \, ,
         \nonumber \\
   && \widetilde{\mathcal{T}}_{h} =
   \widehat{\mathcal{J}}_{h} + \Pi_{h}(\Phi_{G}(x_{0},\cdot) - \widehat{\mathcal{J}}_{h}) \, ,
         \nonumber \\
   && \mathcal{T}_{h} =
   \Phi_{G,h}(\widetilde{x}_{0h},\cdot) =
   \widehat{\mathcal{J}}_{h} + (\Phi_{G,h}(\widetilde{x}_{0h},\cdot) - \widehat{\mathcal{J}}_{h}) \, .
         \nonumber
\end{eqnarray}

Define the first component of the operators $\mathcal{T}$,
$\widetilde{\mathcal{T}}$, $\widetilde{\mathcal{T}}_{h}$,
$\mathcal{T}_{h}$: \\
$\mathcal{T}_{B}x$
   $=$ $Bx$
   $=$ $I_{q+m}(f,\lambda)$
   $+$ $(Bx$
   $-$ $I_{q+m}(f,\lambda))$, \\
$\widetilde{\mathcal{T}}_{B}x$
   $=$ $I_{q+m}(f,\lambda)$
   $+$ $I_{q+m}(Bx$
   $-$ $I_{q+m}(f,\lambda))$, \\
$\widetilde{\mathcal{T}}_{h,B}x$
   $=$ $I_{q+m}(f,\lambda)$
   $+$ $I_{q+m}(Bx$
   $-$ $I_{q+m}(f,\lambda))$, \\
$\mathcal{T}_{h,B}x$
   $=$ $B_{h}x$
   $=$ $I_{q+m}(f,\lambda)$
   $+$ $(B_{h}x$
   $-$ $I_{q+m}(f,\lambda))$.

Define the second component of the operators  $\mathcal{T}$,
$\widetilde{\mathcal{T}}$,
$\widetilde{\mathcal{T}}_{h}$, $\mathcal{T}_{h}$: \\
$\mathcal{T}_{G}x$
   $=$ $DG(x_{0})x$
   $=$ $\mathcal{J}u$
   $+$ $(DF(\lambda_{0},u_{0})(\lambda,u)$
   $-$ $\sum_{i=1}^{q}f^{i}\bar{a}_{i}$
   $-$ $\mathcal{J}u)$, \\
$\widetilde{\mathcal{T}}_{G}x$
   $=$ $\mathcal{J}u$
   $+$ $\pi_{h}^{Z}DF(\lambda_{0},u_{0})(\lambda,u)$
   $-$ $\pi_{h}^{Z}(\sum_{i=1}^{q}f^{i}\bar{a}_{i})$
   $-$ $\pi_{h}^{Z}\mathcal{J}u$, \\
$\widetilde{\mathcal{T}}_{h,G}x$
   $=$ $\mathcal{J}_{h}u$
   $+$ $\pi_{h}^{Z}DF(\lambda_{0},u_{0})(\lambda,u)$
   $-$ $\sum_{i=1}^{q}f^{i}\bar{a}_{ih}$
   $-$ $\pi_{h}^{Z}\mathcal{J}_{h}u$, \\
$\mathcal{T}_{h,G}x$
   $=$ $DG_{h}(\widetilde{x}_{0h})x$
   $=$ $\mathcal{J}_{h}u$
   $+$ $(DF_{h}(\widetilde{\lambda}_{0h},\widetilde{u}_{0h})(\lambda,u)$
   $-$ $\sum_{i=1}^{q}f^{i}\bar{a}_{ih}$
   $-$ $\mathcal{J}_{h}u)$, \\
where $\pi_{h}^{Z}(\sum_{i=1}^{q}f^{i}\bar{a}_{i})$ $=$
$\sum_{i=1}^{q}f^{i}\pi_{h}^{Z}\bar{a}_{i}$ $=$
$\sum_{i=1}^{q}f^{i}\bar{a}_{ih}$.

Let $\bar{x}$ $\in$ $X$. We denote $\mathcal{L}_{G,0}\bar{x}$
   $=$ $DF(\lambda_{0},u_{0})(\bar{\lambda},\bar{u})$
   $-$ $\sum_{i=1}^{q}\bar{f}^{i}\bar{a}_{i}$
   $-$ $\mathcal{J}\bar{u}$. We have, using
   (\ref{e5_8_proprietati}), that
$\| (\mathcal{T} - \widetilde{\mathcal{T}})\bar{x} \|_{Y}$
   $\leq$
   $\| (I_{Z}-\pi_{h}^{Z})\mathcal{L}_{G,0}\bar{x} \|_{Z}$
   $\leq$
   $C \| \mathcal{L}_{G,0}\bar{x} \|_{Z}$.
Using (\ref{e5_45p_conditia_q}), there results
\begin{equation}
\label{trei_ast}
   \| \mathcal{T} - \widetilde{\mathcal{T}} \|_{L(X,Y)}
   =  \sup_{x \in X, \| x \|_{X} \leq 1}
   \| (\mathcal{T} - \widetilde{\mathcal{T}})x \|_{Y}
   \leq q_{G,1} < \| \mathcal{T}^{-1} \|_{L(Y,X)}^{-1} \, .
\end{equation}
Applying Lemma \ref{lema_Sir_Spaces_AF1982_Cor1}, since
$\mathcal{T}$ $=$ $\Phi_{G}(x_{0},\cdot)$ is an isomorphism of $X$
onto $Y$, there results that $\widetilde{\mathcal{T}}$ is an
isomorphism of $X$ onto $Y$ and
\begin{equation}
\label{trei_ast_CONT}
   \| \widetilde{\mathcal{T}}^{-1} \|_{L(Y,X)}
   \leq
   (1-q_{G,1} \| \mathcal{T}^{-1} \|_{L(Y,X)})^{-1}
   \| \mathcal{T}^{-1} \|_{L(Y,X)} \, .
\end{equation}

If $x_{h}'$ $=$ $\widetilde{\mathcal{T}}^{-1}\zeta_{h}$, then
$\widetilde{\mathcal{T}}x_{h}'$ $=$ $\zeta_{h}$ or
$\widehat{\mathcal{J}} x_{h}'$ $+$
$\Pi_{h}(\Phi_{G}(x_{0},x_{h}')$ $-$ $\widehat{\mathcal{J}}
x_{h}')$ $=$ $\zeta_{h}$ or $\widehat{\mathcal{J}} x_{h}'$ $=$
$\zeta_{h}$ $-$ $\Pi_{h}(\Phi_{G}(x_{0},x_{h}')$ $-$
$\widehat{\mathcal{J}} x_{h}')$, where $\zeta_{h}$ $\in$ $Y_{h}$
and $\Pi_{h}(\Phi_{G}(x_{0},x_{h}')$ $-$ $\widehat{\mathcal{J}}
x_{h}')$ $\in$ $Y_{h}$. So $x_{h}'$ $=$
$\widehat{\mathcal{J}}^{-1}(\zeta_{h}$ $-$
$\Pi_{h}(\Phi_{G}(x_{0},x_{h}')$ $-$ $\widehat{\mathcal{J}}
x_{h}'))$ $\in$ $X_{h}$, since $\widehat{\mathcal{J}}(X_{h})$ $=$
$Y_{h}$. There results that the operator $\widetilde{\mathcal{T}}$
has the property that if $\zeta_{h}$ $\in$ $Y_{h}$, then
$\widetilde{\mathcal{T}}^{-1}\zeta_{h}$ $\in$ $X_{h}$.

Let us consider the operator $\widetilde{\mathcal{T}}_{h}$. We
have: $\forall$ $x_{h}$ $\in$ $X_{h}$,
$\widetilde{\mathcal{T}}_{h}x_{h}$ $=$
$\widetilde{\mathcal{T}}x_{h}$. It follows that the operator
$\widetilde{\mathcal{T}}_{h}$ has a continuous inverse that
coincides with $\widetilde{\mathcal{T}}^{-1}$ on $Y_{h}$ and
\begin{equation}
\label{trei_ast_RESTRICTIE}
   \| \widetilde{\mathcal{T}}_{h}^{-1} \|_{L(Y_{h},X_{h})}
   \leq
   \| \widetilde{\mathcal{T}}^{-1} \|_{L(Y,X)} \, .
\end{equation}
It follows that $\widetilde{\mathcal{T}}_{h}$ is an isomorphism of
$X_{h}$ onto $Y_{h}$.

Let $x_{h}$ be an arbitrary element in $X_{h}$. We have $\|
(\widetilde{\mathcal{T}}_{h} - \mathcal{T}_{h})x_{h} \|_{Y_{h}}$
$=$ $\|(B|_{X_{h}} - B_{h})x_{h} \|_{\mathbb{R}^{q+m}}$ $+$ $\|
[\pi_{h}^{Z}DF(\lambda_{0},u_{0})
      -
      DF_{h}(\widetilde{\lambda}_{0h},\widetilde{u}_{0h})](\lambda_{h},u_{h})
      + (I_{Z}-\pi_{h}^{Z})\mathcal{J}_{h}u_{h} \|_{Z_{h}}$.
Using (\ref{e5_8_proprietati}), (\ref{e5_45p_conditia_III}) and
(\ref{e5_45p_conditia_I}), there results
\begin{equation}
\label{patru_ast}
   \| \widetilde{\mathcal{T}}_{h}
   - \mathcal{T}_{h} \|_{L(X_{h},Y_{h})}
   =  \sup_{x_{h} \in X_{h}, \| x_{h} \|_{X_{h}} \leq 1}
   \| (\widetilde{\mathcal{T}}_{h}
   - \mathcal{T}_{h})x_{h} \|_{Y_{h}}
   \leq q_{G,2}  \, .
\end{equation}

We have $\| \widetilde{\mathcal{T}}_{h}^{-1} \|_{L(Y_{h},X_{h})}
\| \widetilde{\mathcal{T}}_{h} - \mathcal{T}_{h}
\|_{L(X_{h},Y_{h})}$ $\leq$ $q_{G,2} \|
\widetilde{\mathcal{T}}_{h}^{-1} \|_{L(Y_{h},X_{h})}$ $\leq$
$q_{G,2} \| \widetilde{\mathcal{T}}^{-1} \|_{L(Y,X)}$ $\leq$ $r$
$<$ $1$, where $r = (1-q_{G,1} \| \mathcal{T}^{-1}
\|_{L(Y,X)})^{-1} q_{G,2} \| \mathcal{T}^{-1} \|_{L(Y,X)}$, using
(\ref{trei_ast_RESTRICTIE}), (\ref{trei_ast_CONT}) and
(\ref{e5_45p_conditia_q}). Applying Lemma
\ref{lema_Sir_Spaces_AF1982_Cor1}, since
$\widetilde{\mathcal{T}}_{h}$ is an isomorphism of $X_{h}$ onto
$Y_{h}$, it follows that $\mathcal{T}_{h}$ $=$
$\Phi_{G,h}(\widetilde{x}_{0h},\cdot)$ is an isomorphism of
$X_{h}$ onto $Y_{h}$ and
\begin{equation*}
\label{patru_ast_CONT}
   \| \mathcal{T}_{h}^{-1} \|_{L(Y_{h},X_{h})}
   \leq
   (1-r)^{-1}
   \| \widetilde{\mathcal{T}}_{h}^{-1} \|_{L(Y_{h},X_{h})} \, ,
\end{equation*}
whence, using (\ref{trei_ast_RESTRICTIE}) and
(\ref{trei_ast_CONT}), we deduce
\begin{equation*}
\label{patru_ast_CONT}
   \| \mathcal{T}_{h}^{-1} \|_{L(Y_{h},X_{h})}
   \leq
   [1-(q_{G,1} + q_{G,2})
      \| \Phi_{G}(x_{0},\cdot)^{-1} \|_{L(Y,X)}]^{-1}
   \| \mathcal{T}^{-1} \|_{L(Y,X)} \, ,
\end{equation*}
so (\ref{e5_45p_conditia_q_estimare}).

\qquad
\end{proof}

\begin{cor}
\label{corolarul_teorema_lema5_8p} Let us adjoin the index $h$ to
$C$ and $\eta_{k}$ in order to indicate the dependence on $h$.
Assume that
\begin{equation}
\label{e_A2_38_TEXT_G_tilde1}
   \lim_{h \rightarrow 0} C_{h}
      =0 \, , \
   \lim_{h \rightarrow 0} \eta_{k,h}
      =0 \, , \
   k=1,2,3,4 \, .
\end{equation}
Then, there exists a real $h_{0} > 0$ such that for all $h$, $h
\leq h_{0}$, $q_{G,h} < 1$ and $q_{H,h} < 1$.
\end{cor}

\subsection{Existence of an approximate bifurcation problem}
\label{subsectiune_din_cap_pr_aprox}

\begin{thm}
\label{teorema_principala_spatii_infinit_dimensionale_widetilde_s_3_0_exact_h}
Assume that the exact problem (\ref{e5_1}) has a bifurcation point
$(\lambda_{0},u_{0})$ that satisfies hypothesis
(\ref{ipotezaHypF}). Assume that, for some fixed $h$, the
hypotheses of Theorem \ref{teorema_lema5_8p} and of Theorem
\ref{teorema_principala_spatii_infinit_dimensionale_widetilde_s_3_0_exact_inf}
for the approximate case are satisfied. Then, there exists an
approximate equation
(\ref{e5_1_sol_widetilde_x_0h_Inv_Fc_Th_ec_DATA_introd_exact_th_h}),
\begin{equation}
\label{e5_1_sol_widetilde_x_0h_Inv_Fc_Th_ec_DATA_introd_exact_th_h}
      F_{h}(\lambda_{h},u_{h})-\varrho_{h}
      = 0 \, ,
\end{equation}
which is of the form of
(\ref{e5_1_sol_widetilde_x_0h_Inv_Fc_Th_ec_DATA_introd_exact}).
The solution $(\lambda_{0h},u_{0h})$ of
(\ref{e5_1_sol_widetilde_x_0h_Inv_Fc_Th_ec_DATA_introd_exact_th_h})
has the same type as the solution $(\lambda_{0},u_{0})$ of
equation (\ref{e5_1}). $(\lambda_{0h},u_{0h})$ satisfies the
hypothesis (\ref{ipotezaHypF}). The radii $a_{h}^{\ast}$ and
$b_{h}^{\ast}$ depend on $h$. Theorem
\ref{teorema_principala_spatii_infinit_dimensionale_widetilde_s_3_0_exact_inf_echiv_G}
holds, that is, there exists a class of maps $C^{p}$ - equivalent
(right equivalent) at $(\widehat{\lambda}_{0h},\widehat{u}_{0h})$
$=$ $(\lambda_{0h},u_{0h})$ to $F_{h}(\cdot)-\varrho_{h}$ at
$(\lambda_{0h},u_{0h})$ and that satisfies the hypothesis
(\ref{ipotezaHypF}) in $(\widehat{\lambda}_{0h},\widehat{u}_{0h})$
$=$ $(\lambda_{0h},u_{0h})$.
\end{thm}

\begin{proof}

The construction related to $(\lambda_{0},u_{0})$, from
\cite{CLBichir_bib_Cr_Ra1990}, presented in Section
\ref{sectiunea_1_preliminaries}, and the construction from Section
\ref{sectiunea_1_main_theorem_on_extended_systems} lead to the
statement (i) of Theorem \ref{teorema_principala_parteaMAIN} and,
then, equivalently, to the statement (ii) of Theorem
\ref{teorema_principala_parteaMAIN}.

Using Lemma \ref{Lema_solutie_01_DS3}, there results that
$\Phi_{G}(x_{0},\cdot)$ $=$ $D\Psi(x)$ is an isomorphism of $X$
onto $Y$, $\Phi_{H}(x_{0},\cdot)$ is an isomorphism of $\Delta$
onto $\Sigma$.

Using Theorem \ref{teorema_lema5_8p} and Lemma
\ref{Lema_solutie_01_DS3}, we obtain that
$D\Psi_{h}(\widetilde{x}_{0h})$ is an isomorphism of $X_{h}$ onto
$Y_{h}$ and $DS_{h}(\widetilde{s}_{0,h})$ is an isomorphism of
$\Gamma_{h}$ onto $\Sigma_{h}$.

We now apply Theorem
\ref{teorema_principala_spatii_infinit_dimensionale_widetilde_s_3_0_exact_inf}
for the case of the approximate formulation.

\qquad
\end{proof}

In the following Theorem, we formulate conditions similar to those
from Theorem IV.3.2, page 304, and Theorem IV.3.7, page 312,
\cite{CLBichir_bib_Gir_Rav1986}, and Corollary 3.1, page 52,
\cite{CLBichir_bib_Cr_Ra1990}.

\begin{thm}
\label{teorema_principala_spatii_infinit_dimensionale_widetilde_s_3_0_lim}
Assume that the exact problem has a bifurcation point that
satisfies the statement (i) of Theorem
\ref{teorema_principala_parteaMAIN}. Assume the hypotheses of
Corollary \ref{corolarul_teorema_lema5_8p}. Let $\alpha$ be an
arbitrarily small fixed positive number, $0 < \alpha < 1$, $\alpha
\neq \frac{1}{2}$. Let $h_{0}$ be the real from Corollary
\ref{corolarul_teorema_lema5_8p}. Consider, for each $h$, $h \leq
h_{0}$, the mappings $\mathcal{G}_{h}$ and $\mathcal{Q}_{h}$ given
by
(\ref{e5_57_forma2_sistem_REG_2_GGGvarianta_infinit_dimensionale_DEM})
and
(\ref{e5_57_forma2_sistem_REG_2_GGGmic_infinit_dimensionale_DEM}).

For all $h \leq h_{0}$, we take $\kappa_{h} = reg \,
\mathcal{A}_{h}$ and $M_{h}$ $>$ $\|
D\mathcal{G}_{h}(\widetilde{s}_{0,h},\widetilde{\phi}_{0,h}') \|$.
Let $L_{h}(\varepsilon)$ $=$
$\widetilde{L}(\mathcal{G}_{h},(\widetilde{s}_{0,h},\widetilde{\phi}_{0,h}'),(s_{h},\phi_{h}'),\varepsilon,\Gamma_{h}
\times \Gamma_{h},\Sigma_{h})$, where we use
(\ref{e_A2_17_TEXT_mu_bar_gen}). Define $\kappa = \sup_{h \leq
h_{0}} \kappa_{h}$, $c_{h}$ $=$ $\frac{1-\kappa
L_{h}(\varepsilon)}{\kappa}$, $\widehat{a}$ $=$ $\sup_{h \leq
h_{0}} \max \{$ $\|a_{1h}\|,\ldots,\|a_{qh}\|$,
$\|b_{1h}\|,\ldots,\|b_{nh}\|$ $\}$, $M = \sup_{h \leq h_{0}}
M_{h}$ and  $\delta_{h}$ $=$ $\|
\mathcal{G}_{h}(\widetilde{s}_{0,h},\widetilde{y}_{0h}') \|$.

Assume that
\begin{equation}
\label{e_A2_38_TEXT_G}
   \lim_{h \rightarrow 0} \delta_{h}
      =0 \, ,
\end{equation}
\begin{equation}
\label{th_A2_2_h_4_var_TEXT_G}
   \lim_{\beta \rightarrow 0} (\sup_{h \leq h_{0}}
      L_{h}(\beta)) = 0 \, .
\end{equation}

Assume that $D\Psi_{h}(\widetilde{x}_{0h})$ is an isomorphism of
$X_{h}$ onto $Y_{h}$ and $DS_{h}(\widetilde{s}_{0,h})$ is an
isomorphism of $\Gamma_{h}$ onto $\Sigma_{h}$.

Then, there exists $h_{1} > 0$, $h_{1} \leq h_{0}$, such that
$\forall h \leq h_{1}$, Theorem
\ref{teorema_principala_spatii_infinit_dimensionale_widetilde_s_3_0_exact_inf}
and Theorem
\ref{teorema_principala_spatii_infinit_dimensionale_widetilde_s_3_0_exact_h}
are applied with $\varepsilon$, $\tau$ and $a^{\ast}$ $=$
$a_{h}^{\ast}$ that do not depend on $h$ (they are constants). The
reals $\tau$, $a^{\ast}$, $b^{\ast}$ and the condition
(\ref{dem_e_A2_38_TEXT_G_estimare_M_a_b_h_conditie_inf_th_exact})
from Theorem
\ref{teorema_principala_spatii_infinit_dimensionale_widetilde_s_3_0_exact_inf}
are given by $\tau$, $a_{h}^{\ast}$, $b_{h}^{\ast}$, from
(\ref{dem_e_A2_38_TEXT_G_estimare_exact_tau_DOI}) -
(\ref{dem_e_A2_38_TEXT_G_estimare_exact_a_h_ast_h_DOI_inf_min_h})
for $\mathcal{G}_{h}$, and
(\ref{dem_e_A2_38_TEXT_G_estimare_M_a_b_h_conditie_cor})
respectively,
\begin{equation}
\label{dem_e_A2_38_TEXT_G_estimare_M_a_b_h_conditie_cor}
   \delta_{h} \leq \frac{1}{2} \cdot \frac{1}{2 \kappa} \cdot a_{h}^{\ast}
   < \frac{1}{2} \cdot b_{h}^{\ast}
      \ , \ \forall h \leq h_{1} \, .
\end{equation}

Let $s_{0h}$ be the solution of the equation (\ref{e5_56_b})
written in this case for a fixed index $h$. We
have
\begin{equation}
\label{e5_49_th_h_dif}
   \| (\lambda_{0},u_{0}) - (\lambda_{0h},u_{0h}) \| \leq
      C_{0}\| s_{0}\|
      + [\gamma_{h} /(1-\gamma_{h} L_{S_{h}}(a)] \cdot
      \| S_{0,h}(\widetilde{s}_{0,h}) \|_{\Sigma_{h}} \, ,
\end{equation}
where $S_{0,h}$ is $S_{0}$ in the approximate case. $C_{0}$
results using (\ref{e5_8_proprietati}) for $\| s_{0} -
\widetilde{s}_{0,h} \|$ $\leq$ $C_{0}\| s_{0}\|$.
\end{thm}

\textbf{Proof of Theorem
\ref{teorema_principala_spatii_infinit_dimensionale_widetilde_s_3_0_lim}}

\begin{proof}

(\ref{th_A2_2_h_4_var_TEXT_G}) implies that we can take an
$\varepsilon > 0$ such that
\begin{equation}
\label{dem_e_A2_37_TEXT_G_prod_alphaastastast}
   2 \kappa L_{h}(\varepsilon) + 2 \kappa \alpha \widehat{a} < 1
      \ , \ \forall h \leq h_{0} \, ,
\end{equation}
that is, (\ref{dem_e_A2_37_TEXT_G_prod_alphaastastast_var_exact}),
$\forall h \leq h_{0}$.

Remark that $\varepsilon$ do not depend on $h$ (for $h \leq
h_{0}$).

For each $c_{h}$, we have
(\ref{dem_e_A2_37_TEXT_G_prod_estimare_lemma}).

(\ref{e_A2_38_TEXT_G}) assures that condition
(\ref{dem_e_A2_38_TEXT_G_estimare_M_a_b_h_conditie_inf_th_exact})
is satisfied. Then, we apply Theorem
\ref{teorema_principala_spatii_infinit_dimensionale_widetilde_s_3_0_exact_inf}.

We use (\ref{dem_e_A2_38_TEXT_G_estimare_exact_tau_DOI}) -
(\ref{dem_e_A2_38_TEXT_G_estimare_exact_a_h_ast_h_DOI_inf_min_h})
from Lemma
\ref{Lema_spatii_infinit_dimensionale_th_Graves_estimare_raze_vecinatati_ast}
(ii) applied to $\mathcal{G}_{h}$. $\varepsilon$, $\tau$ and
$a_{h}^{\ast}$ do not depend on $h$ (they are constants).

Let us remark that, for condition I from Theorem
\ref{teorema_contraction_mapping_principle_Dontchev_Rockafellar2009},
using (\ref{dem_e_A2_37_TEXT_G_prod_estimare_lemma}), we propose
that
\begin{equation}
\label{dem_e_A2_38_TEXT_G_estimare_M_delta_h_conditie_ii}
   \frac{1}{c_{h}}\delta_{h}
   < 2 \kappa \delta_{h}
   \leq \frac{1}{2}a^{\ast}
   < a^{\ast} (1-\lambda_{h}) \, .
\end{equation}
From (\ref{dem_e_A2_38_TEXT_G_estimare_M_delta_h_conditie_ii}),
for condition I, we impose
(\ref{dem_e_A2_38_TEXT_G_estimare_M_a_b_h_conditie_cor}).

We have

$\| (\lambda_{0},u_{0}) - (\lambda_{0h},u_{0h}) \|$ $\leq$ $\|
s_{0} - s_{0h} \|$ $\leq$ $\| s_{0} - \widetilde{s}_{0,h} \|$ $+$
$\| \widetilde{s}_{0,h} - s_{0h} \|$ $\leq$ $C_{0}\| s_{0}\|$ $+$
$\| \widetilde{s}_{0,h} - s_{0h} \|$. We use (\ref{e5_49_th}). We
obtain (\ref{e5_49_th_h_dif}).

\qquad
\end{proof}

\begin{lem}
\label{lema5_teorema_principala_spatii_infinit_dimensionale_epsilon_h}
Let $\delta_{h}$ $=$ $\|
\mathcal{G}_{h}(\widetilde{s}_{0,h},\widetilde{y}_{0h}') \|$ $=$
$\| S_{h}(\widetilde{s}_{0,h}) \|_{\Sigma_{h}}$. If
$\Psi_{h}(\widetilde{x}_{0h})=0$ in $\delta_{h}$, then,
(\ref{e_A2_38_TEXT_G}) is equivalent to the following conditions,
for $i=1,\ldots,q+m$, $k=1,\ldots,n$,
\begin{equation*}
\label{e_A2_38_TEXT_G_fi_G_fi_H}
   \lim_{h \rightarrow 0} (\Phi_{G,h}(\widetilde{x}_{0h},\widetilde{y}_{i,0,h})-[\delta_{i}^{q+m},0]^{T})
      =0 \, ,  \
   \lim_{h \rightarrow 0} (\Phi_{H,h}(\widetilde{x}_{0h},\widetilde{\mathrm{z}}_{k,0,h})-[\delta_{k}^{n},0]^{T})
      =0 \, ,
\end{equation*}
\end{lem}

\begin{cor}
\label{corolarul5_1_teorema_lema5_8p_studiul_pr_fin_dim_q_n}
$q_{h}$ $=$ $q$, $n_{h}$ $=$ $n$, where the index "h" indicates
the approximate case.
\end{cor}

\begin{proof} $q+m$ and $n$ are fixed by $B$ and
$\bar{\mathcal{B}}$.

\qquad
\end{proof}

Corollary
\ref{corolarul_trei_3_teorema_principala_spatii_infinit_dimensionale_widetilde_s_3_0_exact_zero}
gives:
\begin{cor}
\label{corolarul_trei_3_teorema_principala_spatii_infinit_dimensionale_widetilde_s_3_0_exact_zero_h}
Assume that $(\widetilde{\lambda}_{0h},\widetilde{u}_{0h})$
belongs to a solution branch of equation (\ref{e5_9_introd}).
$(\widetilde{\lambda}_{0},\widetilde{u}_{0})$ can be a regular or
a nonregular solution. Assume the hypotheses of Theorem
\ref{teorema_principala_spatii_infinit_dimensionale_widetilde_s_3_0_exact_h}
or of Theorem
\ref{teorema_principala_spatii_infinit_dimensionale_widetilde_s_3_0_lim}.
$\Psi_{h}(\widetilde{x}_{0h})=0$ in $\delta_{h}$. If $\varrho_{h}$
$\neq$ $0$, then, the given problem (\ref{e5_9_introd}) is a
perturbation of the bifurcation problem
(\ref{e5_1_sol_widetilde_x_0h_Inv_Fc_Th_ec_DATA_introd_exact_th_h})
(or of (\ref{e5_1_sol_widetilde_x_0h_Inv_Fc_Th_ec_DATA_introd})).
If $\varrho_{h}$ $=$ $0$, then, the bifurcation point
$(\lambda_{0h},u_{0h})$ belongs to the solution branch of equation
(\ref{e5_9_introd}).
\end{cor}

\begin{rem}
\label{observatia5_1_pc_bif_perturb} We proved that there exists
an approximate bifurcation problem
(\ref{e5_1_sol_widetilde_x_0h_Inv_Fc_Th_ec_DATA_introd_exact_th_h})
(or (\ref{e5_1_sol_widetilde_x_0h_Inv_Fc_Th_ec_DATA_introd})) that
preserves the type of the bifurcation point of (\ref{e5_1}). The
given problem (\ref{e5_9_introd}) is a perturbation of
(\ref{e5_1_sol_widetilde_x_0h_Inv_Fc_Th_ec_DATA_introd_exact_th_h})
(or (\ref{e5_1_sol_widetilde_x_0h_Inv_Fc_Th_ec_DATA_introd}))
(when $\| \varrho_{h} \|$ is small enough.).
\end{rem}

\section{The Dirichlet problem for the stationary Navier-Stokes equations}
\label{sectiunea04_cazul_ecNS}

In the particular case of the stationary Navier-Stokes equations,
we maintain the position from Section
\ref{sectiunea_0_introduction}, where we state that we do not
discuss if the exact bifurcation point is generic or not. As a
consequence of the results from Section \ref{sectiunea06}, in a
certain configuration, if the model of a stationary flow has a
generic or an ungeneric bifurcation point, then there exists (at
least) one approximate equation that has a bifurcation point of
the same type as the exact model. The hypothesis of the existence
of an ungeneric bifurcation point cannot be excluded, see
\cite{CLBichir_bib_Foias_Temam1978, CLBichir_bib_Kielhofer2012,
CLBichir_bib_Temam1995}. For example, we do not exclude
transcritical bifurcation and nonsymmetric pitchfork bifurcation
from our discussion. If they exit, then they are regained in the
approximate case by some perturbed approximate equation of the
form
(\ref{e5_1_sol_widetilde_x_0h_Inv_Fc_Th_ec_DATA_introd_exact_th_h})
(or (\ref{e5_1_sol_widetilde_x_0h_Inv_Fc_Th_ec_DATA_introd})). For
this discussion, other references are mentioned in Section
\ref{sectiunea_0_introduction}.

Here, we show a modality to place the Dirichlet problem for the
stationary Navier-Stokes equations in the framework of the
preceding sections.

\subsection{The setting from \cite{CLBichir_bib_Gir_Rav1986}}
\label{sectiunea04_cazul_ecNS_1}

In order to apply the results for the case of the Dirichlet
problem for the stationary Navier-Stokes equations, formulated in
primitive variables, approximated by finite element method (with
discontinuous pressure), we use the setting of this problem from
Section IV.4.1, \cite{CLBichir_bib_Gir_Rav1986}, in the framework
of Section IV.3.3, \cite{CLBichir_bib_Gir_Rav1986}. We only
indicate the connection to the problem from Section IV.4.1,
\cite{CLBichir_bib_Gir_Rav1986}.

Let $\Omega$ be a bounded and connected open subset of
$\mathbb{R}^{N}$ ($N=2,3$) with a Lipschitz - continuous boundary
$\partial \Omega$. In the sequel, $\textbf{u}$ is the velocity,
$p$ is the kinematic pressure and $\nu$ is the kinematic
viscosity. We take $N=2,3$, $\mathcal{X}$ $=$
$H_{0}^{1}(\Omega)^{N}$ $\times$ $L_{0}^{2}(\Omega)$,
$\mathcal{Y}$ $=$ $H^{-1}(\Omega)^{N}$. $\lambda =1 / \nu > 0$.
$\lambda$ is the bifurcation parameter, $\textbf{u} =
(u_{1},\ldots,u_{N}) \in H_{0}^{1}(\Omega)^{N}$, $\textbf{x} =
(x_{1},\ldots,x_{N}) \in \Omega$. Let $\mathcal{T}_{S} \in
L(\mathcal{Y},\mathcal{X})$ be the Stokes operator that associates
to $\textbf{f} \in \mathcal{Y}$ the solution
$(\textbf{u},p)=\mathcal{T}_{S}\textbf{f}$ of the homogeneous
Stokes problem,
\begin{eqnarray}
   & \ & - \triangle \textbf{u} +
            \textbf{grad} \, p =
            \textbf{f} \; \textrm{in} \; \Omega \, ,
         \nonumber \\
   & \ & div \, \textbf{u} = 0 \; \textrm{in} \; \Omega \, ,
         \label{e1_9_ecStokes} \\
   & \ & \textbf{u} = 0 \; \textrm{on} \; \partial \Omega \, .
         \nonumber
\end{eqnarray}
The functions $F$ and $G$, from Sections IV.3.1 and IV.4.1,
\cite{CLBichir_bib_Gir_Rav1986}, are $F_{NS}$ and $G_{NS}$, where
$F_{NS}:(0,\infty) \times \mathcal{X} \rightarrow \mathcal{X}$,
$F_{NS}(\lambda,v) = v + \mathcal{T}_{S}G_{NS}(\lambda,v)$, and
$G_{NS}:(0,\infty) \times \mathcal{X} \rightarrow \mathcal{Y}$,
$G_{NS}(\lambda,v) = \lambda \cdot (\sum_{j=1}^{N} v_{j} (\partial
\textbf{v}/\partial x_{j}) - \textbf{f})$ for $v =
(\textbf{v},q)$. $G_{NS}$ is of class $C^{\infty}$ and has bounded
derivatives of all order on all bounded subsets of $\mathcal{X}$.
For $(\textbf{u},p), (\overline{\textbf{u}},\overline{p}),
(\overline{\overline{\textbf{u}}},\overline{\overline{p}})$ $\in$
$\mathcal{X}$, we have
$D_{(\textbf{u},p)}G_{NS}(\lambda,\textbf{u},p)(\overline{\textbf{u}},\overline{p})=\lambda
\cdot \sum_{j=1}^{N}(u_{j} (\partial
\overline{\textbf{u}}/\partial x_{j}) + \overline{u}_{j} (\partial
\textbf{\textbf{u}}/\partial x_{j}))$,
$D_{(\textbf{u},p)(\textbf{u},p)}^{2}G_{NS}(\lambda,\textbf{u},p)
((\overline{\textbf{u}},\overline{p}),(\overline{\overline{\textbf{u}}},\overline{\overline{p}}))$
$=$ $\lambda \cdot \sum_{j=1}^{N}(\overline{\overline{u}}_{j}
(\partial \overline{\textbf{u}}/\partial x_{j}) + \overline{u}_{j}
(\partial \overline{\overline{\textbf{u}}}/\partial x_{j}))$.

Let us fix $\lambda$. $(\textbf{u},p)$ is a solution of the
homogeneous Navier-Stokes problem if and only if $u$ $=$
$(\textbf{u},\lambda p)$ is a solution of $F_{NS}(\lambda,u)=0$
(\cite{CLBichir_bib_Gir_Rav1986}).

Let $h$ be a positive parameter tending to zero. For each $h$, let
$\Gamma_{h}$ and $E_{h}$ be two finite-dimensional spaces such
that$\Gamma_{h} \subset {H^{1}(\Omega)}^{N}$, $E_{h} \subset
L^{2}(\Omega)$ and assume that $E_{h}$ contains the constant
functions.

Let $\Gamma_{0h}$ $=$ $\Gamma_{h} \cap H_{0}^{1}(\Omega)^{N}$,
$M_{h}$ $=$ $E_{h} \cap L_{0}^{2}(\Omega)$, $\mathcal{X}_{h}$ $=$
$\Gamma_{0h}$ $\times$ $M_{h}$. Assume:

(a) There exists an operator $r_{h} \in L([H^{2}(\Omega) \cap
H_{0}^{1}(\Omega)]^{N},\Gamma_{0h})$ and an integer $\ell$ such
that
\begin{equation}
\label{HypH1}
   \| \textbf{v}-r_{h}\textbf{v} \|_{1,\Omega} \leq C_{r} \cdot h^{m} \cdot \| \textbf{v} \|_{m+1,\Omega}, \
              \forall \textbf{v} \in H^{m+1}(\Omega)^{N}, \ 1 \leq m \leq \ell \, .
\end{equation}

(b) There exists an operator $s_{h} \in L(L^{2}(\Omega),E_{h})$
such that
\begin{equation}
\label{HypH2}
   \| q-s_{h}q \|_{0,\Omega} \leq C_{s} \cdot h^{m} \cdot \| q \|_{m,\Omega}, \
              \forall q \in H^{m}(\Omega), \ 0 \leq m \leq \ell \, .
\end{equation}

These assumptions (a) and (b) together with the the uniform
inf-sup condition are the hypotheses $H_{1}$, $H_{2}$ and $H_{3}$
for the approximation from \cite{CLBichir_bib_Gir_Rav1986}.
Equation $F_{NS}(\lambda,u)=0$ is approximated by
$\mathcal{F}_{NS,h}(\lambda,u_{h})=0$
(\cite{CLBichir_bib_Gir_Rav1986}), $u_{h}$ $=$
$(\textbf{u}_{h},\lambda p_{h})$, where
$\mathcal{F}_{NS,h}:\mathbb{R} \times \mathcal{X}_{h} \rightarrow
\mathcal{X}_{h}$, $\mathcal{F}_{NS,h}(\lambda,u_{h}) = u_{h} +
\mathcal{T}_{S,h}G_{NS}(\lambda,u_{h})$ and $\mathcal{T}_{S,h} \in
L(\mathcal{Y},\mathcal{X}_{h})$ is the approximate Stokes operator
that associates to $\textbf{f} \in \mathcal{Y}$ the solution
$(\textbf{u}_{h},p_{h}) = \mathcal{T}_{S,h}\textbf{f}$ of the
finite element method approximation of problem
(\ref{e1_9_ecStokes}) (\cite{CLBichir_bib_Gir_Rav1986}),
\begin{eqnarray}
   && (grad \, \textbf{u}_{h},grad \, \textbf{w}_{h})-(p_{h},div \, \textbf{w}_{h})=\langle \textbf{f},\textbf{w}_{h} \rangle, \
              \forall \, \textbf{w}_{h}  \, \in \, \Gamma_{0h} \, ,
         \label{e1_9_ecStokes_h_1} \\
   && (div \, \textbf{u}_{h},\mu_{h})=0, \
              \forall \, \mu_{h}  \, \in \, M_{h} \, .
         \label{e1_9_ecStokes_h_2}
\end{eqnarray}

According to Theorem II.1.8, page 125, and to the proof of Theorem
IV.4.1, page 317, \cite{CLBichir_bib_Gir_Rav1986}, we have:
\begin{equation}
\label{e1_9_ecStokes_h_Gir_Rav_lim_f}
   \lim_{h \rightarrow 0} \| (\mathcal{T}_{S}-\mathcal{T}_{S,h})\textbf{f} \|_{\mathcal{X}}=0 \, ,
   \ \forall  \, \textbf{f} \in \mathcal{Y} \, .
\end{equation}

According to the proof of Theorem IV.4.1, page 317,
\cite{CLBichir_bib_Gir_Rav1986}, we have:
\begin{equation}
\label{e1_9_ecStokes_h_Gir_Rav_lim_f_L_3_2}
   \sum_{j=1}^{N}(u_{j} (\partial
\overline{\textbf{u}}/\partial x_{j}) + \overline{u}_{j} (\partial
\textbf{\textbf{u}}/\partial x_{j})) \in L^{3/2}(\Omega)^{N} \, ,
\end{equation}
where $\textbf{u}$, $\overline{\textbf{u}}$ $\in$
$H_{0}^{1}(\Omega)^{N}$, and, for $\mathcal{Z}$ $=$
$L^{3/2}(\Omega)^{N}$ $\subset$ $\mathcal{Y}$,
\begin{equation}
\label{e1_9_ecStokes_h_Gir_Rav_lim}
   \lim_{h \rightarrow 0} \| \mathcal{T}_{S}-\mathcal{T}_{S,h} \|_{L(\mathcal{Z},\mathcal{X})}=0 \, .
\end{equation}

\subsection{The supplementary variabile $q$}
\label{sectiunea04_cazul_ecNS_3}

Let us write the homogeneous Navier-Stokes problem as a problem of
the form (\ref{e5_1}). $\lambda$ is variable. We introduce a new
variable $q$ $=$ $\lambda p$ and we write the problem from
\cite{CLBichir_bib_Gir_Rav1986} in the following form
\begin{eqnarray}
   & \ & p - \lambda^{-1} q = 0 \, ,
         \label{e1_9_ecStokes_var_q} \\
   & \ & (\textbf{u},q) + \mathcal{T}_{S}G_{NS}(\lambda,\textbf{u},p) = 0  \, ,
         \nonumber
\end{eqnarray}
or
\begin{eqnarray}
   & \ & p + \widetilde{\mathcal{T}}_{S}\widetilde{G}_{NS}(\lambda,q) = 0 \, ,
         \label{e1_9_ecStokes_var_q_op} \\
   & \ & (\textbf{u},q) + \mathcal{T}_{S}G_{NS}(\lambda,\textbf{u},p) = 0  \, ,
         \nonumber
\end{eqnarray}
where $\mathcal{T}_{\ast}(p,\textbf{f})=\left[\begin{array}{l}
      \widetilde{\mathcal{T}}_{S}p \\
      \mathcal{T}_{S}\textbf{f}
      \end{array}\right]
      = \left[\begin{array}{l}
      I_{1}p \\
      \mathcal{T}_{S}\textbf{f}
      \end{array}\right]$,

$\mathcal{G}_{\ast}(\lambda,p,\textbf{u},q)=\left[\begin{array}{l}
      - \lambda^{-1} q \\
      G_{NS}(\lambda,\textbf{u},p)
      \end{array}\right]
      = \left[\begin{array}{l}
      \widetilde{G}_{NS}(\lambda,q) \\
      G_{NS}(\lambda,\textbf{u},p)
      \end{array}\right]$.

$I_{1}$, $I_{2}$ are the identity operators on
$L_{0}^{2}(\Omega)$, $H_{0}^{1}(\Omega)^{N}$ respectively.

$\mathcal{T}_{\ast}^{-1}(\bar{a}_{i})$ are linearly independent if
and only if $\bar{a}_{i}$ are linearly independent.

We denote $\mathcal{W}$ $=$ $L_{0}^{2}(\Omega)$ $\times$
$H_{0}^{1}(\Omega)^{N}$ $\times$ $L_{0}^{2}(\Omega)$. Let
$\mathcal{F}_{\ast}:(0,\infty)$ $\times$ $\mathcal{W}$
$\rightarrow$ $\mathcal{W}$,
$\mathcal{F}_{\ast}(\lambda,p,\textbf{u},q) = (p,\textbf{u},q) +
\mathcal{T}_{\ast}\mathcal{G}_{\ast}(\lambda,p,\textbf{u},q)$.
$(\lambda,\textbf{u},q)$ $=$ $(\lambda,\textbf{u},\lambda p)$ is a
solution of $F_{NS}(\lambda,\textbf{u},q)=0$ is equivalent to say
that $(\lambda,p,\textbf{u},q)$ is a solution of the following
problem
\begin{equation}
\label{e5_1_NS}
   \mathcal{F}_{\ast}(\lambda,p,\textbf{u},q)=0 \, .
\end{equation}

$\mathcal{F}_{\ast}$ is approximated by
$\mathcal{F}_{\ast,h}:(0,\infty)$ $\times$ $\mathcal{W}_{h}$
$\rightarrow$ $\mathcal{W}_{h}$,
$\mathcal{F}_{\ast,h}(\lambda,p,\textbf{u},q) = (p,\textbf{u},q) +
\mathcal{T}_{\ast,h}\mathcal{G}_{\ast}(\lambda,p,\textbf{u},q)$,
$\mathcal{W}_{h}$ $=$ $M_{h}$ $\times$ $\Gamma_{0h}$ $\times$
$M_{h}$.

\subsection{The verification of the hypotheses of Theorem
\ref{teorema_lema5_8p}}
\label{sectiunea04_cazul_ecNS_2}

Let us replace $F$, $\mathbb{R}^{m} \times W$, $Z$, $(\lambda,u)$,
$(\lambda_{0},u_{0})$, in problem (\ref{e5_1}) and in hypothesis
(\ref{ipotezaHypF}), by $\mathcal{F}_{\ast}$, $\mathbb{R}$
$\times$ $\mathcal{W}$, $\mathcal{W}$, $(\lambda,p,\textbf{u},q)$,
$(\lambda_{0},p_{0},\textbf{u}_{0},q_{0})$ respectively. We say
that we study problem (\ref{e5_1_NS}) under hypothesis
(\ref{ipotezaHypF}) and we want to prove that the results of
Theorem \ref{teorema_lema5_8p} hold for the approximation
(\ref{e5_1_NS_h}) of (\ref{e5_1_NS}). For this, it suffices to
verify hypotheses (\ref{e5_8_proprietati}),
(\ref{e5_45p_conditia_izomorfism}) - (\ref{e5_45p_conditia_V}).

Let $\pi_{h}^{\mathcal{W}}$ $=$ $\pi_{h}^{Z}$ $=$ $\pi_{h}^{W}$
$=$ $(s_{h},r_{h},s_{h})$ and $I$ $=$ $(I_{1},I_{2},I_{1})$. We
take
\begin{equation}
\label{HypH1_lambda_sa}
   \widetilde{\lambda}_{0h} = \lambda_{0} \, , \
   \widetilde{p}_{0h} = s_{h} p_{0} \, ,
   \widetilde{\textbf{u}}_{0h} = r_{h} \textbf{u}_{0} \, ,
   \widetilde{q}_{0h} = s_{h} q_{0} \, .
\end{equation}

We have $\widetilde{\lambda}_{0h}$ $=$ $\lambda_{0}$. We denote
$\sigma_{0}$ $=$ $(\lambda_{0},p_{0},\textbf{u}_{0})$,
$\widetilde{\sigma}_{0h}$ $=$
$(\widetilde{\lambda}_{0h},\widetilde{p}_{0h},\widetilde{\textbf{u}}_{0h})$,
$\sigma_{0}^{NS}$ $=$ $(\lambda_{0},\textbf{u}_{0},p_{0})$,
$\widetilde{\sigma}_{0h}^{NS}$ $=$
$(\widetilde{\lambda}_{0h},\widetilde{\textbf{u}}_{0h},\widetilde{p}_{0h})$.

Let us first verify (\ref{e5_45p_conditia_V}). We have
\begin{eqnarray}
   & \ & \textrm{there exists} \ \eta_{4} > 0 \
      \textrm{such that, for every} \
      v_{h} \in W_{h}, \
      \textrm{we have}
         \label{e5_45p_conditia_V_verif}
\end{eqnarray}
$$ \| \pi_{h}^{Z}D_{u}F(\lambda_{0},u_{0})v_{h}
      - D_{u}F_{h}(\widetilde{\lambda}_{0h},\widetilde{u}_{0h})v_{h}
      \|_{Z_{h}} $$
$$ = \|
   \pi_{h}^{\mathcal{W}}D_{(p,\textbf{u},q)}\mathcal{F}_{\ast}(\sigma_{0},q_{0})(\overline{p},\overline{\textbf{u}},\overline{q})
      - D_{(p,\textbf{u},q)}\mathcal{F}_{\ast,h}(\widetilde{\sigma}_{0h},\widetilde{q}_{0h})(\overline{p},\overline{\textbf{u}},\overline{q})
      \|_{\mathcal{W}_{h}} $$
$$ \leq \| (\pi_{h}^{\mathcal{W}}-I)(\overline{p},\overline{\textbf{u}},\overline{q})
      \|_{\mathcal{W}_{h}} $$
$$ + \| \pi_{h}^{\mathcal{W}}
      \left[\begin{array}{l}
      \widetilde{\mathcal{T}}_{S} D_{q}\widetilde{G}_{NS}(\lambda_{0},q_{0})\overline{q} \\
      \mathcal{T}_{S} D_{(\textbf{u},p)}G_{NS}(\sigma_{0}^{NS})(\overline{\textbf{u}},\overline{p})
      \end{array}\right]
      -
      \left[\begin{array}{l}
      \tilde{\mathcal{T}_{S,h}} D_{q}\widetilde{G}_{NS}(\widetilde{\lambda}_{0h},\widetilde{q}_{0h})\overline{q} \\
      \mathcal{T}_{S,h} D_{(\textbf{u},p)}G_{NS}(\widetilde{\sigma}_{0h}^{NS})(\overline{\textbf{u}},\overline{p})
      \end{array}\right]
      \|_{\mathcal{W}_{h}} $$
$$ = \| (\pi_{h}^{\mathcal{W}}-I)(\overline{p},\overline{\textbf{u}},\overline{q})
      \|_{\mathcal{W}_{h}} $$
$$ + \| (r_{h},s_{h}) \mathcal{T}_{S} D_{(\textbf{u},p)}G_{NS}(\sigma_{0}^{NS})(\overline{\textbf{u}},\overline{p})
         - \mathcal{T}_{S,h} D_{(\textbf{u},p)}G_{NS}(\widetilde{\sigma}_{0h}^{NS})(\overline{\textbf{u}},\overline{p})
      \|_{\mathcal{X}_{h}} $$
$$ \leq \| (\pi_{h}^{\mathcal{W}}-I)(\overline{p},\overline{\textbf{u}},\overline{q})
      \|_{\mathcal{W}_{h}} $$
$$ + \| ((r_{h},s_{h})-(I_{2},I_{1})) \mathcal{T}_{S} D_{(\textbf{u},p)}G_{NS}(\sigma_{0}^{NS})(\overline{\textbf{u}},\overline{p})
      \|_{\mathcal{X}_{h}} $$
$$ + \| (\mathcal{T}_{S} - \mathcal{T}_{S,h}) D_{(\textbf{u},p)}G_{NS}(\sigma_{0}^{NS})(\overline{\textbf{u}},\overline{p})
      \|_{\mathcal{X}_{h}} $$
$$ + \| \mathcal{T}_{S,h}(D_{(\textbf{u},p)}G_{NS}(\sigma_{0}^{NS})(\overline{\textbf{u}},\overline{p})
         - D_{(\textbf{u},p)}G_{NS}(\widetilde{\sigma}_{0h}^{NS})(\overline{\textbf{u}},\overline{p}))
      \|_{\mathcal{X}_{h}} $$
$$ \leq
      \eta_{4} \| (\overline{p},\overline{\textbf{u}},\overline{q}) \|_{\mathcal{W}_{h}}
      = \eta_{4} \| v_{h} \|_{W_{h}} \, , $$
where $\eta_{4}$ is obtained using (\ref{HypH1_lambda_sa}),
(\ref{HypH1}), (\ref{HypH2}),
(\ref{e1_9_ecStokes_h_Gir_Rav_lim_f_L_3_2}) si
(\ref{e1_9_ecStokes_h_Gir_Rav_lim}).

Let us verify (\ref{e5_45p_conditia_I}). We have
\begin{eqnarray}
   & \ &  \textrm{there exists} \ \eta_{2} > 0 \
      \textrm{such that, for every} \
      (\lambda_{h},u_{h}) \in \mathbb{R}^{m} \times W_{h}, \
      \textrm{we have}
         \label{e5_45p_conditia_I_verif}
\end{eqnarray}
$$ \| \pi_{h}^{Z}DF(\lambda_{0},u_{0})(\lambda_{h},u_{h})
      - DF_{h}(\widetilde{\lambda}_{0h},\widetilde{u}_{0h})(\lambda_{h},u_{h})
      \|_{Z_{h}} $$
$$ = \|
   \pi_{h}^{\mathcal{W}}D\mathcal{F}_{\ast}(\sigma_{0},q_{0})(\overline{\lambda},\overline{p},\overline{\textbf{u}},\overline{q})
      - D\mathcal{F}_{\ast,h}(\widetilde{\sigma}_{0h},\widetilde{q}_{0h})(\overline{\lambda},\overline{p},\overline{\textbf{u}},\overline{q})
      \|_{\mathcal{W}_{h}} $$
$$ \leq \| \pi_{h}^{\mathcal{W}}
      \left[\begin{array}{l}
      \widetilde{\mathcal{T}}_{S} D_{\lambda}\widetilde{G}_{NS}(\lambda_{0},q_{0})\overline{\lambda} \\
      \mathcal{T}_{S} D_{\lambda}G_{NS}(\sigma_{0}^{NS})\overline{\lambda}
      \end{array}\right]
      -
      \left[\begin{array}{l}
      \tilde{\mathcal{T}_{S,h}} D_{\lambda}\widetilde{G}_{NS}(\widetilde{\lambda}_{0h},\widetilde{q}_{0h})\overline{\lambda} \\
      \mathcal{T}_{S,h} D_{\lambda}G_{NS}(\widetilde{\sigma}_{0h}^{NS})\overline{\lambda}
      \end{array}\right]
      \|_{\mathcal{W}_{h}} $$
$$ + \| \pi_{h}^{\mathcal{W}}D_{(p,\textbf{u},q)}\mathcal{F}_{\ast}(\sigma_{0},q_{0})(\overline{p},\overline{\textbf{u}},\overline{q})
      - D_{(p,\textbf{u},q)}\mathcal{F}_{\ast,h}(\widetilde{\sigma}_{0h},\widetilde{q}_{0h})(\overline{p},\overline{\textbf{u}},\overline{q})
      \|_{\mathcal{W}_{h}} $$
$$ = \| s_{h} \widetilde{\mathcal{T}}_{S} D_{\lambda}\widetilde{G}_{NS}(\lambda_{0},q_{0})\overline{\lambda}
        - \tilde{\mathcal{T}_{S,h}} D_{\lambda}\widetilde{G}_{NS}(\widetilde{\lambda}_{0h},\widetilde{q}_{0h})\overline{\lambda}
      \|_{\mathcal{M}_{h}} $$
$$ + \| (r_{h},s_{h}) \mathcal{T}_{S} G_{NS}(\overline{\lambda},\textbf{u}_{0},p_{0})
         - \mathcal{T}_{S,h} G_{NS}(\overline{\lambda},\widetilde{\textbf{u}}_{0h},\widetilde{p}_{0h})
      \|_{\mathcal{X}_{h}} $$
$$ + \| \pi_{h}^{\mathcal{W}}D_{(p,\textbf{u},q)}\mathcal{F}_{\ast}(\sigma_{0},q_{0})(\overline{p},\overline{\textbf{u}},\overline{q})
      - D_{(p,\textbf{u},q)}\mathcal{F}_{\ast,h}(\widetilde{\sigma}_{0h},\widetilde{q}_{0h})(\overline{p},\overline{\textbf{u}},\overline{q})
      \|_{\mathcal{W}_{h}} $$
$$ \leq \| ((r_{h},s_{h})-(I_{2},I_{1})) \mathcal{T}_{S} G_{NS}(\overline{\lambda},\textbf{u}_{0},p_{0}))
      \|_{\mathcal{X}_{h}} $$
$$ + \| (\mathcal{T}_{S} - \mathcal{T}_{S,h}) G_{NS}(\overline{\lambda},\textbf{u}_{0},p_{0})
      \|_{\mathcal{X}_{h}}
      + \| \mathcal{T}_{S,h}(G_{NS}(\overline{\lambda},\textbf{u}_{0},p_{0})
         - G_{NS}(\overline{\lambda},\widetilde{\textbf{u}}_{0h},\widetilde{p}_{0h}))
      \|_{\mathcal{X}_{h}} $$
$$ + \| \pi_{h}^{\mathcal{W}}D_{(p,\textbf{u},q)}\mathcal{F}_{\ast}(\sigma_{0},q_{0})(\overline{p},\overline{\textbf{u}},\overline{q})
      - D_{(p,\textbf{u},q)}\mathcal{F}_{\ast,h}(\widetilde{\sigma}_{0h},\widetilde{q}_{0h})(\overline{p},\overline{\textbf{u}},\overline{q})
      \|_{\mathcal{W}_{h}} $$
$$ \leq
      \eta_{2} \| (\overline{\lambda},\overline{p},\overline{\textbf{u}},\overline{q}) \|_{\mathbb{R} \times \mathcal{W}_{h}}
      = \eta_{2} \| (\lambda_{h},u_{h}) \|_{\mathbb{R}^{m} \times W_{h}} \, , $$
where $\eta_{2}$ is obtained using
(\ref{e5_45p_conditia_V_verif}), (\ref{HypH1_lambda_sa}),
(\ref{HypH1}), (\ref{HypH2}) and
(\ref{e1_9_ecStokes_h_Gir_Rav_lim_f}) for $\textbf{f}$ $=$
$G_{NS}(1,\textbf{u}_{0},p_{0})$.

Equation (\ref{e5_1_NS}) has the form of equation (\ref{e5_1}).
Assume that (\ref{e5_1_NS}) satisfies the hypotheses of Theorem
\ref{teorema_principala_spatii_infinit_dimensionale_widetilde_s_3_0_exact_h}.
Then the corresponding equation
(\ref{e5_1_sol_widetilde_x_0h_Inv_Fc_Th_ec_DATA_introd_exact_th_h}),
with $\varrho_{h}$ $=$
$(\widetilde{\varrho}_{h},\overline{\varrho}_{h},\widehat{\varrho}_{h})$,
is
\begin{equation}
\label{e5_1_sol_widetilde_x_0h_Inv_Fc_Th_ec_DATA_introd_exact_th_h_NS}
      \mathcal{F}_{\ast,h}(\lambda_{h},p_{h},\textbf{u}_{h},q_{h})-\varrho_{h}
      = 0 \, ,
\end{equation}
or
\begin{equation}
\label{e5_1_sol_widetilde_x_0h_Inv_Fc_Th_ec_DATA_introd_exact_th_h_NS_comp}
      (p_{h},\textbf{u}_{h},q_{h}) +
         \mathcal{T}_{\ast,h}\mathcal{G}_{\ast}(\lambda_{h},p_{h},\textbf{u}_{h},q_{h})
         -(\widetilde{\varrho}_{h},\overline{\varrho}_{h},\widehat{\varrho}_{h})
      = 0 \, ,
\end{equation}
or
\begin{eqnarray}
   && p_{h} - \lambda_{h}^{-1} q_{h} = \widetilde{\varrho}_{h} \, ,
         \label{e1_9_ecStokes_var_q_h_bif} \\
   && (grad \, (\textbf{u}_{h}-\overline{\varrho}_{h}),grad \, \textbf{w}_{h})-((q_{h}-\widehat{\varrho}_{h}),div \, \textbf{w}_{h})
         \label{e1_9_ecStokes_h_1_bif} \\
   && \quad = \lambda_{h}\langle \textbf{f}-\sum_{j=1}^{N} u_{h,j} \frac{\partial \textbf{u}_{h}}{\partial x_{j}},\textbf{w}_{h} \rangle, \
              \forall \, \textbf{w}_{h}  \, \in \, \Gamma_{0h} \, ,
         \nonumber \\
   && (div \, (\textbf{u}_{h}-\overline{\varrho}_{h}),\mu_{h})=0, \
              \forall \, \mu_{h}  \, \in \, M_{h} \, .
         \label{e1_9_ecStokes_h_2_bif}
\end{eqnarray}

With the same settings, the framework of Section 10.2.3 (and also
of Section 10.2.2, rewritten for spectral Galerkin
approximations), from \cite{CLBichir_bib_Quarteroni_Valli2008},
allows the approximate Stokes operator $\mathcal{T}_{S,h}$ to be
constructed using spectral methods.

\section{A complement to Theorem
\ref{teorema_principala_spatii_infinit_dimensionale_widetilde_s_3_0_exact_inf}}
\label{sectiunea_01_O_formulare_pe_spatii_infinit_dimensionale_COMPLEMENTE_NOU}

In this section, we investigate if we can fix only $\hat{y}_{0}'$
in
(\ref{e5_57_forma2_sistem_REG_2_GGGvarianta_infinit_dimensionale_DEM_pc_fix_cont_th_princ}),
in Theorem
\ref{teorema_principala_spatii_infinit_dimensionale_widetilde_s_3_0_exact_inf}.

Let $F:\mathbb{R}^{m} \times W \rightarrow Z$ be a nonlinear
function of class $C^{p}$. Let $(\mu',w')$ $\in$ $\mathbb{R}^{m}
\times W$. Let us define
\begin{equation*}
\label{e5_1_F_tilde}
   \widetilde{F}:\mathbb{R}^{m} \times W \rightarrow Z, \
      \widetilde{F}(\lambda,u)=F(\lambda,u)-DF(\lambda,u)(\mu',w') \, ,
\end{equation*}
and $\widetilde{G}:X \rightarrow Z$, $\widetilde{H}:\mathbb{R}^{m}
\times W \times \Delta \rightarrow Z$, $\widetilde{\Psi}:X
\rightarrow Y$, where $\widetilde{G}$, $\widetilde{H}$,
$\widetilde{\Psi}$ are obtained by replacing $F$ by
$\widetilde{F}$ in the definitions of $G$, $H$, $\Psi$
respectively.

Let us define $\widetilde{S}:\Gamma \rightarrow \Sigma$,
$\widetilde{S}(s)$ $=$ $S(s)$, and $\widetilde{\Phi}: \Gamma
\times \Gamma \rightarrow \Sigma$,
\begin{equation*}
\label{e5_57_Phi_3_modif}
   \widetilde{\Phi}(s,\phi')=\Phi(s,\phi')
      +
      \left[\begin{array}{l}
         \left[\begin{array}{l}
              0 \\
              0
         \end{array}\right] \\
         \left[\begin{array}{l}
              0 \\
              D^{2}F(\lambda,u)((\mu',w'),((\mu_{i},w_{i})-(\mu_{i}',w_{i}')))
         \end{array}\right] \\
         \left[\begin{array}{l}
              0 \\
              D_{u}(DF(\lambda,u)(\mu',w'))(v_{k}-v_{k}')
         \end{array}\right]
      \end{array}\right] \, ,
\end{equation*}
for all $i=1,\ldots,q+m$, $k=1,\ldots,n$.

In the definitions from Subsections
\ref{sectiunea_01_O_formulare_pe_spatii_infinit_dimensionale_1}
and
\ref{sectiunea_01_O_formulare_pe_spatii_infinit_dimensionale_2},
let us replace $F$, $G$, $H$, $\Psi$, $S$, $\Phi$ by
$\widetilde{F}$, $\widetilde{G}$, $\widetilde{H}$,
$\widetilde{\Psi}$, $\widetilde{S}$, $\widetilde{\Phi}$
respectively.

\begin{cor}
\label{corolarul_doi_2_teorema_principala_spatii_infinit_dimensionale_widetilde_s_3_0_exact_fi3_DIF}
Assume the hypotheses of Theorem
\ref{teorema_principala_spatii_infinit_dimensionale_widetilde_s_3_0_exact_inf}
where we replace $F$, $G$, $H$, $S$, $\Phi$ by $\widetilde{F}$,
$\widetilde{G}$, $\widetilde{H}$, $\widetilde{\Psi}$,
$\widetilde{S}$, $\widetilde{\Phi}$ respectively. Then
(\ref{e5_57_forma2_sistem_REG_2_GGGvarianta_infinit_dimensionale_DEM_pc_fix_cont_th_princ})
becomes
\begin{equation}
\label{corolarul_doi_2_e5_57_forma2_sistem_REG_2_GGGvarianta_infinit_dimensionale_DEM_pc_fix_cont_fi3_DIF}
   \widetilde{S}(s_{0})
   - \widetilde{\Phi}(s_{0},\phi_{0}')
   \ni 0 \, ,
\end{equation}

Let us fix $\hat{y}_{0}'$ (whose existence is demonstrated) in
(\ref{corolarul_doi_2_e5_57_forma2_sistem_REG_2_GGGvarianta_infinit_dimensionale_DEM_pc_fix_cont_fi3_DIF}).
Let us take $\theta_{0}$ $=$
$\widetilde{\theta}_{0}+B(\hat{y}_{0}')$. Let us define
\begin{equation} \label{e5_1_F_tilde_0}
   \widetilde{F}_{0}:\mathbb{R}^{m} \times W \rightarrow Z, \
      \widetilde{F}_{0}(\lambda,u)=F(\lambda,u)-DF(\lambda,u)(\hat{\mu}_{0}',\hat{w}_{0}') \, ,
\end{equation}
and $\widetilde{G}_{0}:X \rightarrow Z$,
$\widetilde{H}_{0}:\mathbb{R}^{m} \times W \times \Delta
\rightarrow Z$, $\widetilde{\Psi}_{0}:X \rightarrow Y$, where
$\widetilde{G}_{0}$, $\widetilde{H}_{0}$, $\widetilde{\Psi}_{0}$
are obtained by replacing $F$ by $\widetilde{F}_{0}$ in the
definitions of $G$, $H$, $\Psi$ respectively. Let
$\widetilde{S}_{0}:\Gamma \rightarrow \Sigma$ be $S$ where we
replace $F$ by $\widetilde{F}_{0}$ in (\ref{e5_57}) and we use
$\theta_{0}$ defined above.

Then, $s_{0}$ is the solution of the equation
\begin{equation}
\label{e5_57_forma2_sistem_REG_2_GGGvarianta_infinit_dimensionale_DEM_pc_fix_cont_ec_VAR}
   \widetilde{S}_{0}(s)= 0 \ .
\end{equation}
Equation
(\ref{e5_57_forma2_sistem_REG_2_GGGvarianta_infinit_dimensionale_DEM_pc_fix_cont_ec_VAR})
is of the form of equation (\ref{e5_56}).

Then, the component $(\lambda_{0},u_{0})$ of $s_{0}$ is a solution
of the equation
\begin{equation}
\label{e5_1_sol_widetilde_x_0h_Inv_Fc_Th_ec_DATA_introd_exact_prim}
      F(\lambda,u)-DF(\lambda,u)(\hat{\mu}_{0}',\hat{w}_{0}') = 0 \, ,
\end{equation}
$(\lambda_{0},u_{0})$ $\in$
$\mathbb{B}_{a^{\ast}}(\widetilde{\lambda}_{0},\widetilde{u}_{0})$.

Assume that $D\widetilde{\Psi}_{0}(\widetilde{x}_{0})$ is an
isomorphism of $X$ onto $Y$ and
$D\widetilde{S}_{0}(\widetilde{s}_{0})$ is an isomorphism of
$\Gamma$ onto $\Sigma$. Under some other additional conditions
like in Theorem
\ref{teorema_principala_spatii_infinit_dimensionale_widetilde_s_3_0_exact_inf},
we have:

Then, $s_{0}$ is the unique solution of the equation
(\ref{e5_57_forma2_sistem_REG_2_GGGvarianta_infinit_dimensionale_DEM_pc_fix_cont_ec_VAR})
in some $\mathbb{B}_{a}(\widetilde{s}_{0})$ with $a$ $\geq$
$a^{\ast}$. The system
(\ref{e5_57_forma2_sistem_REG_2_GGGvarianta_infinit_dimensionale_DEM_pc_fix_cont_ec_VAR})
and its solution $s_{0}$ verify the assertions (a) and (b) of the
statement (ii) of Theorem \ref{teorema_principala_parteaMAIN}.

Then, the component $(\lambda_{0},u_{0})$ of $s_{0}$ is the unique
solution of the equation
(\ref{e5_1_sol_widetilde_x_0h_Inv_Fc_Th_ec_DATA_introd_exact_prim})
that satisfies hypothesis (\ref{ipotezaHypF}) and the rest of the
hypotheses of the statement (i) of Theorem
\ref{teorema_principala_parteaMAIN} in some
$\mathbb{B}_{a}(\widetilde{\lambda}_{0},\widetilde{u}_{0})$ with
$a$ $\geq$ $a^{\ast}$. The solution $(\lambda_{0},u_{0})$ is a
bifurcation point of problem
(\ref{e5_1_sol_widetilde_x_0h_Inv_Fc_Th_ec_DATA_introd_exact_prim}).
\end{cor}

\begin{proof}
The proof of Theorem
\ref{teorema_principala_spatii_infinit_dimensionale_widetilde_s_3_0_exact_inf}
remains valid. We only mention that
(\ref{e5_57_forma2_sistem_REG_2_GGGvarianta_infinit_dimensionale_DEM_pc_fix})
has the formulation
\begin{equation*}
\label{e5_57_forma2_sistem_REG_2_GGGvarianta_infinit_dimensionale_DEM_pc_fix_widetilde}
   \widetilde{S}(\bar{s})
   - \widetilde{\Phi}(\bar{x},\bar{\phi}')
   - \widetilde{\Phi}(\widetilde{x}_{0},\xi(\bar{f},\bar{g}_{i},\bar{e}_{k}))
   + \widetilde{\Phi}(\widetilde{x}_{0},\xi(\bar{g}',\bar{g}_{i}',\bar{e}_{k}'))
   \ni 0 \, ,
\end{equation*}
$$\Leftrightarrow$$
\begin{equation*}
\label{e5_57_forma2_sistem_REG_2_GGGvarianta_infinit_dimensionale_DEM_pc_fix_cont_bar_widetilde}
   \left[\begin{array}{l}
      B(0,\bar{\lambda},\bar{u})-\widetilde{\theta}_{0}-B(0,\bar{\mu}',\bar{w}') \\
      F(\bar{\lambda},\bar{u})-DF(\bar{\lambda},\bar{u})(\bar{\mu}',\bar{w}') \\
      B((0,\bar{\mu}_{i},\bar{w}_{i})-(0,\bar{\mu}_{i}',\bar{w}_{i}'))-\delta_{i}^{q+m} \\
      D(F(\bar{\lambda},\bar{u})
         -DF(\bar{\lambda},\bar{u})(\bar{\mu}',\bar{w}'))((\bar{\mu}_{i},\bar{w}_{i})-(\bar{\mu}_{i}',\bar{w}_{i}'))) \\
      \bar{\mathcal{B}}((0,\bar{v}_{k})-(0,\bar{v}_{k}'))-\delta_{k}^{n} \\
      D_{u}(F(\lambda,u)
         -DF(\bar{\lambda},\bar{u})(\bar{\mu}',\bar{w}'))(\bar{v}_{k}-\bar{v}_{k}')
   \end{array}\right]
   \ni 0 \, ,
\end{equation*}

\qquad
\end{proof}

\begin{lem}
\label{Lema_izom_tilde_dif} Assume that $DS(\widetilde{s}_{0})$ is
an isomorphism of $\Gamma$ onto $\Sigma$. If
\begin{equation}
\label{dem_cond_izom_tilde_dif}
   \gamma \|DS(\widetilde{s}_{0})
   -D\widetilde{S}_{0}(\widetilde{s}_{0})\|_{L(\Gamma,\Sigma)} < 1 \, ,
\end{equation}
then $D\widetilde{S}_{0}(\widetilde{s}_{0})$ is an isomorphism of
$\Gamma$ onto $\Sigma$.
\end{lem}

\begin{proof}
This results from Lemma 3.1 \cite{CLBichir_bib_Cr_Ra1990}, Theorem
6A.1 \cite{CLBichir_bib_Dontchev_Rockafellar2009} and Lemma IV.3.3
\cite{CLBichir_bib_Gir_Rav1986}.

\qquad
\end{proof}

\begin{rem}
\label{observatia_diferentiala_G} Instead of a constant $\varrho$
in the equation
(\ref{e5_1_sol_widetilde_x_0h_Inv_Fc_Th_ec_DATA_introd_exact}), we
have obtained $\varrho$ in the form of a function
$\varrho(\lambda,u)$ $=$
$DF(\lambda,u)(\hat{\mu}_{0}',\hat{w}_{0}')$ in the equation
(\ref{e5_1_sol_widetilde_x_0h_Inv_Fc_Th_ec_DATA_introd_exact_prim}).

\end{rem}

\section{Future Work}
\label{sectiunea_CONCLUZII_0}

We have formulated some sufficient conditions for the existence of
an approximate equation
(\ref{e5_1_sol_widetilde_x_0h_Inv_Fc_Th_ec_DATA_introd_exact_th_h})
that has a bifurcation point of the same type as the bifurcation
point of a given exact equation (\ref{e5_1}).

In a further research that will continue the present one, we have
the following purposes:

(i) to prove that, given a function $F$ and an approximation
$F_{h}$ for this, under some conditions, if there exists
$\varrho_{h}$ (which is zero or nonzero) so that the equation
\begin{equation}
\label{e5_1_sol_widetilde_x_0h_Inv_Fc_Th_ec_DATA_introd_exact_th_h_CONCLUZII_0}
      F_{h}(\lambda_{h},u_{h})-\varrho_{h}
      = 0 \,
\end{equation}
has a bifurcation point, then there exists $\varrho$ such that the
equation
\begin{equation}
\label{e5_1_sol_widetilde_x_0h_Inv_Fc_Th_ec_DATA_introd_exact_CONCLUZII_0}
      F(\lambda,u) - \varrho = 0 \,
\end{equation}
has a bifurcation point of the same type as the bifurcation point
of
(\ref{e5_1_sol_widetilde_x_0h_Inv_Fc_Th_ec_DATA_introd_exact_th_h_CONCLUZII_0}).
This idea is inspired by a result from
\cite{CLBichir_bib_Ka_Ak1986} where Kantorovich and Akilov prove
that, given the linear operators that define an exact equation and
an approximate equation, under certain hypotheses, if the
approximate operator is an isomorphism, then the exact operator is
an isomorphism.

(ii) to formulate some algorithms so that, by studying the
approximate equation
\begin{equation}
\label{e5_9_introd_CONCLUZII_0}
   F_{h}(\lambda_{h},u_{h})=0 \, ,
\end{equation}
to decide if there exists $\varrho_{h}$ such that the equation
(\ref{e5_1_sol_widetilde_x_0h_Inv_Fc_Th_ec_DATA_introd_exact_th_h_CONCLUZII_0})
has a bifurcation point, to determine $\varrho_{h}$ (only if this
is necessary) and to determine the type of the bifurcation point
of
(\ref{e5_1_sol_widetilde_x_0h_Inv_Fc_Th_ec_DATA_introd_exact_th_h_CONCLUZII_0})
and hence of
(\ref{e5_1_sol_widetilde_x_0h_Inv_Fc_Th_ec_DATA_introd_exact_CONCLUZII_0}).
In the study for (ii), we will extend the methods introduced in
\cite{CLBichir_bib_PhDThesis2002}.

In this way, we can reduce the study of the qualitative aspects of
a bifurcation problem on infinite-dimensional Banach spaces to the
study of an approximate problem. Finally, the study can be
performed on a computer.

\appendix

\section{A formulation of $D\mathcal{G}(\widetilde{s}_{0},\widetilde{\phi}_{0}')(\overline{s},\overline{\phi}') - D\mathcal{G}(s,\phi')(\overline{s},\overline{\phi}')$}
\label{sectiunea_evaluarea lui_mu_L_h}

\begin{equation}
\label{appendix_e_A2_17_TEXT_mu}
   \Upsilon(\widetilde{s}_{0},\widetilde{\phi}_{0}',s,\phi',\overline{s},\overline{\phi}')
   = D\mathcal{G}(\widetilde{s}_{0},\widetilde{\phi}_{0}')(\overline{s},\overline{\phi}')
   - D\mathcal{G}(s,\phi')(\overline{s},\overline{\phi}')
\end{equation}
\begin{equation*}
\label{abc_e5_57_forma2_sistem_REG_2_GGGvarianta_infinit_dimensionale_DEM_dif_tilde}
   = \frac{1}{2}DS(\widetilde{s}_{0})\overline{s}
   - \frac{1}{2}D\Phi(\widetilde{x}_{0},\widetilde{\phi}_{0}')(\overline{x},\overline{\phi}')
   - (1-\alpha) \Phi(\widetilde{x}_{0},\xi(\overline{f},\overline{g}_{i},\overline{e}_{k}))
   + (1-\alpha) \Phi(\widetilde{x}_{0},\xi(\overline{g}',\overline{g}_{i}',\overline{e}_{k}'))
\end{equation*}
\begin{equation*}
\label{abc_e5_57_forma2_sistem_REG_2_GGGvarianta_infinit_dimensionale_DEM_dif}
   -\frac{1}{2}DS(s)\overline{s}
   + \frac{1}{2}D\Phi(x,\phi')(\overline{x},\overline{\phi}')
   + (1-\alpha) \Phi(\widetilde{x}_{0},\xi(\overline{f},\overline{g}_{i},\overline{e}_{k}))
   - (1-\alpha) \Phi(\widetilde{x}_{0},\xi(\overline{g}',\overline{g}_{i}',\overline{e}_{k}'))
\end{equation*}

\begin{equation*}
\label{abc_e5_57_forma2_sistem_REG_2_GGGvarianta_infinit_dimensionale_DEM_dif_tilde}
   = \frac{1}{2}(DS(\widetilde{s}_{0})\overline{s}
   -DS(s)\overline{s})
\end{equation*}
\begin{equation*}
\label{abc_e5_57_forma2_sistem_REG_2_GGGvarianta_infinit_dimensionale_DEM_dif_tilde}
   - \frac{1}{2}(D\Phi(\widetilde{x}_{0},\widetilde{\phi}_{0}')(\overline{x},\overline{\phi}')
   -D\Phi(x,\phi')(\overline{x},\overline{\phi}'))
\end{equation*}
\begin{equation*}
\label{abc_e5_57_forma2_sistem_REG_2_GGGvarianta_infinit_dimensionale_DEM_dif_tilde}
   - (1-\alpha) \Phi(\widetilde{x}_{0},\xi(\overline{f},\overline{g}_{i},\overline{e}_{k}))
   + (1-\alpha) \Phi(\widetilde{x}_{0},\xi(\overline{g}',\overline{g}_{i}',\overline{e}_{k}'))
\end{equation*}
\begin{equation*}
\label{abc_e5_57_forma2_sistem_REG_2_GGGvarianta_infinit_dimensionale_DEM_dif}
   + (1-\alpha) \Phi(\widetilde{x}_{0},\xi(\overline{f},\overline{g}_{i},\overline{e}_{k}))
   - (1-\alpha) \Phi(\widetilde{x}_{0},\xi(\overline{g}',\overline{g}_{i}',\overline{e}_{k}')) \, .
\end{equation*}

\begin{equation*}
\label{abc_e5_27_S_3_0_tilde}
   \Upsilon(\widetilde{s}_{0},\widetilde{\phi}_{0}',s,\phi',\overline{s},\overline{\phi}')
      = \frac{1}{2}(\left[\begin{array}{l}
      B(\overline{x}) \\
      DG(\widetilde{x}_{0})\overline{x} \\
      B(\overline{y}_{i}) \\
      D^{2}F(\widetilde{\lambda}_{0},\widetilde{u}_{0})
         ((\widetilde{\mu}_{i,0},\widetilde{w}_{i,0}),(\overline{\lambda},\overline{u}))
            +DG(\widetilde{x}_{0})\overline{y}_{i} \\
      \bar{\mathcal{B}}(\overline{\mathrm{z}}_{k}) \\
      D_{(\lambda,u)}(D_{u}F(\widetilde{\lambda}_{0},\widetilde{u}_{0})\widetilde{v}_{k,0})
         (\overline{\lambda},\overline{u})
            +H(\widetilde{\lambda}_{0},\widetilde{u}_{0},\overline{\mathrm{z}}_{k})
      \end{array}\right]
\end{equation*}
\begin{equation*}
\label{abc_e5_27_S_3}
      -\left[\begin{array}{l}
      B(\overline{x}) \\
      DG(x)\overline{x} \\
      B(\overline{y}_{i}) \\
      D^{2}F(\lambda,u)
         ((\mu_{i},w_{i}),(\overline{\lambda},\overline{u}))
            +DG(x)\overline{y}_{i} \\
      \bar{\mathcal{B}}(\overline{\mathrm{z}}_{k}) \\
      D_{(\lambda,u)}(D_{u}F(\lambda,u)v_{k})
         (\overline{\lambda},\overline{u})
            +H(\lambda,u,\overline{\mathrm{z}}_{k})
      \end{array}\right])
\end{equation*}
\begin{equation*}
\label{abc_e5_57_Phi_3_dif_0_tilde}
      - \frac{1}{2}(\left[\begin{array}{l}
      B(\overline{y}') \\
      D^{2}F(\widetilde{\lambda}_{0},\widetilde{u}_{0})
         ((\widetilde{\mu}_{0}',\widetilde{w}_{0}'),(\overline{\lambda},\overline{u}))
            +DG(\widetilde{x}_{0})\overline{y}' \\
      B(\overline{y}_{i}') \\
      D^{2}F(\widetilde{\lambda}_{0},\widetilde{u}_{0})
         ((\widetilde{\mu}_{i,0}',\widetilde{w}_{i,0}'),(\overline{\lambda},\overline{u}))
            +DG(\widetilde{x}_{0})\overline{y}_{i}'
            +A \\
      \bar{\mathcal{B}}(\overline{\mathrm{z}}_{k}') \\
      D_{(\lambda,u)}(D_{u}F(\widetilde{\lambda}_{0},\widetilde{u}_{0})\widetilde{v}_{k,0}')
         (\overline{\lambda},\overline{u})
            +H(\widetilde{\lambda}_{0},\widetilde{u}_{0},\overline{\mathrm{z}}_{k}')
            +B
      \end{array}\right]
\end{equation*}
\begin{equation*}
\label{abc_e5_57_Phi_3_dif}
      -\left[\begin{array}{l}
      B(\overline{y}') \\
      D^{2}F(\lambda,u)
         ((\mu',w'),(\overline{\lambda},\overline{u}))
            +DG(x)\overline{y}' \\
      B(\overline{y}_{i}') \\
      D^{2}F(\lambda,u)
         ((\mu_{i}',w_{i}'),(\overline{\lambda},\overline{u}))
            +DG(x)\overline{y}_{i}'
            +C \\
      \bar{\mathcal{B}}(\overline{\mathrm{z}}_{k}') \\
      D_{(\lambda,u)}(D_{u}F(\lambda,u)v_{k}')
         (\overline{\lambda},\overline{u})
            +H(\lambda,u,\overline{\mathrm{z}}_{k}')
            +D
      \end{array}\right]) \, .
\end{equation*}

\section*{Acknowledgments}
\label{sectiunea_Acknowledgment}

The author owes gratitude to Mr. Costache Bichir and Mrs. Georgeta
Bichir for their moral and financial support without which this
paper would not have existed.

Many thanks to Mrs. Georgeta Stoenescu for her advice with the
English language in Sections \ref{sectiunea_0_introduction} and
\ref{sectiunea_01_O_formulare_pe_spatii_infinit_dimensionale_1}.


\end{document}